\documentclass[psamsfonts]{amsart}
\usepackage{amssymb,amsfonts,lacromay,color,comment}
\usepackage[all,arc]{xy}
\usepackage{enumerate}
\usepackage{mathrsfs}
\usepackage{stmaryrd}

\title[Six model structures for DG-modules over DGAs]{Six model structures for DG-modules over DGAs: \\ Model category theory in homological action}

\author{Tobias Barthel}
\address{Department of Mathematics, Harvard University,
Cambridge, MA \ 02138}
\email{tbarthel@math.harvard.edu}
\author{J.P.~May}
\address{Department of Mathematics, University of Chicago, Chicago, IL \ 60637}
\email{may@math.uchicago.edu}
\author{Emily Riehl}
\address{Department of Mathematics, Harvard University,
Cambridge, MA \ 02138}
\email{eriehl@math.harvard.edu}

\newtheorem{thm}{Theorem}[section]
\newtheorem{cor}[thm]{Corollary}
\newtheorem{prop}[thm]{Proposition}
\newtheorem{lem}[thm]{Lemma}

\theoremstyle{definition}
\newtheorem{defn}[thm]{Definition}

\newtheorem{ex}[thm]{Example}

\newtheorem*{uconv}{Convention}
\newtheorem*{unotn}{Notation}

\theoremstyle{remark}
\newtheorem{rmk}[thm]{Remark}

\makeatletter
\let\c@equation\c@thm
\makeatother
\numberwithin{equation}{section}

\newcommand{\Sq}{\underline{\mathrm{Sq}}}

\newcommand{\cod}{\mathrm{cod}}

\newcommand{\Set}{\mathbf{Set}}

\newcommand{\co}{\colon}
\newcommand{\Alg}{\mathbf{Alg}}
\newcommand{\Coalg}{\mathbf{Coalg}}
\newcommand{\bs}{\boxslash}
\newcommand{\cE}{\mathcal{E}}
\newcommand{\cI}{\mathcal{I}}
\newcommand{\cJ}{\mathcal{J}}
\newcommand{\cM}{\mathcal{M}}
\newcommand{\cL}{\mathcal{L}}
\newcommand{\cR}{\mathcal{R}}

\newcommand{\mL}{L}
\newcommand{\cK}{\mathcal{K}}
\newcommand{\hathom}{\underline{\widehat{\mathrm{hom}}}}
\newcommand{\TotP}{TP}

\makeatletter
\ifx\SK@label\undefined\let\SK@label\label\fi
 \let\your@thm\@thm
 \def\@thm#1#2#3{\gdef\currthmtype{#3}\your@thm{#1}{#2}{#3}}
 \def\mylabel#1{{\let\your@currentlabel\@currentlabel\def\@currentlabel
  {\currthmtype~\your@currentlabel}
 \SK@label{#1@}}\label{#1}}
 \def\myref#1{\ref{#1@}}
\makeatother

\date{\today}

\begin{document}

\begin{abstract}
In Part 1, we describe six projective-type model structures on the category
of differential graded modules over a differential graded algebra
$A$ over a commutative ring $R$.  When $R$ is a field, the six collapse 
to three and are well-known, at least to folklore, but in the 
general case the new relative and mixed model structures offer interesting 
alternatives to the model structures in common use.  The construction of
some of these model structures requires two new variants of the small
object argument, an enriched and an algebraic one, and we describe these
more generally.  In Part 2, we present a variety 
of theoretical and calculational 
cofibrant approximations in these model categories. The classical
bar construction gives cofibrant approximations in the relative model structure,
but generally not in the usual one.  In the usual model structure, there are two
quite different ways to lift cofibrant approximations from
the level of homology modules over homology algebras, where they are 
classical projective resolutions, to the level of DG-modules over DG-algebras. 
The new theory makes model theoretic sense of earlier explicit calculations based
on one of these constructions.  A novel phenomenon we encounter is isomorphic cofibrant 
approximations with different combinatorial structure such that things proven in one 
avatar are not readily proven in the other. 
\end{abstract}

\maketitle

\begin{center} Overview \end{center}

We aim to modernize differential homological algebra model theoretically and to exhibit several 
new general features of model
category theory, the theme being how nicely the generalities of model category 
theory can interact with the calculational specificities of the subject at hand, giving concrete
results inaccessible to either alone.  This protean feature of model category
theory distinguishes it from more abstract and general foundations of homotopical algebra.

The subject of differential homological algebra began with the hyperhomology groups
of Cartan and Eilenberg \cite{CE} and continued with work of Eilenberg and Moore 
\cite{EM, Moore} in which they introduced relative homological algebra and its 
application to differential graded (abbreviated DG hencefoward) modules over a differential 
graded algebra. In \cite{EM2}, they developed the Eilenberg-Moore spectral sequence for the 
computation of the cohomology $H^*(D;R)$ in terms of differential torsion products,
where $D$ is the pullback in a diagram 
\[  \xymatrix {
D \ar[r]  \ar[d] & E \ar[d]^p\\
A \ar[r] &  B\\} \]
in which $p$ is a fibration. This work dates from the mid 1960's, and it all works
with bigraded chain bicomplexes  $X$: $X_n = \sum_{p+q=n} X_{p,q}$ is a bigraded $R$-module with 
commuting horizontal and vertical differentials and a total differential given by their 
sum (with suitable signs).

In the early 1970's, Gugenheim and May \cite{GM,M} gave an ad hoc alternative
treatment of differential homological algebra that was based on bigraded multicomplexes $X$:  
now $d\colon X_n \rtarr X_{n-1}$ is the sum over $r\geq 0$ of partial differentials
$d^r\colon X_{p,q} \rtarr X_{p-r,q+r-1}$, $r\geq 0$. Bicomplexes are the special case with
$d^r=0$ for $r\geq 2$.  The advantage of the generalization was computability, as the cited
papers show and we will illustrate shortly.  While the applications worked, the foundations
were so obscure that it was not even clear that the several definitions of differential torsion
products in sight agreed. 

This paper has several distinct purposes.  The primary purpose is to establish model theoretic foundations for
differential homological algebra over a commutative ground ring $R$ and to integrate the early work into the
modern foundations.  Specialization to a field $R$ simplifies the theory, but the force of the early applications 
depends on working more generally. Then relative homological algebra enters the picture: for DG modules over a 
DG $R$-algebra $A$, there are three natural choices for the weak equivalences: quasi-isomorphism, homotopy equivalence 
of underlying DG $R$-modules, and homotopy equivalence of DG $A$-modules.  We shall explain six related model category structures
on the category $\sM_A$ of DG $A$-modules, one or more for each of these choices.  

Here a second purpose enters.  Some of these model structures cannot be constructed using previously 
known techniques. We develop new enriched and algebraic versions of the classical small object argument
that allow the construction of model category structures that are definitely not cofibrantly generated
in the classical sense.  The classical bar construction always gives cofibrant approximations in 
one of these new relative model structures, but not in the  model structure in (implicit) common use.  
The model category foundations are explained generally, since they will surely have
other applications.

The model categorical cell complexes that underpin our model structures are given by multicomplexes,
not bicomplexes, and a third purpose is to explain the interplay between the several kinds of resolutions
in early work and our model structures.   In particular, we show that the ``distinguished resolutions'' of \cite{GM} are essentially 
model categorical cofibrant approximations.  Our work in this paper is largely model theoretic but, as we explain 
in \S\ref{TopEMSS}, 
the applications in \cite{GM, M, MN} show that it applies {\em directly} to concrete explicit 
calculations.  Here is an example whose statement makes no reference to model categorical machinery.

\begin{thm}\mylabel{Yeah} Let $H$ be a compact Lie group with maximal torus $T^n$ such that $H^*(BT^n;R)$ 
is a free 
$H^*(BH;R)$-module and let $G$ be a connected topological group such that $H^*(BG;R)$ is a polynomial algebra.
Then for any map $f\colon BH\rtarr BG$, 
\[  H^*(Ff;R) \iso \Tor^*_{H^*(BG;R)}(H^*(BH;R),R). \]
\end{thm} 
Here $H^*(BH;R)$ is an $H^*(BG;R)$-module via $f^*$. The space $Ff$ is the fiber of $f$, and it is 
$G/H$ when $f = Bi$ for an inclusion $i$ of $H$ as a closed subgroup of $G$.  
The hypothesis on $H$ holds if $H_*(H;\bZ)$ has no $p$-torsion for any prime $p$ that divides the characteristic 
of $R$. A generalization to $H$-spaces is given in \cite{MN}.  

The use of explicit distinguished resolutions given by model categorical cell complexes is the central feature of the proof. 
The connection between model categorical foundations and explicit calculations is rarely 
as close as it is here.  
\tableofcontents

\section*{Introduction}\label{algexamch}

We shall show that there are (at least) six compellingly reasonable related model 
structures on the category of DG modules over a DG algebra, and we shall show how some
of these model structures relate to explicit computations.  In fact,
calculational applications
were announced in 1968 \cite{M} and explained in the 1974 memoir \cite{GM} and its 2002 
generalization \cite{MN}.  In \cite{GM}, we gave ad hoc definitions of differential $\Tor$
functors (called ``torsion products'' in those days) and Ext functors in terms of certain 
general types of resolutions.  We wrote then that our definitions have 
``the welcome merit of brevity, although we should admit that this is largely due to the
fact that we can offer no categorical justification (in terms of projective objects, etc)
for our definitions.''  Among other things, at the price of some sacrifice of brevity, we belatedly give
model categorical justifications here.  

We let $\sM_R$ denote the category of unbounded chain complexes over a fixed ring $R$, which we
always call DG $R$-modules. We have two natural categories of weak equivalences in $\sM_R$.  We define 
$h$-equivalences to be homotopy equivalences of DG $R$-modules and $q$-equivalences to be quasi-isomorphisms, 
namely those maps of DG $R$-modules that induce an isomorphism on passage to the homology 
of the underlying chain complexes. We call the subcategories consisting of these classes 
of weak equivalences $\sW_{h}$ and $\sW_q$.  Since chain homotopic maps induce the same 
map on homology, $\sW_{h}\subset \sW_q$. Both categories are closed under retracts 
and satisfy the two out of three property.  Similarly, it will be evident that all
classes of cofibrations and fibrations that we define in this paper are subcategories 
closed under retracts. 

As usual, let $\sK_R$ denote the homotopy category of $\sM_R$ obtained by identifying
homotopic maps; it is called the classical homotopy  category of $\sM_R$.  Also as usual,
let $\sD_R$ denote the category obtained from $\sM_R$ (or $\sK_R$) by inverting the 
quasi-isomorphisms; it is called the derived category of $R$.  
We recall three familiar model structures on $\sM_R$ that lead to these homotopy
categories  in \S\ref{SecqmodelR}, \S\ref{SechmodelR}, 
and \S\ref{1mixed}. They are analogues of the Quillen, classical, 
and mixed model structures on spaces \cite{Cole, MP}. We name them as follows.

The Quillen, or projective, model structure is denoted by
\begin{equation}\label{four}  (\sW_q, \sC_q, \sF_q).  \end{equation}
The $q$-fibrations are the degreewise surjections. Its homotopy category is $\sD_R$. 

The classical, or Hurewicz, model structure is denoted by
\begin{equation}\label{one}  (\sW_h, \sC_h, \sF_h).   \end{equation}
Its homotopy category is $\sK_R$. The $h$-fibrations are the degreewise split surjections.  
We sometimes use the alternative notation $(\sW_r, \sC_r, \sF_r)$, the ``$r$'' standing for ``relative.''
In fact, we give two a priori different definitions of fibrations and cofibrations 
that turn out to be identical. When we generalize to DG $A$-modules, where $A$ is a DGA over a commutative ring
$R$, we will give different $h$- and $r$-model structures; they happen to coincide when $A=R$, but 
not in general.

We can mix these two model structures. Since we will shortly have several mixed model structures in sight, 
we denote this one\footnote{The cofibrations
were denoted $\sC_m$ in \cite{MP}, $m$ standing for mixed.} by
\begin{equation}\label{sixtoo}  (\sW_q, \sC_{q,h}, \sF_{h}).
\end{equation}
Its homotopy category is again $\sD_R$.  
For clarity of exposition, we defer all discussion of mixed model structures like this to 
\S\ref{6model}. 

 In   \S\ref{chainch} and \S\ref{1mixed}, we allow the ring $R$ to be non-commutative. Except in these sections, we use the short-hand $\otimes$ and $\Hom$ for $\otimes_R$ and $\Hom_R$.  As we explain in 
\S\ref{Rmod}, when $R$ is commutative the $r$-model 
structure on $\sM_R$ has an alternative conceptual interpretation in terms of {\em enriched} 
lifting properties and {\em enriched} weak
factorization systems.  It is compactly generated in an enriched sense,
although it is {\em not} compactly or cofibrantly generated in the traditional sense.  Here ``compactly generated'' is a 
variant of ``cofibrantly generated'' that applies when only sequential cell complexes are required.
It is described in \myref{compgen} and discussed in detail in \cite[\S15.2]{MP}. 
The variant is essential to the philosophy expounded in this paper since use of 
sequential cell complexes is needed if one is to forge a close calculational connection between the 
abstract cell complexes of model category theory and the concrete cell complexes that arise from
analogues of projective resolutions.  After all, projective resolutions in homological 
algebra are never given transfinite filtrations.

Starting in \S\ref{DGAch}, we also fix a ($\bZ$-graded) DG $R$-algebra $A$. Thus 
$A$ is a DG $R$-module and an $R$-algebra with a unit cycle in 
degree zero and a product $A\otimes A\rtarr A$ that commutes with the differentials.   
Our conventions on graded 
structures are that we never add elements in different degrees.  The product is given 
by maps $A_i\otimes A_j \rtarr A_{i+j}$ and the differential is given by maps
$d\colon A_n\rtarr A_{n-1}$.  We can shift to cohomological 
grading, $A^i = A_{-i}$, without changing the mathematics.  

We let $\sM_A$ denote the category of left DG $A$-modules.\footnote{We suppress the adjective
``left'', but we use the adjective ``right'' when appropriate. }
An object $X$ in $\sM_A$ is a DG $R$-module $X$ with an
$A$-module structure $A\otimes X\rtarr X$ that commutes with the differentials. 
We use the term $A$-module when we choose to forget the differential and consider 
only the underlying (graded) $A$-module structure, as we shall often have occasion to do.  

In \S\ref{DGAch}, which is parallel to \S\ref{chainch}, we define Quillen and classical model
structures on $\sM_A$,  using the same notations as in (0.1) and (0.2). The maps in $\sW_q$ 
are the quasi-isomorphisms and the maps in $\sW_h$ are the homotopy equivalences of DG $A$-modules.  
The $q$-fibrations, like the $q$-equivalences, are created in $\sM_R$ and thus depend 
only on the underlying DG $R$-modules.  The $h$-fibrations are the maps that satisfy 
the covering homotopy property in the category $\sM_A$; they do not appear  to admit an easily 
verifiable characterization in more familiar algebraic terms. We defer
discussion of the associated mixed model structure generalizing (0.3) to \S\ref{3mixed}.

There is a subtlety in proving the factorization axioms for the $h$-model structure,
but to minimize interpolations of general theory in the direct line of 
development, we have deferred the relevant model categorical underpinnings to \S\ref{appendix}.
If $A$ has zero differential, the $h$-and $q$-model structures and the associated mixed model structure 
are the obvious generalizations from ungraded rings $R$ to graded rings $A$ of the model 
structures in \S\ref{SecqmodelR} and \S\ref{SechmodelR}, and the differential adds relatively
little complication.  These model structures are independent of the assumption that $A$ is an 
$R$-algebra, encoding no more information than if we regard $A$ as a DG ring.
 
We are interested in model structures that remember that $A$ is an $R$-algebra.
We define $\sW_r$ to be the category of maps of DG $A$-modules that are homotopy
equivalences of DG $R$-modules.  These are the appropriate equivalences for {\em relative} 
homological algebra, which does remember $R$.  Of course, 
\[ \sW_h\subset \sW_r\subset \sW_q. \] 
We consider $\sW_r$ to be a very natural category of weak equivalences in $\sM_A$, and we 
are interested in model structures with these weak equivalences and their relationship with 
model structures that take $\sW_h$ or $\sW_q$ as the weak equivalences.

We have three homotopy categories of DG $A$-modules. We let $\sK_A$ denote the ordinary homotopy 
category of $\sM_A$ 
and call it the absolute homotopy category.  It is obtained from $\sM_A$ by passing to homotopy classes 
of maps or, equivalently, by inverting the homotopy equivalences of DG $A$-modules.
We let $\sD_A^r$ denote the homotopy category obtained by formally inverting the 
$r$-equivalences.  We let 
$\sD_A$  denote the category obtained from $\sM_A$, or equivalently from 
$\sK_A$ or $\sD^r_A$, by formally inverting the quasi-isomorphisms.  It is called the 
{\em derived category} of the category of DG $A$-modules.  We call $\sD_A^r$ the 
{\em relative derived category} of $A$. We hope to convince the reader that $\sD_A^r$ is 
as natural and perhaps even as important as $\sD_A$. 

In \S\ref{sec:rmodel}, we  construct the relative model structure
\begin{equation}\label{two} (\sW_{r}, \sC_r, \sF_{r}). \end{equation}
The $r$-fibrations are the maps in $\sM_A$ that are $r$-fibrations (= $h$-fibrations) when 
regarded as maps in $\sM_R$.  That is, like the $r$-equivalences, the $r$-fibrations are created 
by the $r$-model structure on $\sM_R$.  Here again there is a subtlety in the proof of the
factorization axioms, discussion of which is deferred to \S\ref{appendix}.

Along with the inclusions $\sW_h\subset \sW_r\subset \sW_q$, we have inclusions
\[ \sF_h\subset \sF_r\subset \sF_q. \]
There result three mixed model structures on $\sM_A$, the $(r,h)$-model structure 
\begin{equation}\label{three}  
(\sW_{r}, \sC_{r,h}, \sF_{h}) 
\end{equation}
and the $(q,r)$-model structure
\begin{equation}\label{six} (\sW_q, \sM_{q,r}, \sF_{r}).
\end{equation}
joining the $(q,h)$-model structure $(\sW_q, \sM_{q,h}, \sF_{h})$ that generalizes (0.3). We 
discuss these in \S\ref{3mixed}. 
They have advantages over the $q$- and $r$-model structures analogous to those
described in \S\ref{1mixed} and in more detail in \cite[\S18.6]{MP} in the classical 
case of model structures on $\sM_R$. 

In all of these model structures, all objects are fibrant.  By an observation of
Joyal, two model structures with the same cofibrations and fibrant objects are the same (cf.~\cite[15.3.1]{Riehl}).  Thus, in principle, our six model structures differ only in their cofibrations.  We shall see in \S\ref{appendix} that recent work in model category theory \cite{BR, Garner, Riehl} illuminates the cofibrations in our new model structures.  

However, the distinction we emphasize is seen most clearly in the fibrations.  The lifting property that defines $q$-fibrations implies that they are degreewise surjections.  The lifting property that defines $r$-fibrations implies that they are degreewise split surjections.  The splittings promised by the lifting properties are merely functions in the former case, but they are maps of $R$-modules in the latter case.  The new theory explains the distinction in terms of enriched model category theory.  As we describe in \S\ref{Rmod}, when $R$ is commutative the $(h=r)$-model structure on $\sM_R$ is the $R$-module enrichment of its $q$-model structure, in a sense that we shall make precise.
Similarly, as we explain in \S\ref{sec:rmodel}, the $r$-model structure on $\sM_A$ is the $R$-module enrichment of its $q$-model structure.

The construction of our $h$- and $r$-model structures on $\sM_A$ requires new model theoretic foundations,
without which we would not know how to prove the factorization axioms.  In \S\ref{appendix},
we introduce ``enriched'' and ``algebraic'' generalizations of Garner's variant of Quillen's small object argument (SOA). We shall implicitly use Garner's variant in all of our model theoretic work, and we shall use its generalized 
versions to obtain the required factorizations.  

This material is of independent interest in model category theory, and we have collected it in \S\ref{appendix} 
both to avoid interrupting the flow and to make it more readily accessible to readers interested in other applications.  
These results expand on work of two of us in \cite{BR}, where mistakes in the literature concerning 
$h$-model structures in topology are corrected.  The new variants of the SOA provide systematic general ways to construct interesting model structures that are not cofibrantly generated in the classical sense. As the quite different applications in \cite{BR} and here illustrate, the new theory can be expected to apply to a variety of situations to which the ordinary SOA does not apply.  That is a central theme of Part 1. 

In Part 2, we are especially interested in understanding $q$- and 
$r$-cofibrant approximations and relating them to projective resolutions in traditional homological 
algebra. We shall give three homological constructions of cofibrant approximations that a priori bear no 
obvious relationship to the model theoretic cofibrant approximations provided by either the classical or the
enriched SOA.  

Beginning with $q$-cofibrant approximations, we show in \S\ref{versus} that the classical projective 
resolutions of DG $R$-modules that Cartan and Eilenberg introduced and used to construct the K\"unneth 
spectral sequence in \cite[XVII]{CE} give $q$-cofibrant approximations of DG $R$-modules, 
even though they are specified as bicomplexes with no apparent relationship to the retracts of
$q$-cell complexes that arise from model category theory.  They are {\em isomorphic} to such retracts,
but there is no obvious way to construct the isomorphisms, which can be viewed as changes of filtrations.

More generally, in \S\ref{versustoo} we show that we can obtain $q$-cofibrant approximations of DG $A$-modules as the total complexes $\TotP$ of projective resolutions $P$, where the $P$ are suitable bicomplexes. 
The construction is due to Moore \cite{Moore}, generalizing Cartan and Eilenberg \cite[XVII]{CE}.
The $\TotP$ must be retracts of $q$-cell complexes, but, as bicomplexes, they come in nature with entirely 
different non-cellular filtrations and it is not obvious how to compare filtrations. 
Precisely because they are given in terms of bicomplexes, they allow us to prove some things that are not 
readily accessible to $q$-cell complexes.  For example, these $q$-cofibrant approximations allow us to derive 
information from the assumption that the underlying $A$-module of a DG $A$-module is flat and to view the 
Eilenberg-Moore spectral sequence (EMSS) as a generalized K\"unneth spectral sequence under appropriate hypotheses. 

We head towards alternative cofibrant approximations in \S\ref{GMStuff}.   We give theorems that characterize the $q$- and $r$-cofibrations and cofibrant objects in the parallel sections \S\ref{qchar} and \S\ref{rchar}. In \S\ref{splitone} and
\S\ref{splittwo}, we introduce a common generalization of model theoretic cell DG $A$-modules and the total complexes
$\TotP$ of projective resolutions, together with a concomitant generalization of model theoretic cofibrations.  The key notion is that of a split DG $A$-module, which was already defined in \cite{GM}.  The model theoretic $q$-cell and $r$-cell 
DG $A$-modules, the projective DG $A$-modules of \S\ref{versustoo}, and the classical bar resolutions are
all examples of split DG $A$-modules.

We single out a key feature of split DG $A$-modules.  Prior to \cite{GM}, differential homological algebra used only 
bicomplexes, as in our \S\ref{projqcof}.  Split DG $A$-modules are
``multicomplexes,''\footnote{Multicomplexes in the sense used here were first introduced 
in a brief paper of Wall \cite{Wall}.} which means that they are filtered and have differentials with filtration-lowering components that are closely related to the differentials of the associated spectral sequences. We now see that the generalization from bicomplexes to multicomplexes in \cite{GM}, which then seemed esoteric and artificial,   
is forced by model theoretic considerations: our $q$- and $r$-model structures are constructed in 
terms of $q$-cell and $r$-cell complexes, as dictated by the SOA, and these are multicomplexes,
almost never bicomplexes. 

In \S\ref{GMStuff1}, we head towards applications by relating split DG $A$-modules to the EMSS.  The differentials in
the EMSS are built into the differentials of the relevant multicomplexes and they have interpretations in 
terms of matric Massey products, as we indicate briefly. We illustrate the use of this interpretation in
\S\ref{ExtA}, where we recall from \cite{GM} that when $A$ is a connected algebra (not DG algebra) over a field $R$, 
$\Ext_A(R,R)$ is generated under matric Massey products by its elements of degree $1$, which are the duals of
the indecomposable elements of $A$.

In \S\ref{barsec}, we return to the relationship between the $q$- and $r$-model structures. We show that the 
bar construction always gives $r$-cofibrant approximations.  Unless $R$ is a field, the bar construction is 
usually not $q$-cofibrant, but when $A$ is $R$-flat, for example 
when $A=C^*(X;R)$ for a space $X$ and a commutative Noetherian ring $R$, bar constructions very often behave homologically as if they were 
$q$-cofibrant or at least $(q,h)$-cofibrant, although they are generally not.  Precisely, we prove
that they give ``semi-flat resolutions'' under mild hypotheses.  This implies that the two different definitions 
of differential torsion products obtained by applying homology to the tensor product derived from the $q$- and 
$r$-model structures agree far more often than one would expect from model categorical considerations alone.

In \S\ref{GMStuff2}, which follows \cite{GM}, we show how to start from a classical projective resolution 
of $H_*(M)$ as an $H_*(A)$-module and construct from it a ``distinguished resolution'' $\epz\colon X\rtarr M$
of any given DG $A$-module $M$.  This resolution is very nearly a $q$-cofibrant approximation: $X$ is $q$-cofibrant, 
and $\epz$ is a $q$-equivalence.  However, $\epz$ need not be a degreewise epimorphism, which means
that $\epz$ need not be a $q$-fibration.  It follows that $X$ is $h$-equivalent over $M$ to any chosen 
$q$-cofibrant approximation  $Y\rtarr M$, so there is no loss of information.  The trade-off is a huge gain in
calculability.  We show how this works explicitly when $H_*(A)$ is a polynomial algebra in \S\ref{cupone}.
In turn, we show how this applies to prove \myref{Yeah} in \S\ref{TopEMSS}. 

Our work displays a plethora of different types of cell objects, ranging from general
types of cell objects used in our enriched and algebraic variants of the SOA in \S{\ref{appendix}
to special types of cell objects used for both calculations and theoretical
results in our specific category $\sM_A$ of DG $A$-modules.  Focusing on cellular
approximations, we have two quite different special types of $q$-cofibrant approximations, 
namely {\em distinguished resolutions}, which are defined in \myref{ancdef} and 
constructed in \S{\ref{resolve}}, and {\em projective resolutions}, which are defined and 
constructed in \S{\ref{versustoo}}. The former are multicomplexes and the latter are bicomplexes.  
It is almost never the case that a resolution is both distinguished and projective, and 
each is used to prove things we do not know how to prove with the other.  Both are
examples of K\"unneth resolutions, which are defined in \S\ref{Kunndef} and which give
precisely the right generality to construct the algebraic EMSS but are not always $q$-cofibrant. 
We also have the bar resolution in \S\ref{barsec}}, which always gives $r$-cofibrant 
approximations and sometimes gives $q$-cofibrant approximations. Without exception, all of these 
types of DG $A$-modules are examples of split DG $A$-modules, as defined in \myref{ancdef}. 

We are moved to offer some philosophical comments about model category theory in general.
In serious applications within a subject, it is rarely if ever true that 
all cofibrant approximations of a given object are of equal calculational
value.  The most obvious example is topological spaces, where the 
general cell complexes given by the SOA are of no particular interest
and one instead works with CW complexes, or with special types of CW
complexes. This is also true of spectra and much more so of $G$-spectra, 
where the calculational utility of different types of cell complexes depends 
heavily on both the choice of several possible Quillen equivalent model categories 
in which to work and the choice of cell objects within the chosen category; 
see \cite[\S IV.1]{MM} and \cite[\S 24.2]{MaySig} for discussion.

Philosophically, our theory epitomizes the virtues of model category theory, illustrating the 
dictum ``It is the large generalization, limited by a happy particularity, which is the fruitful conception.''\footnote{G.H. Hardy \cite[p.~109]{Hardy}, quoting 
A.N. Whitehead.}  Because model category theory axiomatizes structure that is already present in the 
categories in which one is working, it can be combined directly with those particulars that enable concrete calculations: it works within the context at hand rather than translating it to one that is chosen for purposes of 
greater generality and theoretical convenience, however useful that may sometimes be (albeit rarely if ever for purposes of calculation).

It will be clear to the experts that some of our work can be generalized from DG algebras to DG categories.  
We will not go into that, but we hope to return to it elsewhere.  It should be clear to everyone that generalizations and analogues in other contexts must abound. Model structures as in Part I should appear whenever one has a category $\sM$ of structured objects enriched in a  category $\sV$ with two canonical model structures 
(like the $h$- and $q$-model structures on spaces and on DG $R$-modules).  The category $\sM$ then has three natural notions of weak equivalences, the structure preserving homotopy equivalences ($h$-equivalences), the homotopy equivalences of underlying objects in $\sV$ ($r$-equivalences), and the weak equivalences of underlying objects in 
$\sV$ ($q$-equivalences).  These can be expected to yield $q$-, $r$-, and $h$-model structures with accompanying mixed $(r,h)$-, $(q,h)$-, and $(q,r)$-model structures.

\vspace{2mm}

{\em Dedication:} The authors dedicate this paper to John Moore, who pioneered this area of mathematics.  
He was the senior author's adviser, and his mathematical philosophy pervades this work and indeed pervades
algebraic topology at its best. 

{\em Acknowledgments:} We thank Takashi Suzuki for reminding us of Moore's early paper 
\cite{Moore},\footnote{It appears
in a 1959-60 Cartan Seminar and is not on MathSciNet.} which he has found useful
in new applications of Mac\,Lane homology in algebraic geometry.  

The third author was supported by a National Science Foundation postdoctoral research fellowship  DMS-1103790.

\part{Six model structures for DG-modules over DGAs}

\section{The $q$- and $h$-model structures on the category $\sM_R$}\label{chainch}

Although $R$ will be required to be commutative later, $R$ can be any
ring in this section.  We describe the $q$- and $h$-model structures 
on the category $\sM_R$ of (left) DG $R$-modules. 
In particular, of course, we could replace $R$ by an algebra $A$ regarded 
just as a ring.  This section is a summary of material treated in
detail in \cite{MP}, to which we refer the reader for all proofs. 

\subsection{Preliminaries}\label{ssec:prelim}

The category $\sM_R$ is bicomplete. Limits and colimits 
in $\sM_R$ are just limits and colimits of the underlying graded 
$R$-modules, constructed degreewise, with the naturally induced differentials.  We
reserve the term $R$-module for an ungraded $R$-module, and we often regard
$R$-modules as DG $R$-modules concentrated in degree zero.  

It is convenient to use the category theorists' notion of a cosmos, namely a
bicomplete closed symmetric monoidal category.  When $R$ is commutative, 
$\sM_{R}$ is a cosmos under $\otimes_R$ and $\Hom_R$.  In this section, 
we use the cosmos $\sM_{\bZ}$, and we write $\otimes$ and $\Hom$ for tensor products
and hom functors over $\bZ$.  Recall that 
\[(X\otimes Y)_n= \sum_{i+j=n} X_i\otimes Y_j\quad \mathrm{and}\quad
\Hom(X,Y)_n = \prod_{i} \Hom(X_i,Y_{i+n})\] 
with differentials given by 
\[ d(x\otimes y) = d(x)\otimes y + (-1)^{\mathrm{deg}x}x\otimes d(y)
\ \ \ \mathrm{and}\ \ \  (df)(x) = d(f(x)) - (-1)^nf(d(x)). \]

The category $\sM_R$ is enriched, tensored, and cotensored over $\sM_{\bZ}$.  We 
say that it is a bicomplete $\sM_{\bZ}$-category.  The chain complex
(DG $\mathbf{\bZ}$-module) of morphisms $X\rtarr Y$ is $\Hom_R(X,Y)$, 
where $\Hom_R(X,Y)$ is the subcomplex of $\Hom(X,Y)$ consisting of those 
maps $f$ that are maps of underlying $R$-modules.  Tensors are given
by tensor products $X\otimes K$, noting that the tensor product of a 
left $R$-module and an abelian group is a left $R$-module.  Similarly,
cotensors are given by $X^K = \Hom(K,X)$.  Explicitly,
for $X\in\sM_R$ and $K\in\sM_{\bZ}$, the chain complexes $X\otimes K$ and $\Hom(K,X)$ are 
DG $R$-modules with $r(x\otimes k) = (rx)\otimes k$ and $(rf)(k) = rf(k)$ 
for $r\in R$, $x\in X$, $k\in K$, and $f\in \Hom(K,X)$. We have the adjunctions
\[ \Hom_R(X\otimes K,Y) \iso \Hom(K,\Hom_R(X,Y))\iso \Hom_R(X,Y^K). \]

To emphasize the analogy with topology, we give algebraic objects topological
names.  Since the zero module $0$ is initial and terminal in $\sM_R$, the analogy 
is with based rather than unbased spaces.
For $n\in\bZ$, we define $S^n$, the $n$-sphere chain complex, to be $\bZ$
concentrated in degree $n$ with zero differential. For any integer $n$, we define 
the $n$-fold suspension $\SI^nX$ of a DG $R$-module $X$ to be $X\otimes S^n$. 
Thus $(\SI^nX)_{n+q}\iso X_q$.  The notation is motivated by the observation that if 
we define $\pi_n(X)$ to be the abelian group of chain homotopy classes of maps 
$S^n\rtarr X$ (ignoring the $R$-module structure on $X$), then $\pi_n(X) = H_n(X)$.  

Analogously, we define $D^{n+1}$ to be
the $(n+1)$-disk chain complex.  It is $\bZ$ in degrees $n$ and $n+1$ and zero in
all other degrees. There is only one differential that can be non-zero, and 
that differential is the identity map
$\bZ\rtarr \bZ$.  The copy of $\bZ$ in degree $n$ is identified with $S^n$ and is
the boundary of $D^{n+1}$.  We write $S^n_R = R\otimes S^n$
and $D^{n+1}_R = R\otimes D^{n+1}$. 

We define $I$ to be the chain complex with one basis element $[I]$  in degree $1$, 
two basis elements $[0]$ and $[1]$ in degrees $0$, and differential $d([I]) = [0]-[1]$.  
A homotopy $f\htp g$ between maps of DG $R$-modules
$X\rtarr Y$ is a map of DG $R$-modules $h\colon X\otimes I\rtarr Y$ that restricts 
to $f$ and $g$ on $X\otimes [0]$ and $X\otimes [1]$.  Letting
$s(x) =(-1)^{\mathrm{deg}\, x}h(x\otimes [I])$, 
$h$ specifies a chain homotopy $s\colon f\htp g$ in the usual sense. In all
of our model structures, this notion of homotopy can be used interchangeably
with the model categorical notion of homotopy. 

\begin{rmk}
To elaborate, the natural cylinder object $X \otimes I$ is not necessarily a cylinder object in the model theoretic sense because the map $X \coprod X \rtarr X \otimes I$ is not necessarily a cofibration. We will see that it is always an $h$- and $r$-cofibration, but $X$ must be $q$-cofibrant to 
ensure that it is a $q$-cofibration.  However this subtlety is immaterial since \cite[16.4.10 and 16.4.11]{MP} ensure that the classical and model theoretic notions of homotopy really can be used
interchangeably.
\end{rmk}

\subsection{The $q$-model structure}\label{SecqmodelR} This is the model structure in standard use. 

\begin{defn}\mylabel{qRmod} Let $\cI_R$ denote the set of inclusions 
$S^{n-1}_R\rtarr D^n_R$
for all $n\in \bZ$ and let $\cJ_R$ denote the set of maps $0 \rtarr D_R^n$ for all $n \in \bZ$.
A map in $\sM_R$ is a {\em $q$-fibration} if it satisfies the right lifting property (RLP) against $\cJ_R$.  
A map is a {\em $q$-cofibration} if it satisfies the left lifting property (LLP) against all $q$-acyclic $q$-fibrations, 
which are the maps that have the RLP against $\cI_R$.  Let $\sC_q$ and $\sF_q$ denote the subcategories of 
$q$-cofibrations and $q$-fibrations. Recall that $\sW_q$ denotes the subcategory of quasi-isomorphisms of DG $R$-modules. 
\end{defn}

\begin{rmk} In \cite{MP}, $\cJ_R$ was taken to be the set of maps 
$i_0\colon D^n_R\rtarr D^n_R\otimes I$ for all $n\in \bZ$ in order to emphasize 
the analogy with topology.  The proof of \cite[18.4.3]{MP} makes clear that either 
set can be used.
\end{rmk}

One proof of the following result is precisely parallel to that of its topological
analogue, but there are alternative, more algebraically focused, arguments.
Full details are given in \cite{MP} and elsewhere.

\begin{thm}\mylabel{qmodelR}  
The subcategories $(\sW_q,\sC_q,\sF_{q})$ 
define a compactly generated model category structure on $\sM_R$ 
called the $q$-model structure.  The sets $\cI_R$ and $\cJ_R$ are generating 
sets for the $q$-cofibrations and the $q$-acyclic $q$-cofibrations.
Every object is $q$-fibrant and the $q$-model structure 
is proper.  If $R$ is commutative, the cosmos $\sM_R$ is a 
monoidal model category under $\otimes$. In general,
$\sM_R$ is an $\sM_{\bZ}$-model category. 
\end{thm}

It is easy to characterize the $q$-fibrations
directly from the definitions.

\begin{prop}\mylabel{fibepi} A map is a $q$-fibration 
if and only if it is a degreewise epimorphism. 
\end{prop} 

Of course, one characterization of the $q$-cofibrations and $q$-acyclic $q$-cofibrations
is that they are retracts of relative $\cI_R$-cell complexes and relative $\cJ_R$-cell 
complexes; cf.~\myref{cellob} and \myref{SOA}. We record several alternative
characterizations. 

\begin{defn} A DG $R$-module $X$ is {\em $q$-semi-projective}  if it is degreewise projective
and if $\Hom_R(X,Z)$ is $q$-acyclic for all $q$-acyclic DG $R$-modules $Z$. 
\end{defn}

\begin{prop}\mylabel{qcofibrant}  Let $X$ be a DG $R$-module and consider the following 
statements.
\begin{enumerate}[(i)]
\item $X$ is $q$-semi-projective
\item $X$ is $q$-cofibrant
\item $X$ is degreewise projective
\end{enumerate}
Statements (i) and (ii) are equivalent and imply (iii); if $X$ is bounded below,
then (iii) implies (i) and (ii).  Moreover, $0\rtarr X$ is a $q$-acyclic $q$-cofibration 
if and only if $X$ is a projective object of the category $\sM_R$. 
\end{prop}

We return to the $q$-cofibrant objects in \S\ref{versus}, where we use 
\myref{qcofibrant} to show that every DG $R$-module $M$ has 
a $q$-cofibrant approximation that a priori looks nothing like a 
retract of an $\cI_R$-cell complex.  We prove a generalization of
\myref{qcofibrant} in \myref{qcharcofibrant}.

\begin{rmk}\mylabel{projectives} If all $R$-modules are projective, that is if $R$ is semi-simple, then all objects of $\sM_R$ are $q$-cofibrant (see \myref{Ohyeah}). However, in general (iii) does not imply (i) and (ii). Here is a well-known counterexample (see e.g.~\cite[1.4.2]{Weibel}). Let $R = \bZ/4$ and let $X$ be the degreewise free $R$-complex
\[\xymatrix@1{ \cdots \ar[r]^{2} & \bZ/4 \ar[r]^{2} & \bZ/4 \ar[r]^{2} & \cdots}. \]
Then $X$ is $q$-acyclic.  Remembering that all objects are $q$-fibrant, so that a $q$-equivalence
between $q$-cofibrant objects must be an $h$-equivalence, we see that $X$ cannot be $q$-cofibrant
since it is not contractible.  
\end{rmk}

\begin{prop}\mylabel{qcofibration}  A map $i\colon W\rtarr Y$ is a $q$-cofibration if and only if it is a mono\-morphism and $Y/W$ is $q$-cofibrant, and then $i$ is a degreewise split monomorphism.
\end{prop}

Regarding an ungraded $R$-module $M$ as a DG $R$-module concentrated in
degree $0$, a $q$-cofibrant approximation of $M$ is exactly a projective 
resolution of $M$.  There is a dual model structure that encodes injective 
resolutions \cite[2.3.13]{Hovey}, but we shall say nothing about that in this paper.

\subsection{The $h$-model structure}\label{SechmodelR}

The topological theory of $h$-cofibrations and $h$-fibrations transposes directly
to algebra. 

\begin{defn}\mylabel{coffibalg}  An {\em $h$-cofibration}
is a map $i$ in $\sM_R$ that satisfies the 
homotopy extension property (HEP).
That is, for all DG $R$-modules $B$, $i$ satisfies the LLP against the map 
$p_0\colon B^I\rtarr B$ given by evaluation at the zero cycle $[0]$.
An {\em $h$-fibration} is a map $p$ that satisfies the 
covering homotopy property (CHP). That is, for all DG $R$-modules $W$,
$p$ satisfies the RLP against the map $i_0\colon W\rtarr W\otimes I$. 
Let $\sC_h$ and $\sF_h$ denote the classes of $h$-cofibrations and $h$-fibrations.
Recall that $\sW_h$ denotes the subcategory of homotopy equivalences of DG $R$-modules. 
\end{defn}  

An elementary proof of the model theoretic versions of the lifting properties 
of $h$-cofibrations and $h$-fibrations can be found in \cite{MP}, but here we want to emphasize a parallel set of
definitions that set up the framework for our later work.
In fact, the $h$-cofibrations and $h$-fibrations admit a more familiar description, which 
should be compared with the description of $q$-cofibrations and $q$-fibrations given by Propositions \ref{fibepi} and \ref{qcofibration}.

\begin{defn}  A map of DG $R$-modules is an 
{\em $r$-cofibration} if it is a degreewise split monomorphism. It is an 
{\em $r$-fibration} if it is a  degreewise split epimorphism.  We use the
term $R$-split for degreewise split from now on.  
\end{defn}

Of course, such splittings are given by maps of underlying 
graded $R$-modules that need not be maps of DG $R$-modules.
However, the splittings can be deformed to DG $R$-maps if 
the given $R$-splittable maps are $h$-equivalences. 

\begin{prop}\mylabel{Rsplitrh} Let
\[\xymatrix@1{ 0 \ar[r] & X\ar[r]^-{f} & Y \ar[r]^-{g} & Z \ar[r] & 0\\} \]
be an exact sequence of DG $R$-modules whose underlying exact sequence of 
$R$-modules splits.  If $f$ or $g$ is an $h$-equivalence, then the 
sequence is isomorphic under $X$ and over $Z$ to the canonical split exact sequence
of DG $R$-modules
\[\xymatrix@1{ 0 \ar[r] & X\ar[r] & X\oplus Z \ar[r] & Z \ar[r] & 0.\\} \]
\end{prop}

This result does {\em not} generalize to DG $A$-modules.  In the present context, it leads 
to a proof of the $r$-notion half of the following result.  The $h$-notion half
is proven in analogy with topology and will generalize directly to DG $A$-modules.

\begin{prop}\mylabel{liftrh} Consider a commutative diagram of DG $R$-modules
\[ \xymatrix{
W\ar[r]^-g \ar[d]_i & E \ar[d]^p\\
X \ar@{-->}[ur]^{\la} \ar[r]_-f & B.\\} \]
Assume either that $i$ is an $h$-cofibration and $p$ is an $h$-fibration
or that $i$ is an $r$-cofibration and $p$ is an $r$-fibration. If 
either $i$ or $p$ is an $h$-equivalence, then there exists a lift $\la$.
\end{prop}

In turn, this leads to a proof that our $r$-notions and $h$-notions coincide.

\begin{prop}\mylabel{reassure}  A map of DG $R$-modules is an $h$-cofibration if and only if it is an 
$r$-cofibration; it is an $h$-fibration if and only if it is an $r$-fibration.
\end{prop}

\begin{thm}\mylabel{hmodelR}  
The subcategories $(\sW_h,\sC_h,\sF_{h})$ 
define a model category structure on $\sM_R$ called the $h$-model 
structure.  The identity functor is a Quillen right adjoint from the 
$h$-model structure to the $q$-model structure. 
Every object is $h$-cofibrant and $h$-fibrant, hence
the $h$-model structure is proper.  If $R$ is commutative, the cosmos 
$\sM_R$ is a monoidal model category under $\otimes$. In general,
$\sM_R$ is an $\sM_{\bZ}$-model category.  
\end{thm}

\begin{rmk}\mylabel{factrmk}
Implicitly, we have two model structures on $\sM_R$ that 
happen to coincide. If we define an $r$-equivalence to be an $h$-equivalence, 
then \myref{reassure} says that the $h$-model structure and the $r$-model 
structure on $\sM_R$ are the same.  An elementary proof of the factorization axioms for the $(h=r)$-model structure is given in \cite{MP} and sketched above. However, that argument does {\em not} extend to either the $h$-model structure or the $r$-model structure on $\sM_A$.  
\end{rmk}

\begin{rmk}  Christensen and Hovey \cite{CH}, Cole \cite{Cole3}, and Schw\"anzl and Vogt \cite{SV}
all noticed the $h$-model structure on $\sM_R$ around the year 2000.
\end{rmk}

\section{The $r$-model structure on $\sM_R$ for commutative rings $R$}\label{Rmod}

\subsection{Compact generation in the $R$-module enriched sense}

Let us return to the $r$-model structure on $\sM_R$, which happened to coincide with the $h$-model structure.  
While that observation applies to any $R$, we can interpret it more conceptually when $R$ is
commutative, which we assume from here on out (aside from \S\ref{1mixed}).  Recall that a map $p \colon E \rtarr B$ is an $r$-fibration if and only if it is 
an $R$-split epimorphism, that is, if and only if it admits a section as a map of graded $R$-modules. A key observation
is that this definition can be encoded via an {\em enriched} reformulation of the lifting property
\begin{equation}
\label{eq:samplelift} \xymatrix{ 0 \ar[d] \ar[r] & E \ar[d]^p \\ D^n_R \ar@{-->}[ur] \ar[r] & B.}
\end{equation} 
Letting the bottom arrow vary and choosing lifts, if $p$ is a $q$-fibration
we obtain a section of $p_n$ in the category of sets for each $n \in \mathbb{Z}$. For $p$ 
to be an $r$-fibration, we must have sections that are maps of \emph{$R$-modules} and not 
just of sets, and that is what the enrichment of the lifting property encodes.

Our interest in this enrichment is two-fold.  Firstly, it precisely characterizes the $r$-fibrations, proving that they are ``compactly generated'' in the $R$-module enriched sense, despite the fact that this class is generally   \emph{not} compactly generated in the usual sense \cite[\S 5]{CH}. This observation will allow us to construct the $r$-model structure on 
$\sM_A$ by an enriched variant of the standard procedure for lifting compactly generated model structures along adjunctions.

Secondly, and more profoundly, our focus on enrichment in the category of (ungraded) $R$-modules precisely describes the difference between the $r$-model structure and the $q$-model structure on both $\sM_R$ and $\sM_A$.  Interpreted in the usual (set-based) sense, the lifting property displayed in \eqref{eq:samplelift} characterizes the 
$q$-fibrations: $q$-fibrations are degreewise epimorphisms, that is, maps admitting a section given by a map of underlying graded \emph{sets}. 
The notion of $R$-module enrichment transforms $q$-fibrations into $r$-fibrations. Similarly, $R$-module enrichment transforms $q$-acyclic $q$-fibrations into $r$-acyclic $r$-fibrations.  We summarize these results in a theorem, which will be proven in \S\ref{subsec:enrichedr} below.

\begin{thm}\mylabel{enriched-comparison}
Let $R$ be a commutative ring and define \[ \cI_R = \{ S^{n-1}_R \rtarr D^n_R \mid n \in \mathbb{Z} \} \qquad \text{and} \qquad \cJ_R = \{ 0 \rtarr D^n_R \mid n \in \mathbb{Z}\}.\]  Then $\cI_R$ and $\cJ_R$ are generating sets of cofibrations and acyclic cofibrations for the $q$-model structure, when compact generation is understood in the usual set based sense, and for the $r$-model structure, when compact generation is understood in the $R$-module enriched sense.
\end{thm}

The role of $R$-module enrichment in differentiating the $r$- and $q$-model structures is also visible on the side of the cofibrations. Among the $q$-cofibrations are the relative cell complexes.  They are maps that can be built as countable composites of pushouts of coproducts of the maps $S^{n-1}_R \rtarr D^n_R$; 
see~\myref{cellob} . We refer to
these as the $q$-cellular cofibrations. Any $q$-cofibration is a retract of a $q$-cellular cofibration. 

By contrast, among the $r$-cofibrations are the enriched relative cell complexes.
They are maps that can be built as countable composites of pushouts of coproducts of tensor products of the maps 
$S^{n-1}_R \rtarr D^n_R$ with any (ungraded) $R$-module $V$; see~\myref{cellob2}. If $R$ is not semi-simple, we have $R$-modules $V$ that are not projective, 
and they are allowed.  We refer to these as the $r$-cellular cofibrations. Any $r$-cofibration is a retract 
of an $r$-cellular cofibration. Clearly $q$-cofibrations are $r$-cofibrations, but not conversely.

This discussion, including \myref{enriched-comparison}, will generalize without change to $\sM_A$, as we 
shall see in \S\ref{subsec:rforDGA}.

\begin{rmk} We have often used the term enrichment, and it will help if the reader has seen some enriched category theory.  In fact, the category $\sM_R$ is naturally enriched in three different
categories: the category $\sV_R$ of ungraded $R$-modules, the category of graded $R$-modules, and itself (since it is closed symmetric monoidal).  Our discussion focuses on enrichment in $\sV_R$ for simplicity and relevance.  The $\sV_R$-enriched hom objects in $\sM_R$ are just the $R$-modules $\sM_R(M,N)$ of maps of DG $R$-modules $M\rtarr N$, so the reader unfamiliar with 
enriched category theory will nevertheless be familiar with the example we use. 
\end{rmk}

\subsection{The enriched lifting properties} 

We recall the definition of a weak factorization system (WFS) in \myref{wfs}, but this structure is already familiar:   The most succinct 
among the equivalent definitions of a model structure is that it consists of a class $\sW$ of maps that satisfies
the two out of three property together with two classes of maps $\sC$ and $\sF$ such that $(\sC \cap \sW,\sF)$ 
and $(\sC,\sF \cap \sW)$ are WFSs.  This form of the definition is due to Joyal and
Tierney \cite[7.8]{JoyalTierney}, and expositions are given in  \cite{MP, Riehl}.  Quillen's SOA, which we use in the original sequential form given in \cite{Quillen},  codifies a procedure for constructing (compactly generated) WFSs.   

There are analogous {\em enriched} WFSs, as defined in \myref{ewfs}.  A general treatment
is given in \cite[Chapter 13]{Riehl}, but we shall only consider enrichment in the cosmos
$\sV_R$ of $R$-modules, with monoidal structure given by the tensor product. Henceforth, we say ``enriched'' to mean 
``enriched over $\sV_R$''.  From now on, for DG $R$-modules $M$ and $N$ we agree to write $M\otimes N$ and 
$\Hom(M,N)$ for the DG $R$-modules $M\otimes_R N$ and $\Hom_R(M,N)$, to simplify notation.  With this notation, $\sM_R(M,N)$ is the $R$-module of degree zero cycles 
in $\Hom(M,N)$.

Since $\sV_R$ embeds in $\sM_R$ as the chain complexes concentrated in degree zero, 
$M\otimes V$ and $\Hom(V,M)$ are defined for $R$-modules $V$ and DG $R$-modules $M$.
Categorically, these give tensors and cotensors in the $\sV_R$-category $\sM_R$. Since $\sM_R$ is
bicomplete in the usual sense, this means that $\sM_R$ is a bicomplete $\sV_R$-category: it has all enriched limits and colimits, and the ordinary limits and colimits satisfy enriched universal properties. 

Enriched WFSs are defined in terms of enriched lifting properties, which we specify here.
Let $i \colon W \rtarr X$ and $p \colon E \rtarr B$ be maps of DG $R$-modules. 
Let $\Sq(i,p)$ denote the $R$-module (not DG $R$-module) of commutative squares 
from $i$ to $p$ in $\sM_R$. It is defined via the pullback square of $R$-modules
\begin{equation}\label{plsquare}
\xymatrix{ \Sq(i,p) \ar[d] \ar[r] \ar@{}[dr] & \sM_R(W,E) \ar[d]^{p_*} \\ \sM_R(X,B) \ar[r]_{i^*} & \sM_R(W,B).}\\
\end{equation}
The underlying set of the $R$-module $\Sq(i,p)$ is the set of commutative squares
\begin{equation}\label{ptsquare}
 \xymatrix{
W\ar[d]_{i}  \ar[r] & E \ar[d]^{p} \\
X \ar[r] & B \\} 
\end{equation}
of maps of DG $R$-modules. The unlabeled maps in \eqref{plsquare} pick out the unlabeled maps 
in (\ref{ptsquare}).  The maps $p_*$ and $i^*$ induce a map of $R$-modules
\begin{equation}\label{epsilon-defn} \epz \colon \sM_R(X,E) \rtarr \Sq(i,p).\end{equation}

\begin{defn}\mylabel{enrichedLP} The map $i$ has the \emph{enriched left lifting property} against $p$, or 
equivalently the map $p$ has the \emph{enriched right lifting property} against $i$, written 
$i \underline{\boxslash} p$, if $\epz \colon \sM_R(X,E) \rtarr \Sq(i,p)$ is a split epimorphism  
of $R$-modules.  That is, $i \underline{\boxslash} p$ if there is an $R$-map $\et\colon \Sq(i,p)\rtarr \sM_R(X,E)$
such that $\epz \et = \id$.
\end{defn} 

\begin{lem} If $i$ has the enriched LLP against $p$, then $i$ has the usual unenriched LLP
against $p$.
\end{lem}
\begin{proof} If $\epz\et = \id$, then $\et$ applied 
to the element of $\Sq(j,f)$ displayed in \eqref{ptsquare} is a lift 
$X\rtarr E$ in that square.
\end{proof}

The notion of an enriched WFS is obtained by replacing lifting properties by
enriched lifting properties in the definition of the former; see \myref{ewfs}. It is easy to verify 
from the lemma that an enriched WFS is also an ordinary WFS. 
In particular, a model structure can be specified using a pair of enriched WFSs. 

Our interest in enriched lifting properties is not academic: we will shortly characterize the $r$-fibrations and 
$r$-acyclic $r$-fibrations as those maps satisfying enriched RLPs.  These characterizations will later be used to construct appropriate factorizations for the $r$-model structures on $\sM_A$.  

The proofs employ a procedure called the \emph{enriched} SOA.  As in our work in this paper, it can be used in situations to which the ordinary SOA does not apply.  Just as the classical SOA gives a uniform method for constructing compactly (or cofibrantly) generated WFSs, so the enriched SOA gives a uniform method for constructing compactly (or cofibrantly) generated enriched WFSs.
To avoid interrupting the flow and to collect material of independent interest in model category theory in one place, we defer technical discussion of the enriched SOA and related variant forms of the  SOA to \S\ref{appendix}, but we 
emphasize that the material there is essential to several later proofs. 
The following trivial example may help fix ideas.

\begin{ex}\mylabel{trivial} Consider $j \colon 0 \rtarr R$ and $p \colon E \rtarr B$ in $\sV_R$. Since $\sM_R(0,B) = 0$ and 
$\sM_R(R,B) \iso B$, $\Sq(j,p) \cong B$ and $j \underline{\boxslash} p$ if and only if $p \colon E \rtarr B$ is a split epimorphism. 
For the moment, write $\cJ$ for the singleton set $\{0 \rtarr R\}$.
Via the unenriched SOA, $\cJ$ generates a WFS on 
$\sV_R$ whose right class consists of the epimorphisms and whose left class consists of the monomorphisms with projective cokernel. Garner's variant of Quillen's SOA factors a map $X \rtarr Y$ in $\sV_R$ more economically through the direct sum  $X \oplus (\oplus_Y R)$ of $X$ with the free 
$R$-module on the underlying set of $Y$. Via the enriched SOA, $\cJ$ generates an enriched WFS on $\sV_R$ whose right class consists of the  $R$-split epimorphisms and whose left class consists of the monomorphisms. 
The enriched version of Garner's SOA (which is the enriched version we focus on) factors a map $X \rtarr Y$ as 
$X \rtarr X \oplus Y \rtarr Y$.
\end{ex}

\subsection{Enriching the $r$-model structure}\label{subsec:enrichedr}

With enriched WFSs at our disposal, we turn to the proof of statements
about the $r$-model structure on $\sM_R$ in \myref{enriched-comparison}.
We first expand \myref{trivial}. Recall from \myref{fibepi} that the set $\cJ_R$ generates a WFS
 on $\sM_R$ whose right class consists of the degreewise epimorphisms. 

\begin{ex}\mylabel{cofgenrtrivialcofibrations} Consider $j_n \colon 0 \rtarr D^n_R$ and a map $p \colon E \rtarr B$ in 
$\sM_R$. Since $\sM_R(0,B) = 0$ and $\sM_R(D^n_R, B) \cong B_n$, 
$$\epz \colon \sM_R(D^n,E) \rtarr \Sq(j_n,p)$$ is isomorphic to 
$$p_n \colon E_n \rtarr B_n.$$ 
Thus $j_n \underline{\boxslash} p$ if and only if $p_n$ is an $R$-split epimorphism. If this
holds for all $n$, then $\cJ_R \underline{\boxslash} p$. That is, $p$ has the enriched RLP against each map in $\cJ_R$ if and only if 
$p$ is an $R$-split epimorphism, which means that $p$ is an $r$-fibration. Since an enriched WFS, like an ordinary one, is determined by its right class, we conclude that the enriched WFS generated by $\cJ_R$ is the ($r$-acyclic $r$-cofibration, $r$-fibration) WFS.
\end{ex}

\begin{rmk}\mylabel{efficient-factorization}
The factorization produced by the enriched SOA 
applied to $\cJ_R$ is the precise algebraic analogue of the standard 
topological mapping {\em cocylinder} construction, as 
specified in \myref{mapping-factorization} and \eqref{MfNffact}. See \cite[\S 13.2]{Riehl}.
\end{rmk} 

\begin{ex} Consider $i_n \colon S^{n-1}_R \rtarr D^n_R$ and $p \colon E \rtarr B$ in $\sM_R$. We have a natural isomorphism $\sM_R(S^{n-1}_R,B) \cong Z_{n-1}B$ since a DG $R$-map $S^{n-1}_R \rtarr B$ is specified by an $(n-1)$-cycle in $B$.  It follows that
$$\epz \colon \sM_R(D^n,E) \rtarr \Sq(i_n,p)$$ 
is isomorphic to 
$$ (p_n,d)\colon E_n \rtarr B_n \times_{Z_{n-1}B} Z_{n-1}E.$$ 
By definition, $i_n\underline{\boxslash} p$ if and only if this map of $R$-modules has a section $\et_n$.
\end{ex}

It turns out that the enriched right lifting property against $\cI_R$ characterizes 
the $r$-acyclic $r$-fibrations. This is analogous to \myref{cofgenrtrivialcofibrations}, but less obvious.

\begin{lem}\mylabel{cofgenrcofibrations} A map $p\colon E\rtarr B$ in $\sM_R$ satisfies the enriched 
RLP against $\cI_R$ if and only if $p$ is an $r$-acyclic $r$-fibration.
\end{lem}
\begin{proof} Recall that the $r$-acyclic $r$-fibrations are exactly the $h$-acyclic $h$-fibrations. 
By \cite[Corollary 18.2.7]{MP}, $p$ is an $h$-acyclic $h$-fibration if and only if $p$ is isomorphic to the projection map $B \oplus C \rtarr B$ where $C \cong \ker p$ is contractible. Suppose given such a map and let maps  $s_n \colon C_n \rtarr C_{n+1}$ give a contracting homotopy, so that $ds + sd = \mathrm{id}_C$. The pullback $B_n \times_{Z_{n-1}B} (Z_{n-1}B \oplus Z_{n-1}C)$ is isomorphic to $B_n \oplus Z_{n-1}C$. We can define a section of the map $B_n \oplus C_n \rtarr B_n \oplus Z_{n-1}C$ by sending $(b,c)$ to $(b, s(c))$; here $c = ds(c) +sd(c) = ds(c)$ since $c$ is a boundary. This shows that the $r$-acyclic $r$-fibrations satisfy the enriched RLP against $\cI_R$.

Conversely, suppose that $p$ has the enriched RLP. Identify $Z_nB$ with the submodule 
$Z_nB \times \{0\}$ of the pullback $B_n \times_{Z_{n-1}B}Z_{n-1}E$. Restriction of the postulated 
section $\et_n$ gives a section $\et_n\colon Z_nB \rtarr Z_nE$ of $p_n\vert_{Z_nE}$. Define 
$\si_n \colon B_n \rtarr E_n$  by
\[ \si_n(b) = \et_n(b, \et_{n-1}d(b)). \]
Since $\epz = (p_n,d)$, $\epz \et_n = \id$, and $d^2=0$, we see that $p_n\si_n(b) = b$ and 
\[ d\si_n(b)= \pi_2\epz \et_n(b,\et_{n-1}d(b)) = \et_{n-1}d(b) =  \si_{n-1}d(b). \]
Therefore $\si$ is a section of $p_n$ and a map of DG $R$-modules. 

The section $\si$ and the inclusion $\ker p\subset E$ define a chain map $B \oplus \ker p \rtarr E$ over $B$. We claim that it is an isomorphism. It is injective since if $(b,c)\in B \oplus \ker p$ maps to zero then $\si(b)+c=0$, hence $b=p\si(b)+p(c)=0$, and thus $c=-\si(b) = 0$. It is surjective since it sends
$(p(c), c-\si p(c))$ to $c$.

It remains to show that $\ker p$ is $h$-acyclic.  We define a contracting homotopy $s$ on $\ker p$
by letting $s_n \colon \ker p_n \rtarr \ker p_{n+1}$ send $c$ to $\et_{n+1}(0,c - \et_n(0,d(c)))$. Then
\begin{align*} (d s_n + s_{n-1} d)(c) &= d\et_{n+1}(0,c - \et_n(0,d(c))) + 
\et_n(0, d(c) - \et_{n-1}(0,d^2(c)))  \\ &= c- \et_n(0,d(c)) + \et_n(0,d(c) - \et_{n-1}(0,0)) \\ 
&= c - \et_n(0,d(c)) + \et_n (0,d(c)) = c. \qedhere \end{align*} 
\end{proof}

\begin{rmk}\mylabel{efficient-factorization-2}
The factorization produced by the enriched SOA 
applied to $\cI_R$ is the precise algebraic analogue of the standard 
topological mapping {\em cylinder} construction, as 
specified in \myref{mapping-factorization} and \eqref{MfNffact}. See \cite[\S 13.4]{Riehl}.
This observation and the less surprising \myref{efficient-factorization} illustrate some advantages of the variant forms of the SOA we promote in \S\ref{appendix}. Like Quillen's SOA, these are a priori infinite constructions; however in practice, they may converge much sooner.
\end{rmk}

\myref{enriched-comparison} is immediate from \myref{cofgenrcofibrations} and \myref{cofgenrtrivialcofibrations}:
The $r$-model structure was established in \myref{hmodelR} and the cited results
show that its two constituent WFSs are generated in the enriched sense by the  sets $\cI_R$ and $\cJ_R$. 

For commutative rings $R$, we now have a structural understanding of the $r$-cofibrant and $r$-acyclic and $r$-cofibrant
objects that was invisible to our original proof of the model structure. It is a special case of \myref{enrichedboxslashclosure} below.

\begin{cor}\mylabel{retractr}  A DG $R$-module is $r$-cofibrant or $r$-acyclic and $r$-cofibrant if and only if it is a 
retract of an enriched $\cI_R$-cell complex or an enriched $\cJ_R$-cell complex.
\end{cor}
 
\section{The $q$- and $h$-model structures on the category $\sM_A$}\label{DGAch}

Now return to the introductory context of a commutative ring $R$ and a 
DG $R$-algebra $A$.  If we forget the differential and the 
$R$-module structure on $A$, then \S\ref{chainch} (applied to modules over graded rings) gives the category 
of left $A$-modules $q$- and $h$-model structures.  The fact that $A$ is an $R$-algebra 
is invisible to these model structures. Similarly, as we explain in this 
section, we can forget the $R$-module structure or, equivalently, let $R = \bZ$, 
and give the category $\sM_A$ of (left) DG $A$-modules $q$-, $h$-, and therefore $(q,h)$-model 
structures. 
Most of the proofs are similar or identical to those given in \cite{MP} 
for the parallel results in \S\ref{chainch},  and we indicate points of
difference and alternative arguments. The main exception is the verification of the factorization axioms for the $h$-model structure, which requires an algebraic generalization of the small object argument discussed in \S\ref{ASOA}. 

\subsection{Preliminaries and the adjunction $\bF \dashv \bU$} 
Remember that $\otimes$ and $\Hom$ mean $\otimes_R$ and $\Hom_R$. The category $\sM_A$ is bicomplete; its 
limits and colimits are limits and colimits in $\sM_R$ with the induced actions of $A$.  It is also 
enriched, tensored, and cotensored over the cosmos $\sM_R$.  The internal hom objects  are 
the DG $R$-modules $\Hom_A(X,Y)$, where $\Hom_A(X,Y)$ is the subcomplex of 
$\Hom(X,Y)$ consisting of those maps $f$ that commute with the action of $A$. Precisely, remembering 
signs, for a map $f\colon X\rtarr Y$ of degree $n$ with components $f_i\colon X_i\rtarr Y_{i+n}$,
$ f(ax) = (-1)^{n\text{deg}(a)}a f(x)$.\footnote{As usual, we are invoking the rule of 
signs which says that whenever two things with a degree are permuted, the appropriate sign should be 
introduced.}
For a DG $A$-module $X$ and a DG $R$-module $K$, 
the tensor $X\otimes K$ and cotensor $X^K = \Hom(K,X)$ are the evident DG 
$R$-modules with left $A$-actions given by $a(x\otimes k) = (ax)\otimes k$ and 
$(af)(k) = (-1)^{\deg(a)\deg(f)}f(ak)$. 
We have the adjunctions
\begin{equation}\label{2variable}
 \Hom_A(X\otimes K,Y) \iso \Hom(K,\Hom_A(X,Y))\iso \Hom_A(X,\Hom(K,Y)).
\end{equation}

If $A$ is commutative, where of course the graded sense of commutativity is
understood, then $\sM_A$ is a cosmos; the tensor product $X\otimes_A Y$ and
internal hom $\Hom_A(X,Y)$ inherit $A$-module structures from $X$ or, equivalently, $Y$.

Define the extension of scalars functor $\bF\colon \sM_R\rtarr \sM_A$ by $\bF X = A\otimes X$. 
It is left adjoint to the underlying DG $R$-module functor $\bU\colon \sM_A\rtarr \sM_R$.  The action 
maps $A\otimes X\rtarr X$ of $A$-modules $X$ give the counit $\al$ of the adjunction. The unit of $A$ 
induces maps $K = R\otimes K \rtarr A\otimes K$ of DG $R$-modules that give the unit $\io$ of the adjunction.
Categorically, a DG $R$-algebra $A$  is a monoid in the symmetric monoidal category $\sM_R$, and a DG $A$-module 
is the same structure as an algebra over the monad $\bU\bF$ associated to the monoid $A$.  That is, the adjunction 
is monadic.  

Logically, we have two adjunctions $\bF \dashv \bU$ in sight, one between graded $R$-modules and graded 
$A$-modules and the other between DG $R$-modules and DG $A$-modules, but we shall only use the latter here. We briefly use the former in \S\ref{relproj}, where we discuss the sense in which $\bF$
should be thought of as a ``free $A$-module functor''.  Unless $X$ is free as an $R$-module, 
$\bF X$ will not be free as an $A$-module.  In general, $\bF X$ is free in a relative sense 
that we make precise there.  We use $\bF$ to construct our model structures on 
$\sM_A$, but when developing the $q$-model structure we only apply it to free $R$-modules.

\subsection{The $q$-model structure}\label{SecqmodelA}  Again, this is the model structure in common use.
We can construct it directly, without reference to $\sM_R$, or we can use a standard
argument recalled in \myref{cofcrit} to lift the $q$-model structure from $\sM_R$ to $\sM_A$.  
We summarize the latter approach because its enriched variant will appear when we transfer the 
$r$-model structure from $\sM_R$ to $\sM_A$ in \S\ref{sec:rmodel}. 
Thus define the $q$-model structure on $\sM_A$ by requiring $\bU$ to create the
weak equivalences and fibrations from the $q$-model structure on $\sM_R$.
Recall \myref{qRmod}.

\begin{defn} Define $\bF\cI_R$ and $\bF\cJ_R$ to be the sets of maps
in $\sM_A$ obtained by applying $\bF$ to the sets of maps $\cI_R$ and $\cJ_R$
in $\sM_R$. Define $\sW_q$ and $\sF_q$ to be the subcategories of maps $f$ in $\sM_A$ such 
that $\bU f$ is in $\sW_q$ or $\sF_q$ in $\sM_R$; that is, $f$ is a quasi-isomorphism
or surjection.  Define $\sC_q$ to be the subcategory of maps that have the LLP with
respect to $\sF_q\cap \sW_q$.
\end{defn}

\begin{thm}\mylabel{qmodelA}  
The subcategories $(\sW_q,\sC_{q},\sF_{q})$ 
define a compactly generated model category structure on $\sM_A$ 
called the $q$-model structure.  The sets $\bF\cI_R$ and $\bF\cJ_R$ 
are generating sets for the $q$-cofibrations and the $q$-acyclic $q$-cofibrations.
Every object is $q$-fibrant and the $q$-model structure is proper.  
If $A$ is commutative, the cosmos $\sM_A$ is a monoidal model category under $\otimes_A$. In general,
$\sM_A$ is an $\sM_R$-model category, and $\bF\dashv\bU$ is a Quillen adjunction
between the $q$-model structures on $\sM_A$ and $\sM_R$.  In particular, $\bF$ preserves $q$-cofibrations and $q$-acyclic $q$-cofibrations.
\end{thm}
\begin{proof}  We refer to \myref{cofcrit}.
The sets of maps $\bF\cI_R$ and $\bF\cJ_R$ are compact in $\sM_A$ 
since their domains are free $A$-modules on $0$ or $1$ generator. Acyclicity follows from the proof of \myref{FJacyclicity} below: the relative $\bF\cJ_R$-cell complexes are contained in the enriched relative $\bF\cJ_R$-cell complexes, and the argument given there shows that these are $r$-equivalences, and hence $q$-equivalences.
Properness is proven in the same way as for $\sM_R$ in \cite[\S18.5]{MP}.

For $X, Y \in \sM_R$, the associativity isomorphism $(A \otimes X) \otimes Y\iso A\otimes (X\otimes Y)$ shows that $\bF$ 
preserves cotensors by $\sM_R$. Therefore \myref{cofcrit} implies that the $q$-model structure makes $\sM_A$ an $\sM_R$-model 
category. When $A$ is commutative,  $(A \otimes X) \otimes_A (A \otimes Y) \cong A \otimes (X \otimes Y)$ so that $\bF$ is 
a monoidal functor. Since the unit $A$ for $\otimes_A$ is cofibrant, it follows that the $q$-model structure on $\sM_A$ is monoidal.
\end{proof}

\subsection{The $h$-model structure}\label{SechmodelA}

The basic definitions are the same as for the $h$-model structure on $\sM_R$. 
We write $I$ for the  DG $R$-module $R \otimes I$ in this section.

\begin{defn}\mylabel{coffibalgtoo}  Just as in \myref{coffibalg}, an \emph{$h$-cofibration} 
is a map in $\sM_A$ that satisfies the homotopy extension property (HEP) and an
\emph{$h$-fibration} is a map that satisfies the 
covering homotopy property (CHP).  Let $\sC_h$ and $\sF_h$ denote the subcategories of 
$h$-cofibrations and $h$-fibrations. An \emph{$h$-equivalence} is a homotopy equivalence of
DG $A$-modules, and $\sW_h$ denotes the subcategory of $h$-equivalences.  
\end{defn}  

\begin{thm}\mylabel{hmodelA}  
The subcategories $(\sW_h,\sC_h,\sF_{h})$ 
define a model category structure on $\sM_A$ called the $h$-model 
structure.  The identity functor is a Quillen right adjoint from the 
$h$-model structure to the $q$-model structure. Every object is $h$-cofibrant 
and $h$-fibrant, hence
the $h$-model structure is proper. If $A$ is commutative, then 
$\sM_A$ is a monoidal model category.  In general, $\sM_A$ is an 
$\sM_R$-model category.
\end{thm}

The starting point of the proof, up through the verification of the factorization
axioms, is the same as the starting point in the special case $A=R$, and the proofs 
in \cite[\S18.2]{MP} of the following series of results work in precisely the same fashion.

Suppose we have DG $A$-modules $X$ and $Y$ under a DG $A$-module $W$, with given maps 
$i\colon W\rtarr X$ and $j\colon W\rtarr Y$.  Two maps  $f,g\colon X\rtarr Y$ under $W$ are
homotopic under $W$ if there is a homotopy $h\colon X\otimes I\rtarr Y$ between them such 
that $h(i(w)\otimes [I])=0$ for $w\in W$.  A cofiber homotopy equivalence
is a homotopy equivalence under $W$.  The notion of fiber homotopy equivalence is 
defined dually.

\begin{lem}\mylabel{CofEquiv} Let $i\colon W\rtarr X$ and $j\colon W\rtarr Y$ be $h$-cofibrations 
and $f\colon X\rtarr Y$ be a map under $W$.  If $f$
is a homotopy equivalence, then $f$ is a cofiber homotopy equivalence.
\end{lem}

\begin{prop}\mylabel{DefRetCor1}  A map $i\colon W\rtarr Y$ is an $h$-acyclic $h$-cofibration if and only if $i$ is a monomorphism, $Y/W$ is contractible, and $i$ is isomorphic under $W$ to the 
inclusion $W\rtarr W\oplus Y/W$.
\end{prop}

\begin{lem}\mylabel{FibEquiv} Let $p\colon E\rtarr B$ and $q \colon F\rtarr B$ be 
$h$-fibrations and $f\colon E\rtarr F$ be a map over $B$, so that $q f = p$. 
If $f$ is a homotopy equivalence, then $f$ is a fiber homotopy equivalence.
\end{lem}

\begin{prop}\mylabel{DefRetCor2}  A map $p\colon E\rtarr B$ is an $h$-acyclic $h$-fibration if and only if
$p$ is an epimorphism, $\ker(p)$ is contractible, and $p$ is isomorphic over $B$ to the 
projection $B\oplus \ker(p)\rtarr B$. 
\end{prop}

\begin{lem}\mylabel{Splitting0} Let
\[\xymatrix@1{ 0 \ar[r] & X\ar[r]^-{f} & Y \ar[r]^-{g} & Z \ar[r] & 0\\} \]
be an exact sequence of $A$-chain complexes whose underlying exact sequence of 
$A$-modules, with differentials ignored, splits.  Then $f$ is an $h$-equivalence 
if and only if $Z$ is contractible and $g$ is an $h$-equivalence if and only if $X$ 
is contractible.  
\end{lem}

\begin{prop}\mylabel{liftrhtoo} Consider a commutative diagram of $A$-chain complexes
\[ \xymatrix{
W\ar[r]^-g \ar[d]_i & E \ar[d]^p\\
X \ar@{-->}[ur]^{\la} \ar[r]_-f & B\\} \]
in which $i$ is an $h$-cofibration and $p$ is an $h$-fibration.
If either $i$ or $p$ is an $h$-equivalence, then there exists a lift $\la$.
\end{prop}

Whereas the proofs of the results above are the same as in \cite[\S 18.2]{MP},  the proofs there of the factorization axioms do not generalize. 
We no longer have an identity of $r$- and $h$-model structures 
since we no longer have an analogue of \myref{Rsplitrh} when $A$ has non-zero
differential.  Therefore, we no longer have simple explicit descriptions
of the $h$-fibrations and $h$-cofibrations.  To begin with, we mimic a standard argument in 
topology.

\begin{defn}\mylabel{mapping-factorization}  Let $f\colon X\rtarr Y$ be a map of DG $A$-modules.  Define 
the \emph{mapping cylinder} $Mf$ to be the pushout $Y\cup_f (X\otimes I)$ of the diagram
\[ \xymatrix@1{   Y & X \ar[l]_-{f} \ar[r]^-{i_0} & X\otimes I.\\} \]
Define the \emph{mapping cocylinder} $Nf$ to be the pullback $X\times_f Y^I$ of the 
diagram
\[ \xymatrix@1{   X \ar[r]^-{f} & Y & Y^I  \ar[l]_-{p_0} .\\} \]
\end{defn}

Just as in topology, we have the following naive factorization results.

\begin{lem}\mylabel{naivefactorization}
Any map $f: X \rtarr Y$ in $\sM_A$ factors as composites
\begin{equation}\label{MfNffact} 
\xymatrix@1{X\ar[r]^-{j} & Mf \ar[r]^-{r} & Y}
\ \ \ \text{and} \ \ \ \xymatrix@1{X\ar[r]^-{\nu} & Nf \ar[r]^{\rh} & Y \\} \
\end{equation}
where $r$ and $\nu$ are $h$-equivalences, $j$ is an $h$-cofibration, and $\rh$ 
is an $h$-fibration.
\end{lem}

\begin{rmk} When $A=R$, a quick inspection shows that $j$ and $\nu$ are $R$-split monomorphisms 
and $r$ and $\rho$ are $R$-split epimorphisms.  Therefore these factorizations 
are model theoretic factorizations, completing one proof of \myref{hmodelR}.  
\end{rmk}

\begin{proof} Since the topological proofs of the equivalences do not transcribe directly to algebra, we indicate a quick proof that $r$ is an $h$-equivalence. Here $j(x) = x\otimes [1]$, $r(y)=y$,
$r(x\otimes [1])=f(x)$, and $r(x\otimes [I]) =0$. Define
$i\colon Y\rtarr Mf$ by $i(y)=y$. Then $ri = \text{id}_Y$.  A homotopy 
$h\colon Mf\otimes I\rtarr Mf$ from $ir$ to 
$\text{id}_{Mf}$ is given by 
\[  h(z\otimes [I]) = \left \{  \begin{array}{ll} 
0 & \mbox{if $z\in Y$ (or $z=x \otimes [0]$)} \\
x\otimes [I] & \mbox{if $z = x\otimes [1]$}\\
0 & \mbox{if $z = x\otimes [I]$.} \end{array} \right.  \]
A small check, taking care with signs, shows that this works.  The definitions
of $\nu$ and $\rh$ are dual to those of $i$ and $r$, and a dual proof shows that
$\nu$ is an $h$-equivalence.  

We next prove that $j$ is an $h$-cofibration.  We can factor $j$ as the bottom composite in the 
diagram 
\[ \xymatrix{
& X\oplus X \ar[r]^-{i_0+i_1} \ar[d]_{f\oplus\text{id}}  & X\otimes I \ar[d] \\
X \ar[r]_-{(0,\text{id})} & Y\oplus X \ar[r]_-{i+i_1} & Mf,\\} \]
in which the square is a pushout.  Since a pushout of an $h$-cofibration is 
an $h$-cofibration, $j$ is an $h$-cofibration if $X\otimes I$ is a good 
cylinder object, that is, if the natural map $i_0+i_1\colon X \oplus X \rtarr X \otimes I$ 
is an $h$-cofibration. Recall from \myref{DefRetCor2} that a map 
$p\colon E \rtarr B$ is an $h$-acyclic $h$-fibration if and only if  $E$ is isomorphic
over $B$ to the projection 
$B \oplus C \rtarr B$ where $C=\mathrm{ker}(p)$ contractible. Therefore, we are reduced to 
showing that $i_0+i_1$ has the left lifting problem against $\mathrm{id}_B\colon B \rtarr B$ and against $C \rtarr 0$ for all $B$ and 
all contractible $C$.  The first part is obvious.  For the second, let $(u_0,u_1)\colon X \oplus X \rtarr C$ be a given map. If $h\colon C \otimes I \rtarr C$ denotes a homotopy between $\mathrm{id}_C$ and $0$ and $h'$ is $h$ with reversed direction, we construct the desired extension as $u=h(u_0 \otimes I) + h'(u_1 \otimes I)$,
\[ \xymatrix{X \oplus X \ar[d]_{(i_0, i_1)} \ar[r]^-{(u_0, u_1)} & C \\
X \otimes I. \ar[ru]_{u}}\]
The dual argument works to show that $Y^I$ is a good cocylinder object and therefore
$\rh$ is an $h$-fibration. 
\end{proof}

\begin{rmk}\mylabel{rmk:naive}
Although $j\colon X \rtarr Mf$ is an $h$-cofibration in $\sM_A$, it is not generally an 
$r$-cofibration (or $q$-cofibration). Indeed, take $X =0$ and let $Y$ be an object in $\sM_A$ 
that is not $r$-cofibrant. Then the map $0 \rtarr Mf = Y$ is not an $r$-cofibration.
\end{rmk}

Unfortunately, in \myref{naivefactorization} there is no reason to expect $r$ to be an 
$h$-fibration or $\nu$ to be an $h$-cofibration.  We therefore give an entirely different 
proof of the factorization axioms.  The idea is to iterate the construction 
of the mapping cocylinder, but the details are more subtle than one might expect. The same issues 
arose in the topological context and the solution is identical to the one given there in \cite{BR}. 
 We prove the following result in \S\ref{ASOA} where we discuss the algebraic SOA. 

\begin{prop}\mylabel{hfactor}
Any map $f: X \rtarr Y$ factors as the composite of an $h$-acyclic $h$-cofibration and an $h$-fibration.
\end{prop}

\begin{cor}  Any map $f$ also factors as the composite of an $h$-cofibration and an $h$-acyclic $h$-fibration.  
\end{cor}
\begin{proof}
We obtain the factorization from $\xymatrix@1{X  \ar[r]^-{j} & Mf \ar[r]^-{r} & Y}$ by factoring $r$ 
into an $h$-acyclic $h$-cofibration followed by an $h$-fibration.
\end{proof}

We have completed the proof that $\sM_A$ is a model category.
Since $h$-acyclic $h$-cofibrations and $h$-acyclic $h$-fibrations are inclusions and 
projections of deformation retractions, by Propositions \ref{CofEquiv} and \ref{FibEquiv}, 
every object is both $h$-cofibrant and $h$-fibrant, hence the model structure 
is proper. The proofs that $\sM_A$ is an $\sM_R$-model category and that 
$\sM_A$ is monoidal when $A$ is commutative are the same as in the case $A=R$ given in
\cite[p.~383]{MP}. 

\section{The $r$-model structure on $\sM_A$}\label{sec:rmodel} 

\subsection{Relatively projective $A$-modules}\label{relproj}
Forget the differentials for a moment and consider a graded $R$-algebra $A$.  Classically \cite{EM, MacHom},
absolute homological algebra considers exact sequences of (graded) $A$-modules, which of course 
are just sequences of $A$-modules whose underlying sequences of (graded) $R$-modules are exact in each degree.  
For us, relative homological algebra considers exact sequences of $A$-modules whose 
underlying sequences of $R$-modules are split exact, which means that they are degreewise split exact.  The two 
notions agree when $R$ is semi-simple.  

Graded $A$-modules of the form $\bF V = A\otimes V$, for a graded $R$-module $V$, are said to be 
{\em relatively free}.  They need not be free in the usual sense of having a basis, and they need not
be free or even projective as $R$-modules.  A direct summand (as an $A$-module) of such an $A$-module 
is said to be {\em relatively projective}.  The term is justified by the following result, which is peripheral to our work but relates it to a classical context.  It uses the 
notion of a projective class, which is the classical starting point of relative homological algebra. 
Projective classes axiomatize  the relationship between the projective objects and the epimorphisms in an abelian category. The full definition would be digressive here, but we recall it in \myref{projclass} below, 
to which we refer in the proof. 

\begin{lem}\mylabel{proj}  Let $\sP$ be the class of relatively projective $A$-modules (not DG $A$-modules) and 
let $\sE$ be the class of $R$-split epimorphisms of $A$-modules.  Then $(\sE,\sP)$ is a projective class
in the category of $A$-modules.
\end{lem}
\begin{proof} Let $P$ be a relatively projective $A$-module and 
$p\colon E\rtarr M$ be an $R$-split epimorphism of $A$-modules.  Then for any 
map of $A$-modules $f\colon P\rtarr M$, there is a map $\tilde{f}\colon  P\rtarr E$ 
of $A$-modules such that $p \tilde f = f$.  
\[ \xymatrix{
& P \ar[d]^f \ar@{-->}[dl]_{\tilde f} & \\
E \ar[r]_-{p} & M \ar[r] & 0 \\} \]
To see that, choose a map $j\colon M\rtarr E$ of underlying graded $R$-modules,
such that $pj = \text{id}$.  Let $i\colon P\rtarr A\otimes K$ and 
$r\colon A\otimes K\rtarr P$ be maps of $A$-modules such that $ri=\text{id}$.  
The composite of $j$ and the restriction of $fr$ to $K$ gives a map of graded 
$R$-modules $K\rtarr E$.  Its adjoint $\bar{f}\colon A\otimes K\rtarr E$ satisfies 
$p\bar f = fr$, hence the composite $\tilde{f} = \bar{f}i$ satisfies $p\tilde{f} = f$. 

We must still verify (i)-(iii) of \myref{projclass}. 
For (i), we must show that if $p\colon E\rtarr M$ is a map of $A$-modules such that the lifting 
property above holds for all $P\in \sP$, then $p\in \sE$.  The hypothesis gives that
the action map $f\colon A\otimes M\rtarr M$ of $A$-modules lifts to a map 
$\tilde f\colon A\otimes M\rtarr E$ of $A$-modules.   Its restriction to $M$ gives a map 
$s\colon M\rtarr E$ of $R$-modules such that $ps = \id$,
hence $p$ is an $R$-split epimorphism.  For (iii), we must show that for every $M$ there is a map 
$p\colon P\rtarr M$ of $A$-modules such that $P\in \sP$ and $p\in \sE$.  Since the action map $f$ is an $R$-split
epimorphism, it is such a map. For (ii), we must show that if $P$ is such that the lifting property
of the first paragraph holds for all $p\in \sE$, then $P\in \sP$.  As the action map 
$f\colon A\otimes P\rtarr P$ is in $\sE$, the hypothesis gives that it is a split surjection
of $A$-modules, so this is clear.
\end{proof} 

Returning to our DG context, we say that a DG $A$-module $P$ is relatively free or relatively projective 
if its underlying $A$-module is so. 
If $\sP$ denotes the class of relatively projective DG $A$-modules, then the corresponding class 
$\sE$ of maps that are $P$-surjective for all $P \in \sP$ has another name: it is the class of 
$r$-fibrations in the $r$-model category we construct next. However, it is not true 
in general that $(\sP,\sE)$ is a projective class in $\sM_A$.

It is natural to ask if there is a useful projective class $(\sP,\sE)$ in the
category of DG $A$-modules itself, and we shall show that there is in \S\ref{projclassA}. 

\begin{rmk} As already noted, projective classes in abelian categories are the starting point of the general
subject of relative homological algebra.  It was developed classically by Eilenberg, Mac\,Lane,
and Moore \cite{EM, MacHom, Moore} and model theoretically by Christensen and Hovey \cite{CH},
whose work has influenced ours. However, it does not apply to give the model structures 
we develop here; see \myref{why-enriched}.
\end{rmk}

\subsection{Construction of the $r$-model structure}\label{subsec:rforDGA}

Returning to our model theoretic work, recall that $\sW_r$ denotes the category of $r$-equivalences,
namely the maps that are homotopy equivalences of underlying DG $R$-modules. 

\begin{defn}\mylabel{rhdefns} A map $f$ of DG $A$-modules is an \emph{$r$-fibration} 
if it is an $R$-split epimorphism, that is, if $\bU f$ is an $r$-fibration.  A 
map is an \emph{$r$-cofibration} if it satisfies the LLP against the 
$r$-acyclic $r$-fibrations.  Let $\sC_r$ and $\sF_{r}$ denote
the classes of $r$-cofibrations and $r$-fibrations.
\end{defn}

By definition, the right adjoint $\bU$ creates the $r$-equivalences and $r$-fibrations in $\sM_A$ from the 
$r$-equivalences and $r$-fibrations in $\sM_R$.  Since the adjunction $\bF \dashv \bU$ is enriched over $R$-modules,  the $r$-fibrations in $\sM_A$ are exactly 
the maps that have the enriched RLP against $\bF \cJ_R$ and the $r$-acyclic $r$-fibrations are 
exactly the maps that have the enriched RLP against $\bF \cI_R$ \cite[\S 13.3]{Riehl}. 
As observed in the proof of \myref{qmodelA}, these sets of maps are compact in $\sM_A$, so we can construct factorizations using the enriched SOA described in \S\ref{sec:SOA2}.  Therefore, by \myref{ecofcrit}, which is the enriched analogue of the standard result for lifting model structures along adjunctions recalled in \myref{cofcrit}, to prove that these classes define a model structure on $\sM_A$, 
it suffices to prove the following acyclicity condition. As spelled out in detail 
in \myref{cellob2}, an enriched relative $\bF\cJ_R$-cell complex is a composite of pushouts of 
coproducts of tensors of maps in $\bF\cJ_R$ with $R$-modules.

\begin{prop}\mylabel{FJacyclicity}
Enriched relative $\bF\cJ_R$-cell complexes are $r$-equivalences.
\end{prop}
\begin{proof}  The adjunction $\bF \dashv \bU$ is monadic and the monad $A \otimes -$ preserves colimits 
in $\sM_R$ since $\sM_R$ is closed monoidal. It follows 
that the forgetful functor $\bU \colon \sM_A \rtarr \sM_R$ creates and therefore preserves both limits and colimits \cite[4.3.2]{BorII} and also tensors with $R$-modules. Because $\bU$ also creates the $r$-equivalences, it suffices to show that enriched relative $A \otimes \cJ_R$-cell complexes are $r$-equivalences in $\sM_R$. By \myref{enriched-comparison}, it suffices to show that $0 \to A \otimes D^n_R$ is an $r$-acyclic $r$-cofibration. But this is clear: all objects in $\sM_R$ are $(r=h)$-cofibrant and tensoring with $A$ preserves the contracting homotopy that witnesses the $(r=h)$-acyclicity of $D^n_R$.
\end{proof}

By \myref{ecofcrit}, this implies the following analogue of \myref{enriched-comparison}. 

\begin{thm}\mylabel{enrichedDGAcomparison}
Let $A$ be a DG-algebra over a commutative ring $R$. Then $\bF\cI_R$ and $\bF\cJ_R$ are sets of generating cofibrations and acyclic cofibrations for the $q$-model structure on $\sM_A$, when compact generation is meant in the usual sense, and for the $r$-model 
structure on $\sM_A$, when compact generation is meant in the $\sV_R$-enriched sense.
\end{thm}

As with our previous model categories, we have the following elaboration.

\begin{thm}\mylabel{rmodelA-monoidal}  If $A$ is commutative, the $r$-model structure is monoidal.  In general,
the $r$-model structure makes $\sM_A$ into an $\sM_R$-model category with respect to the $r$-model structure on $\sM_R$,
and $\bF\dashv\bU$ is a Quillen adjunction between the $r$-model structures on $\sM_A$ and $\sM_R$.  In particular,  $\bF$ preserves $r$-cofibrations and $r$-acyclic $r$-cofibrations.
\end{thm}
\begin{proof}  By \myref{ecofcrit-monoidal}, this is a formal consequence of our characterization 
of the $r$-model structure on $\sM_A$ as a lift of the $\sV_R$-compactly generated $r$-model structure 
on $\sM_R$.  There is some delicacy to formulating the argument precisely since it involves the double
enrichment, over $\sM_R$ and $\sV_R$, of all categories in sight; details are given in \S\ref{EMC}. 
\end{proof}

\begin{rmk}\mylabel{why-enriched}  Parenthetically, \cite[3.4]{CH} claimed without proof that the $r$-model 
structure on $\sM_A$ exists. However, as noted in \cite[5.12]{CH}, the $r$-model structure on $\sM_R$ is 
usually not cofibrantly generated in the classical sense, and the arguments the authors had in mind cannot 
be applied, as they agree.\footnote{They say this in a nice postscript that they added to the arXived version of \cite{CH}. }  This emphasizes the importance of the enriched SOA: {\em we know of no other proof that $\sM_A$ has the factorizations necessary to construct the $r$-model structure}. 
\end{rmk}

\section{The six model structures on $\sM_A$}\label{6model} 

\subsection{Mixed model category structures in general}\label{mmodel}
We recall the following results of Cole \cite{Cole}; see also \cite[\S17.3]{MP}.
These sources give more detailed information than we will include here about 
mixed model structures in general. The notation in this section is generic: 
the pair $(q,h)$ can and will vary. 

\begin{thm}\mylabel{generalmixedstructure} Let $\sM$ be a model category with model structures 
$(\sW_h,\sC_h,\sF_h)$ and $(\sW_q,\sC_q,\sF_q)$ such that
$\sW_h\subset \sW_q$ and $\sF_h\subset \sF_q$ and therefore $\sC_q\subset \sC_h$.  
Then there is a mixed model structure, called the $(q,h)$-model structure,
\[ (\sW_q,\sC_{q,h},\sF_h).  \]
It satisfies the following properties.
\begin{enumerate}[(i)]
\item A map is a $(q,h)$-cofibration
if and only if it is an $h$-cofibration that factors as a 
composite $f i$, where $i$ is a $q$-cofibration and
$f$ is an $h$-equivalence.  
\item An object is 
$(q,h)$-cofibrant if and only if it is $h$-cofibrant
and has the $h$-homotopy type of a $q$-cofibrant object.
\item The identity functor on $\sM$ is a right Quillen equivalence from 
the $(q,h)$-model structure to the $q$-model structure, hence is 
a left Quillen equivalence from the $q$-model structure to the 
$(q,h)$-model structure.  
\item If $\sM$ is $q$-proper, then $\sM$ is $(q,h)$-proper. 
\item If $\sM$ is a cosmos that is monoidal in the $h$- and $q$-model structures, then $\sM$ is monoidal in the $(q,h)$-model structure.  
\item Under the hypotheses of (v), if $\sN$ is an $\sM$-bicomplete $\sM$-category 
that has an analogous pair of model structures such that $\sW_h\subset \sW_q$, 
$\sF_h\subset \sF_q$, and $\sN$ is an $\sM$-model category with respect to the 
$h$- and $q$-model structures, then $\sN$ is an $\sM$-model
category with respect to the $(q,h)$-model structures. 
\end{enumerate} 
\end{thm} 

Conceptually, the $(q,h)$-model structure is a resolvant (or colocalization) model structure.
The $(q,h)$-cofibrant, or resolvant, objects can be characterized as those $h$-cofibrant objects 
$C$ such that
\[ f_*\colon h\sM(C,Y) \rtarr  h\sM(C,Z) \]
is a bijection for all maps $f\colon Y\rtarr Z$ in $\sW_q$. A $(q,h)$-cofibrant approximation
$\GA X \rtarr X$ can be thought of as a resolution of $X$.  This includes approximations of spaces by 
CW complexes in topology and projective resolutions in algebra, but it allows such
approximations up to homotopy equivalence.  

\subsection{The mixed model structure on $\sM_R$}\label{1mixed}
In this section only, we drop the requirement that $R$ be commutative.
In $\sM_R$, every object is $h$-cofibrant and we have the following special case.

\begin{thm}\mylabel{mmodelR} $\sM_R$ has a proper $(q,h)$-model structure.  It is monoidal if $R$
is commutative and is an $\sM_{\bZ}$-model structure in general. It has the following properties.
\begin{enumerate}[(i)]  
\item The $(q,h)$-cofibrations are the $h$-cofibrations ($R$-split monomorphisms) that factor as composites of $q$-cofibrations and $h$-equivalences.
\item The $(q,h)$-cofibrant objects are the DG $R$-modules 
of the homotopy types of $q$-cofibrant DG $R$-modules; they are homotopy equivalent to 
degreewise projective $R$-modules, and the converse holds in the bounded below case. 
\item
The identity functor on $\sM_R$ is a right Quillen equivalence from the $(q,h)$-model structure 
to the $q$-model structure, hence is 
a left Quillen equivalence from the $q$-model structure to the 
$(q,h)$-model structure.
\end{enumerate}
\end{thm}  

The $(q,h)$-model structure is sometimes more natural than the $q$-model structure.
For example, when $R$ is commutative, the dualizable objects of $\sD_R$ are 
the perfect complexes, namely the objects of $\sM_R$ that are homotopy equivalent to 
bounded complexes of finitely generated projective $R$-modules.\footnote{Dualizability is usually thought of in $\sD_R$, and for that it is equivalent to define perfect complexes to be complexes quasi-isomorphic to bounded chain complexes of finitely generated projective $R$-modules. But for work before passage to derived categories, where dualizability already
makes sense, it is much more natural to define perfect in terms of homotopy equivalence.}  The homotopy invariance
means that these objects live naturally as $(q,h)$-cofibrant objects, although they need not
be $q$-cofibrant.  Using a $(q,h)$-cofibrant approximation $X$ of a DG $R$-module $M$,
we have
\begin{equation}\label{TorExt} 
\Tor_*^R(N,M) = H_*(N\otimes_R X) \ \ \ \mathrm{and}\ \ \ \Ext^*_R(M,N) = H^*\Hom_R(X,N), 
\end{equation}
the latter regraded cohomologically.  We can think of these as obtained by first applying the 
derived functors of $N\otimes_R(-)$ and $\Hom_R(-,N)$ and then taking homology groups.  When 
$M$ is an $R$-module regarded as a DG $R$-module concentrated in degree $0$, these are the 
$\Tor$ and $\Ext$ functors of classical homological algebra.

\begin{rmk} 
Although the $(q,h)$-cofibrant objects are precisely analogous to spaces of the homotopy 
types of CW complexes in topology \cite[\S17.4]{MP} and are of comparable conceptual
interest \cite[\S18.6]{GM}, they are not comparably easy to recognize.  We lack an
analogue of Milnor's classic characterization \cite{Milnor} of CW homotopy types that 
would let us recognize objects of $\sM_R$ that are homotopy equivalent to $q$-cofibrant objects
when we see them.
\end{rmk}

\begin{rmk}\mylabel{Ohyeah}  If $R$ is semi-simple, then $\sF_h=\sF_q$ since all epimorphisms split.  
It is well-known that $\sW_q=\sW_h$ in this case.  Indeed, if $f\colon X\rtarr Y$ is in $\sW_q$,
then it is in $\sW_h$ since its cofiber, the evident pushout $C = Y \cup_f (X\otimes D^1_R)$, is 
$q$-acyclic and therefore contractible, splittings $C_n\iso  Z_n\oplus B_n$ determining a 
contracting homotopy.  Therefore our three model structures on $\sM_R$ coincide in this case.  
Our interest is in commutative ground rings that are not semi-simple.
\end{rmk}

\subsection{Three mixed model structures on $\sM_A$}\label{3mixed}
Returning to our usual context of a DG $R$-algebra $A$,
we display and compare the six projective-type model structures that we have in sight on $\sM_A$.
There are actually more, but these are the ones that seem to us to be of most obvious interest.  Let us write
$\sM_A^t$ generically for $\sM_A$ with the $t$-model structure, with a pair of superscripts for mixed model structures.  
In the previous two sections, we discussed $\sM_A^q$, $\sM_A^h$, and $\sM_A^r$.  We complete their comparison
to the model structures on $\sM_R$ in the following observation.

\begin{lem}\mylabel{hqrSum} The adjunction $\bF \dashv \bU$ is a Quillen adjunction with respect to the $h$-, $q$-, or $r$-model
structures on both $\sM_R$ and $\sM_A$.  
\end{lem}
\begin{proof} We have already observed that this holds for the $q$- and $r$-model structures since those were
created on $\sM_A$ by lifting the corresponding model structures on $\sM_R$.  Obviously $\bU$ takes $A$-homotopy
equivalences to $R$-homotopy equivalences.  It takes $h$-fibrations $p$ in $\sM_A$ to $h$-fibrations since 
$\bF(K\otimes I)\iso (\bF K)\otimes I$, so that $\bU p$ has the RLP against $i_0\colon K\rtarr K\otimes I$ for all
$K\in \sM_R$ by adjunction.
\end{proof}    

By their definitions, we have inclusions
\[  \sW_h\subset \sW_r \subset \sW_q. \]
The following further inclusions should be almost obvious, but it seems worthwhile to give proofs.

\begin{lem}\mylabel{hqrSum3}  The following inclusions hold:
\[ \sF_h \subset \sF_r\subset \sF_q \ \ \tand\ \ \sC_h\supset \sC_r\supset \sC_q. \]
\end{lem}
\begin{proof}
The proof of \myref{hqrSum} shows that if $p$ is an $h$-fibration in $\sM_A$, then $\bU p$ is an $(h=r)$-fibration in $\sM_R$,  hence $p$ is an $r$-fibration.  If $i$ is an $r$-cofibration, it has the LLP against 
the $r$-acyclic $r$-fibrations $p_0\colon X^I\rtarr X$ and is thus an $h$-cofibration.

If $i$ is a $q$-cofibration, then it is a retract of an $\bF \cI_R$-cell complex
and thus of an enriched $\bF \cI_R$-cell complex, hence it is an $r$-cofibration.  Similarly, if $i$ is a $q$-acyclic
$q$-fibration, then it is a retract of an $\bF \cJ_R$-cell complex and thus of an enriched $\bF \cJ_R$-cell complex, 
hence it is an $r$-acyclic $r$-cofibration.  If $p$ is an $r$-fibration, it has the enriched RLP against $\bF\cJ_R$ and hence also the weaker unenriched RLP and is thus a $q$-fibration. 
\end{proof} 

Therefore \myref{generalmixedstructure} gives us mixed model structures
\[  \sM_A^{q,h} = (\sW_q, \sC_{q,h}, \sF_h), \]
\[  \sM_A^{q,r} = (\sW_q, \sC_{q,r}, \sF_r), \]
\[  \sM_A^{r,h} = (\sW_r, \sC_{r,h}, \sF_h). \]
The identity functor on $\sM_A$ gives right Quillen adjoints displayed in the diagram
\[ \xymatrix{ \sM_A^h \ar[r] \ar[dr] & \sM_A^r \ar[r] & \sM_A^q \\
& \sM_A^{r,h} \ar[r] \ar[dr] \ar[u] & \sM_A^{q,r} \ar[u] \\
& & \sM_A^{q,h}. \ar[u]\\} \]
From left to right, these arrows induce the evident functors $\sK_A \rtarr \sD_A^r \rtarr \sD_A$ on passage to 
homotopy categories.  From bottom to top they induce isomorphisms on homotopy categories. The comparisons of weak 
equivalences and fibrations are built into the definitions and \myref{hqrSum3}.

By \cite[17.3.3]{MP}, the following result is a formal consequence of \myref{hqrSum}. Note that $\sM_R^{q,h} = \sM_R^{q,r}$. 

\begin{lem}\mylabel{hqrSum2} The adjunction $\bF \dashv \bU$ is a Quillen adjunction with respect to the 
$(q,r)$- or $(q,h)$-model structures on both $\sM_R$ and $\sM_A$.
\end{lem}

Since every object is fibrant in all six of our model categories, we pass to homotopy categories by cofibrant 
approximation and passage to homotopy classes.  It is worth emphasizing the obvious: the relevant notion of 
homotopy is always that between maps of DG $A$-modules, which is the notion of homotopy used to define $\sW_h$. 
The $h$-fibrations are very natural, being the algebraic analogue of Hurewicz fibrations in topological situations, and they are the fibrations of the $h$-, $(r,h)$-, and $(q,h)$-model structures defined with respect to our three 
classes of weak equivalences.  The comparisons of cofibrations among our various model structures on $\sM_A$ are 
of interest, and we focus on cofibrant objects.  Since every object is $h$-cofibrant, the $(r,h)$- and $(q,h)$-cofibrant 
objects are just the $h$-homotopy types of $r$-cofibrant and $q$-cofibrant objects. 

Now focus further on the interrelationships among the cofibrant objects in the $r$-, $q$-, and $(q,r)$-model 
structures. The $r$-cofibrant objects are the retracts of $\sV_R$-enriched $\bF\cI_R$-cell complexes, the 
$q$-cofibrant objects are the retracts of ordinary $\bF\cI_R$-cell complexes, and the $(q,r)$-cofibrant
objects are the $r$-cofibrant objects that have the $r$-homotopy type (and thus the $q$-homotopy
type) of $q$-cofibrant objects. 

The derived category $\sD_A$ is our preferred homotopy category of interest, so we are most 
interested in the $q$-equivalences.  In the applications of [13], the relevant DG algebras 
$A$ are typically $R$-flat but not necessarily $R$-projective.  In such a situation, the most 
natural cofibrant approximations are given by bar constructions.  They are $r$-cofibrant, and 
we shall see in \S\ref{barsec} that they often behave homologically as if they are $(q,h)$-cofibrant, 
although they are generally not.   Bar constructions are large, of great theoretical importance, but of little 
calculational utility. On the other hand, there are calculationally accessible $q$-cofibrant approximations 
that can be compared to the bar construction, as we shall see by mimicry of classical homological algebra.

\section{Enriched and algebraic variants of the small object argument}\label{appendix} 

To construct functorial factorizations for the $q$-, $r$-, and $h$-model structures on $\sM_A$, we use three different versions of the small object argument (SOA), namely:  
\begin{itemize}
 \item  Garner's version of the classical SOA, used to construct  factorizations for the $q$-model structure (\S\ref{sec:SOA1});
\item the enriched SOA, used for  the $r$-model structure (\S\ref{sec:SOA2});
 \item the algebraic SOA, generalizing Garner's SOA to algebraically controlled classes and categories of generators, used for the $h$-model structure (\S\ref{ASOA}).
\end{itemize}
In more detail, the $q$-model structure is compactly generated in the classically understood sense (see \myref{compgen}), but the $r$- and $h$-model structures are not. As in the case of the $r$-model structure 
on $\sM_R$ described in \S\ref{Rmod}, the $r$-model structure on $\sM_A$ is compactly generated in an $R$-module enriched sense.  We present 
the necessary model theoretic machinery for the classical and enriched factorizations in parallel in \S\ref{sec:SOA1} and \S\ref{sec:SOA2}, 
respectively. 

In \S\ref{EMC}, we describe conditions under which the model structure created by a right adjoint from an existing enriched or monoidal model category is again enriched or monoidal. These general results are then used to prove that the $q$- and $r$-model structures on $\sM_A$ are $\sM_R$-model structures, monoidal if $A$ is commutative. For the $r$-model structure, our observations concerning enrichment are vital: 
the usual model structure lifting theorems take compact (or cofibrant) generation as a hypothesis, and the ($h=r$)-model 
structure on $\sM_R$  generally fails to satisfy that hypothesis in the traditional unenriched sense. 

Neither the classical nor the enriched SOA seem to be able to produce the desired factorization for the $h$-model structure on $\sM_A$, which also fails to be compactly generated. Instead, drawing inspiration from the Garner SOA, in \S\ref{ASOA} we describe an algebraic variant of the SOA, which allows the construction of weak factorization systems generated by \emph{classes} of maps that are algebraic in a sense made precise there. We conclude that section with a proof of the factorization axiom for the $h$-model structure on $\sM_A$, \myref{hfactor}.

Note that we restrict ourselves to the SOA based on $\omega$-indexed colimits, which is important for the applications in Part 2. The constructions and results in this section generalize effortlessly for any regular cardinal $\kappa$ in place of $\aleph_0$. 

\subsection{The classical small object argument}\label{sec:SOA1}

To provide context, we briefly recall  the classical SOA, used to produce the 
factorizations for the $q$-model structures on $\sM_R$ and $\sM_A$. Let $\cI$ be a set of maps in a cocomplete category. Under certain set-theoretical conditions, the SOA 
constructs a functorial factorization such that the right factor of any map has the RLP against $\cI$ 
and the left factor is a relative $\cI$-cell complex.  This construction demonstrates the existence of the weak factorization system compactly generated by $\cI$.

 The version of the SOA we present is a variant of 
Quillen's original construction,  due in its 
general form to Garner \cite{Garner} and in the special case used here to the Ph.D. thesis of Radulescu-Banu \cite{Radu}. 
The use of this version of the SOA to construct factorizations for the $q$-model structures is merely a matter of taste but very much in accordance with the philosophy of compact generation. 
Other expositions of our philosophy can be found in \cite[Chapter 15]{MP} 
and \cite[Chapter 12]{Riehl}.

\begin{defn}\mylabel{wfs}
A \emph{weak factorization system} (WFS) $(\cL,\cR)$ on a category consists of two classes of morphisms $\cL$ and $\cR$ such that
\begin{enumerate}[(i)]
\item Every morphism can be factored as $r l$ with $l \in \cL$ and $r \in \cR$.
\item $\cL$ is the class of maps with the LLP against $\cR$, and $\cR$ is the 
class of maps with the RLP against $\cL$.
\end{enumerate}
\end{defn}

\begin{defn}\mylabel{cellob}  Let $\sM$ be a cocomplete category. Let $\cI$ be a set of maps in $\sM$ and
let $X\in \sM$.  A  \emph{relative} $\cI$-\emph{cell complex} under $X$ is a map
$f\colon X\rtarr Y$, where $Y = \colim F_nY$ is the colimit of a sequence of maps
$F_nY\rtarr F_{n+1}Y$ such that $F_0Y = X$ and $F_{n+1}Y$ is obtained as 
the pushout in a diagram
\[  \xymatrix{
\coprod J_q \ar[d]_{\coprod{i_q}} \ar[r]^-{j} & F_{n}Y \ar[d] \\
\coprod K_q \ar[r]_-k & F_{n +1}Y,\\} \]
where each $i_q \in \cI$ and the coproducts are indexed by some set. 
The components of $j$ are called attaching maps, and the components
of $k$ are called cells. An object $C$ of $\sM$ is \emph{compact} with 
respect to $\cI$ if for every relative $\cI$-cell complex 
$f\colon X\rtarr Y$, the canonical map
\[ \colim_{n}\sM(C, F_nY) \rtarr \sM(C,Y) \]
is a bijection.  The set $\cI$ is \emph{compact} if every domain object  of a map 
in $\cI$ is compact with respect to $\cI$.
\end{defn}

For a class $\cI$ of maps in a category, let $\cI^{{\boxslash}}$ denote the class of maps with the RLP with respect to $\cI$; 
similarly, let ${}^\boxslash(\cI^\boxslash)$ denote the class of maps with the LLP with respect to $\cI^\boxslash$. The SOA provides a constructive proof of the following theorem.

\begin{thm}\mylabel{SOA}
Any compact set of arrows $\cI$ in a cocomplete category  generates a weak factorization system whose right class is $\cI^\boxslash$.  Moreover, the left class ${}^{\boxslash}(\cI^{{\boxslash}})$ is precisely the class of retracts of relative $\cI$-cell complexes.
\end{thm}

When the relative $\cI$-cell complexes are monomorphisms, as is always true in the cases we consider, the difference between Quillen's SOA and Garner's 
SOA is simple to describe. Quillen constructs factorizations in which the left factor is a sequential colimit of pushouts of coproducts of generating maps; the coproducts are indexed over all commutative squares between the generating arrows and the right factor constructed at the previous stage. 

Garner's construction is similar, except that ``cells are attached only once,'' meaning that any commutative square whose attaching map factors through some previous stage of the sequential colimit is omitted from the indexing coproduct. When $\cI$ is compact, this process \emph{converges} at $\omega$: it is not possible to attach any new non-redundant cells. See \cite{Garner} or \cite[Chapter 12]{Riehl} for more details. 

To illustrate the difference between Quillen's and Garner's SOA, we include a simple relevant example.

\begin{ex}
We  use Garner's SOA to factor a map $f \colon X \rtarr Y$ in accordance with the WFS generated by 
$\cJ_R$ on $\sM_R$. Commutative squares between the generating arrow $0\rtarr D^n_R$ and $f$ are indexed by the underlying graded set of 
$Y_n$. We have a ``one-step'' factorization of $f$ defined via the left-hand pushout square in
\[ \xymatrix{ 0 \ar[r] \ar[d]  \ar@{}[dr] & X \ar@{=}[r] \ar[d]^{\lambda_f} & X \ar[d]^f \\ \oplus_n \oplus_{Y_n} D^n_R \ar[r] & X \oplus (\oplus_n \oplus_{Y_n} D^n_R) \ar[r]_-{\rho_f} & Y.}\] 
Quillen's SOA would proceed by composing $\lambda_f$ with the pushout of generating arrows indexed over commutative squares with codomain $\rho_f$. Garner's SOA performs no further attachments and terminates after the first step. Indeed, $\rho_f$ is a degreewise epimorphism that already has the RLP against $\cJ_R$.
\end{ex}

The main applications of the SOA are to the construction of model structures.
We turn to the context used to construct the $q$-model structures. 

\begin{defn}\mylabel{compgen} A model structure $(\sW,\sC,\sF)$ on a category $\sM$ is \emph{compactly 
generated} if there are compact sets $\cI$ and $\cJ$ of maps in $\sM$ such 
that $\sC$ is the subcategory of retracts of $\cI$-cell complexes
and $\sC\cap \sW$ is the subcategory of retracts of $\cJ$-cell complexes. In this case, $\sF = \cJ^\boxslash$ and $\sF \cap \sW = \cI^\boxslash$. The sets $\cI$ and $\cJ$ are called the generating cofibrations
and generating acyclic cofibrations.
\end{defn}

We use the next result to lift the $q$-model structure from $\sM_R$ to $\sM_A$.  In it, we have in mind our standard adjunction
$\bF\dashv \bU$ between $\sM_R$ and $\sM_A$, with enrichment in $\sU = \sM_R$.  We assume in the rest of this section that $\sU$ is a cosmos with a monoidal model structure; we refer to such a category as a monoidal model category.   Later, we must add in enrichment of $\sU$ in a second  category, which we will denote 
by $\sV$; in our DG context, $\sV$ will be $\sV_R$.

\begin{thm}\mylabel{cofcrit} Consider an adjunction $\xymatrix@1{\bF\colon \sM \ar@<.4ex>[r] & \ar@<.4ex>[l]\sN\colon \bU}$ between cocomplete categories such that $\sM$ has a model structure compactly generated by sets $\cI$ and $\cJ$.
Assume the following two conditions hold.
\begin{enumerate}[(i)] \item (Compactness condition) The sets $\bF\cI$ and $\bF\cJ$ are compact in $\sN$. 
\item (Acyclicity condition) The functor $\bU$ carries every relative $\bF\cJ$-cell complex to a weak equivalence.
\end{enumerate}
Then $\sN$ has a model structure whose fibrations and weak equivalences are created by the right adjoint $\bU$. It is compactly 
generated by $\bF\cI$ and $\bF\cJ$, and $\bF \dashv \bU$ is a Quillen adjunction. 

Moreover, if $\sM$ and $\sN$ are bicomplete $\sU$-categories, $\sM$ is a $\sU$-model category, and $\bF$ preserves tensors, 
then $\sN$ is again a $\sU$-model category. If $\sM$ is a monoidal model category, $\sN$ is a monoidal category, and $\bF$ preserves the monoidal product, then $\sN$ is a monoidal model category provided the unit condition is satisfied.
\end{thm}

A proof of the first part can be found, for example, in \cite[11.3.1-2]{Hirschhorn} or \cite[16.2.5]{MP}. The second part should be equally standard, but we do not know a published reference; for a proof of the enriched version of this result, see \S\ref{EMC}.  

\begin{rmk}
The ``unit condition'' referred to in the second part of \myref{cofcrit} is
described in \cite[4.2.6, 4.2.18]{Hovey}.  It is needed to ensure that the monoidal unit of $\sN$ or $\sU$ gives 
rise to a unit for the monoidal structure or $\mathrm{Ho}\sU$-enrichment on the homotopy category of $\sN$. This 
condition is automatically satisfied when the monoidal unit is cofibrant, as is always the case for the model 
structures that we consider in this paper, so we will not discuss it further.
\end{rmk}

\subsection{Enriched weak factorization systems and relative cell complexes}\mylabel{sec:SOA2}

We now describe the definition and construction of enriched WFSs in analogy to the unenriched setting of the last section. 
A more thorough account of this theory is given in \cite[Chapter 13]{Riehl}.  We assume that $\sV$ is a cosmos; we do not
assume that it has a given model structure.

\begin{defn}\mylabel{ewfs} Let $\sM$ be a bicomplete $\sV$-category. An \emph{enriched weak factorization system} consists of 
classes of maps $\cL$ and $\cR$ such that 
\begin{enumerate}[(i)]
\item Every morphism can be factored as $r  l$ with $l \in \cL$ and $r \in \cR$.
\item $\cL$ is the class of maps with the enriched LLP against $\cR$, and $\cR$ is the class of maps with the 
enriched RLP against $\cL$.
\end{enumerate}
\end{defn}

\begin{defn}\mylabel{cellob2}
Let $\sM$ be a bicomplete $\sV$-category. Let $\cI$ be a set of maps in $\sM$ and
let $X\in \sM$.  An \emph{enriched relative} $\cI$-\emph{cell complex} under $X$ is a map
$f\colon X\rtarr Y$, where $Y = \colim F_nY$ is the colimit of a sequence of maps
$F_nY\rtarr F_{n+1}Y$ such that $F_0Y = X$ and $F_{n+1}Y$ is obtained as 
the pushout in a diagram
\[  \xymatrix{
\coprod J_q\otimes V_q \ar[d]_{\coprod{i_q}\otimes \id} \ar[r]^-{j} & F_{n}Y \ar[d] \\
\coprod K_q\otimes V_q \ar[r]_-k & F_{n +1}Y,\\} \]
where each $i_q \in \cI$, each $V_q \in \sV$, and the coproducts are indexed by some set. 
The components of $j$ are called attaching maps and the components
of $k$ are called cells. An object $C$ of $\sM$ is \emph{compact} with 
respect to $\cI$ if for every enriched relative $\cI$-cell complex 
$f\colon X\rtarr Y$, the canonical map
\[ \colim_{n}\sM(C, F_nY) \rtarr \sM(C,Y) \]
is a bijection.  The set $\cI$ is called \emph{compact} if every domain object of a map 
in $\cI$ is compact with respect to $\cI$.
\end{defn}

For a class $\cI$ of maps in $\sM$, let $\cI^{\ul{\boxslash}}$ denote the class of maps that have the enriched RLP against $\cI$;
similarly, let ${}^{\ul{\boxslash}}(\cI^{\ul{\boxslash}})$ denote the class of maps with the enriched LLP against $\cI^{\ul{\boxslash}}$. The enriched SOA \cite[\S 13.2]{Riehl} provides a constructive proof of the following theorem.

\begin{thm}\mylabel{enrichedboxslashclosure} 
Any compact set of arrows
$\cI$ in a bicomplete $\sV$-category generates an enriched weak factorization system whose right class is $\cI^{\ul{\boxslash}}$.  Moreover, 
its left class ${}^{\ul{\boxslash}}(\cI^{\ul{\boxslash}})$ contains all retracts of enriched relative $\cI$-cell complexes and consists 
precisely of such retracts when all enriched relative $\cI$-cell complexes are monomorphisms.
\end{thm}

\begin{rmk}
In contrast with \myref{SOA}, in \myref{enrichedboxslashclosure} we ask that the ambient category be $\sV$-bicomplete. Colimits and tensors are used to construct the factorizations produced by the enriched SOA. The presence of cotensors guarantees that this defines a $\sV$-enriched functorial factorization.
\end{rmk}

As in the unenriched situation, enriched WFSs can be created through an enriched adjunction \cite[13.5.1]{Riehl}. 

\begin{thm}\mylabel{ecofcrit}
Consider a $\sV$-adjunction $\xymatrix@1{\sM \ar@<.4ex>[r]^-\bF & \ar@<.4ex>[l]^-\bU} \sN$ between $\sV$-bicomplete categories such that $\sM$ has a model structure that is compactly generated in the $\sV$-enriched sense by the sets $\cI$ and $\cJ$. Assume the following two conditions hold.
\begin{enumerate}[(i)] 
\item (Compactness condition) The sets $\bF\cI$ and $\bF\cJ$ are compact in $\sN$. 
\item (Acyclicity condition) The functor $\bU$ carries every relative $\bF\cJ$-cell complex to a weak equivalence.
\end{enumerate}
Then $\sN$ has a model structure whose fibrations and weak equivalences are created by the right adjoint $\bU$. It is compactly generated in the $\sV$-enriched sense by $\bF\cI$ and $\bF\cJ$, and $\bF \dashv \bU$ is a Quillen adjunction.
\end{thm}
As in \myref{cofcrit}, it is often possible to infer that the model structure created on $\sN$ by \myref{ecofcrit} is monoidal or enriched when the model structure on $\sM$ is so.

\begin{thm}\mylabel{ecofcrit-monoidal} Assume in addition to the hypotheses of \myref{ecofcrit} that $\sM$ is a monoidal model category, $\sN$ is a monoidal category, and $\bF$ preserves the monoidal product.  Then $\sN$, equipped with the model structure constructed in \myref{ecofcrit}, is a monoidal model category provided the unit condition is satisfied. 
\end{thm}

The proof of \myref{ecofcrit-monoidal} is no more difficult than in the unenriched case. In the next section, we prove it simultaneously with the analogous result for enriched model structures, but we must first clarify the relevant double enrichments of the ``enriched model categories'' used when both $\sU$ and $\sV$ are present. 

\subsection{The two kinds of enriched model categories}\label{EMC}

The main theorems in the previous two sections provide criteria to transfer a model structure on a category $\sM$ along an adjunction $\xymatrix@1{\bF\colon \sM \ar@<.4ex>[r] & \ar@<.4ex>[l] \sN \colon \bU}$. The weak equivalences and fibrations in the lifted model structure on $\sN$ are created by the right adjoint $\bU$. A key hypothesis is that the model structure on $\sM$ is compactly generated, in either the enriched or the unenriched sense. 

In this section we assume further that the model category $\sM$ has a compatible enriched or monoidal structure and establish conditions under which the transferred model structure on $\sN$ has analogous properties. Since this material has not appeared in the literature before, we work in slightly greater generality than is strictly necessary for our applications. 

We have two distinct notions of ``enriched model category'' appearing simultaneously here. The first is a model category whose constituent WFSs are enriched over a cosmos $\sV$, as explained in the previous section.  As in the case $\sV= \sV_R$, no model structure on $\sV$ is needed for that, although one may well be present and relevant.  We will call such categories \emph{$\sV$-enriched model categories}. Secondly, there is the more standard notion of a \emph{$\sU$-model category} generalizing Quillen's definition of a simplicial model category (e.g. \cite[4.2.18]{Hovey}). Here $\sU$ is a cosmos with a monoidal model structure. The $\sM_{\bZ}$-model categories of Theorems \ref{qmodelR}, \ref{hmodelR}, and \ref{mmodelR}, and the $\sM_R$-model categories of Theorems \ref{qmodelA}, \ref{hmodelA}, and \ref{rmodelA-monoidal} are of this type.

When all objects are cofibrant, the WFSs of a $\sU$-model category are $\sU$-enriched WFSs, but in general 
that is not the case! The comparison is discussed further in \cite[\S 13.5]{Riehl} but that is irrelevant to 
our applications here.  

The enriched version of \myref{ecofcrit-monoidal} has both sorts of enrichments occurring simultaneously. Suppose we have two cosmoi 
$\sV$ and $\sU$, where the latter is a monoidal model category, together with a \emph{strong monoidal adjunction} 
$\xymatrix@1{\sV \ar@<.4ex>[r] & \ar@<.4ex>[l] \sU}$; that means that the left adjoint commutes with the monoidal products. It follows that there is a ``change of base'' $2$-functor that translates any $\sU$-enriched category, functor, or adjunction into a $\sV$-enriched category, functor, or adjunction where the $\sV$-enrichment is defined by applying the right adjoint; see \cite[\S 3.7]{Riehl}. For instance, in our applications, $\sV = \sV_R$, $\sU = \sM_R$ equipped with the $q$- or $r$-model structure, and in the relevant adjunction 
$\xymatrix@1{\sV_R \ar@<.4ex>[r] & \ar@<.4ex>[l] \sM_R}$
the left adjoint includes an $R$-module in degree $0$ and the right adjoint takes cycles in degree $0$.

The context for the $\sU$-model structure transfer result is a $\sU$-adjunction between bicomplete $\sU$-categories $\sM$ and $\sN$ 
such that the induced $\sV$-enrichments satisfy the hypotheses of \myref{ecofcrit}.  The second half of \myref{cofcrit} is then contained 
in the following theorem, which remains true under weaker hypotheses (regarding the interactions between the enrichments) than we state here.

\begin{thm}\mylabel{ecofcrit-enriched}
Let $\sU$ be a monoidal model category, let $\sV$ be a cosmos, and suppose that we are given a strong monoidal adjunction 
$\xymatrix@1{\sV \ar@<.4ex>[r] & \ar@<.4ex>[l] \sU}$. Let $\sM$ be a compactly generated $\sU$-model category, $\sN$ be 
a bicomplete $\sU$-category, and $\xymatrix@1{\bF \colon \sM \ar@<.4ex>[r] & \ar@<.4ex>[l] \sN \colon \bU}$
be a $\sU$-adjunction such that the underlying $\sV$-adjunction satisfies the hypotheses of \myref{ecofcrit}. 
Then the $\sV$-enriched model structure on $\sN$ defined by \myref{ecofcrit} makes $\sN$ a $\sU$-model category.
\end{thm}

Theorems \ref{ecofcrit-monoidal} and \ref{ecofcrit-enriched} follow easily from the following lemma.  
Let $\sN$ be a bicomplete $\sU$-category, such as $\sU=\sN$.  Then the tensor, cotensor, hom adjunctions (as for example
in (\ref{2variable})) are enriched over $\sU$.  We have analogous adjunctions relating the tensor, cotensor, and hom
functors $(\widehat{\otimes}, \widehat{[\, ,\, ]}, \hathom)$ on the arrow 
categories of $\sN$ and $\sU$ defined using pushouts and pullbacks (see e.g.~\cite[\S 11.1]{Riehl}) and they too are enriched over $\sU$. The strong monoidal adjunction 
$\xymatrix@1{\sV \ar@<.4ex>[r] & \ar@<.4ex>[l] \sU}$ is used to enrich these adjunctions 
over $\sV$. In our applications, we will take $\sV=\sV_R$ and take $\xymatrix@1{\sV \ar@<.4ex>[r] & \ar@<.4ex>[l] \sU}$ to be
\[ \xymatrix@1{\sV_R \ar@<.4ex>[r] & \ar@<.4ex>[l] \sM_R} \ \ \text{or}\ \ 
\xymatrix@1{\sV_R \ar@<.4ex>[r] & \ar@<.4ex>[l] \sM_R \ar@<.4ex>[r] & \ar@<.4ex>[l] \sM_A\\ } \]
for the enriched and monoidal cases, respectively; in the latter enrichment, the arrow $\hathom(i,p)$ is the map  $\epz$ of \eqref{epsilon-defn}.  

\begin{lem}\mylabel{boxclosure}
Suppose $\cI, \cJ, \cK$ are sets of arrows in $\sN$ with the property that $(\cI \widehat{\otimes}  \cJ) \subset \cK$. Abbreviating $\overline{S}={}^{\underline{\boxslash}}(S^{\underline{\boxslash}})$, we have
\[ (\overline{\cI} \widehat{\otimes} \overline{\cJ}) \subset \overline{\cK}.\]
\end{lem}
\begin{proof}
The $\sV$-enriched $(\widehat{\otimes}, \widehat{[\, ,\, ]}, \hathom)$ adjunctions respect the $\sV$-enriched lifting properties in the expected sense: 
\[ (i \widehat{\otimes} j) \underline{\boxslash} f\qquad \mathrm{iff} \qquad i \underline{\boxslash} \widehat{[j,f]} \qquad \mathrm{iff} \qquad j \underline{\boxslash}\, \hathom(i,f).\] 
Since $(\cI \widehat{\otimes}  \cJ) \subset \cK \subset \overline{\cK}$, we have $(\cI \widehat{\otimes} \cJ) \underline{\boxslash} \cK^{\underline{\boxslash}}$. By adjunction, $\cI \subset {}^{\underline{\boxslash}}\,\hathom(\cJ, \cK^{\underline{\boxslash}})$ and thus $\overline{\cI} {\underline{\boxslash}}\, \hathom(\cJ, \cK^{\underline{\boxslash}})$. Again using adjunction, we see that $\overline{\cI} \widehat{\otimes} \cJ \subset \overline{\cK}$. Now apply the dual argument to $\cJ$ to arrive at the claim. 
\end{proof}

\begin{proof}[Proof of Theorems \ref{ecofcrit-monoidal} and \ref{ecofcrit-enriched}]
We give the proof of \myref{ecofcrit-monoidal}, as that of \myref{ecofcrit-enriched} is completely analogous. Denoting the generating cofibrations and generating trivial cofibrations in $\sM$ by $\cI$ and $\cJ$ respectively, \myref{boxclosure} shows that it is enough to check the pushout-product axiom on the generators on $\sN$; see \cite{schwedeshipley}. Since $\bF$ is monoidal and a left adjoint, $\bF$ preserves the pushout-product, that is, $\bF(f) \widehat{\otimes} \bF(g)  \cong \bF(f \widehat{\otimes}g)$. Therefore, since $\bF$ is left Quillen,
\[ \bF(\cI) \widehat{\otimes} \bF(\cI) \cong \bF(\cI \widehat{\otimes} \cI) \subset  \bF(\sC_{\sM}) \subset \sC_{\sN}, \]
where $\sC_{\sM}$ and $\sC_{\sN}$ are the cofibrations in $\sM$ or $\sN$. Similarly, we obtain 
\[ \bF(\cI) \widehat{\otimes} \bF(\cJ) \subset \sW_{\sN} \cap \sC_{\sN}.\] 
The result follows. 
\end{proof}

By specializing Theorems \ref{ecofcrit-monoidal} and \ref{ecofcrit-enriched} to $\sV  = \sV_R$, $\sU = \sM = \sM_R$, and $\sN = \sM_A$, 
we obtain the results we need to complete our work on the $q$- and $r$-model structures.

\begin{cor}
The $q$- and $r$-model structures on $\sM_A$ are $\sM_R$-model structures, where $\sM_R$ is equipped with the $q$- and $r$-model structures, respectively. If $A$ is commutative, then both model categories are monoidal.
\end{cor}

\subsection{The algebraic small object argument}\label{ASOA}

Garner's version of the SOA is an ``algebraization'' of the Quillen SOA: the functorial factorization it produces defines an algebraic weak factorization system, which is a weak factorization system with additional structure that leads to better categorical properties. Garner's categorical description of this construction suggests the possibility of a generalization to classes and even (large!)~categories of generating arrows---provided that the right lifting property so-encoded can be controlled algebraically in a sense we will make precise.  

These ideas were first developed in a topological context in \cite{BR} to construct functorial factorizations for the $h$-model structures discussed there. In parallel with \myref{why-enriched}, such model structures had been previously asserted in the literature, but the proofs of the factorization axioms given in \cite{Cole2,MP} fail. Here we introduce a generalized form of what we'll call the \emph{algebraic SOA} not simply because we find these ideas compelling: again, we know of no other proof that $\sM_A$ has the factorizations necessary to define the $h$-model structure.  There are many other prospective applications of the work here and in \cite{BR}.

Let $\sM^{\mathbf{2}}$ denote the arrow category of $\sM$.  Roughly speaking, a WFS $(\cL, \cR)$ on $\sM$ is said to be \emph{algebraic} if there exists a comonad $\bL$ and a monad $\bR$ on $\sM^{\mathbf{2}}$ defining a functorial factorization $f = \bR f \circ \bL f$ such that $\cL$ and $\cR$ are the retract closures of $\Coalg_{\bL}$ and $\Alg_{\bR}$, respectively. Here $\Coalg_{\bL}$ and $\Alg_{\bR}$ denote the categories of coalgebras and algebras for the comonad and monad. In order for $\bL$ and $\bR$ to define a functorial factorization, we require the functor $\bR$ to commute with the codomain projection $\sM^{\mathbf{2}} \rtarr \sM$ and we require the codomain components of the unit and multiplication natural transformations to be identities.  The monad $\bL$ has dual requirements.  We refer the reader to \cite{BR, Riehl, Riehlalg} for the precise definitions and further discussion. 

The idea is that the extra algebraic structure present in an algebraic WFS ensures a very close relationship between the given factorizations and the lifting properties of the classes $\cL$ and $\cR$. We shall see that the algebraic description of the class $\cR$, as (retracts of) algebras for a monad or something similar, provides a useful replacement for the kind of characterization $\cR = \cI^\boxslash$ present in the cofibrantly generated case. More precisely, we will characterize the right class $\cR$ of an algebraic WFS as algebras for a pointed endofunctor $R$ of 
$\sM^{\mathbf{2}}$ over cod: i.e., an endofunctor $R$ admitting a natural transformation $\id \rtarr R$ whose codomain component is the identity. The notation $R$ for an arbitrary endofunctor should not be confused with the symbol $\bR$, which is reserved for monads, as for example the algebraically free monad on $R$. The notion of an algebra for a pointed endofunctor is analogous to the notion of an algebra for a monad, except that there is no associativity requirement.

The definition of $\cI^{\boxslash}$ can be extended in two directions: Instead of a class of morphisms $\cI$ we can take as input a subcategory $\cI \hookrightarrow \sM^{\mathbf{2}}$ of the arrow category, and instead of a class of morphisms with the right lifting property we can construct a category of such morphisms. To this end, let $\cI \hookrightarrow \sM^{\mathbf{2}}$ be a (typically non-full) subcategory.  Define a category $\cI^{\boxslash}$ in which an object is an arrow $f$ of $\sM$ equipped with a function $\phi_f$ specifying a solution to any lifting problem against a map in $\cI$, subject to the following condition: these chosen lifts are natural with respect to morphisms $j' \rtarr j \in \cI$ in the sense that the following diagram of lifts commutes 
\[ \xymatrix{ J' \ar[r] \ar[d]_{j'} & J \ar[d]_(.3)j \ar[r] & X \ar[d]^f \\ K' \ar[r] \ar@{-->}[urr] & K \ar@{-->}[ur] \ar[r] & Y.}\] 
Morphisms $(u,v) \co (f, \phi_f) \rtarr (f',\phi_{f'})$ in the category $\cI^\boxslash$ are commutative squares so that the triangle of lifts displayed below commutes 
\[ \xymatrix{ J \ar[d]_j \ar[r] & K \ar[d]^(.65)f \ar[r]^u & X' \ar[d]^{f'} \\ B \ar@{-->}[ur] \ar[r] \ar@{-->}[urr] & Y \ar[r]_v & Y'.}\] 
The left class ${}^\boxslash\cI$ can be categorified similarly.\footnote{ There are standard set-theoretic foundations that permit the definition of a function whose domain is a class, rather than a set (e.g., a pair of nested Grothendieck universes).}

In general, the category $\cI^\boxslash$ is too large to be of practical use. However, in the examples considered in \cite{BR} and also here, the lifting function $\phi_f$ associated to a morphism $f$ can be encoded in an  alternative way: the data of the lifting function $\phi_f$ is equivalent to the data of an $R$-algebra structure on the morphism $f$, where $R$ is a certain pointed endofunctor over cod, as described above. In the proof of \myref{hfactor}, we will show that $\cI^\boxslash$ is isomorphic to the category 
$\Alg_R$ of $R$-algebras for such an $R$. This is an ordinary locally small category with a class of objects. There are no higher universes needed.

In order for the algebraic version of the SOA to apply, the endofunctor $R$ must satisfy a smallness condition, the precise general statement of which requires just a little terminology.  An \emph{orthogonal factorization system} $(\cE,\cM)$ in a category $\sM$ is a WFS for which both the factorizations and the liftings are unique. It is called \emph{well-copowered} if every object has a mere set of $\cE$-quotients, up to isomorphism. When $\sM$ is cocomplete it follows that the maps in 
$\cE$ are epimorphisms \cite[1.3]{Kellyunified}. This general context gives a technically convenient class of maps $\cM$ that behave like monomorphisms, although they need not always be such.

Consider a well-copowered orthogonal factorization system $(\cE,\cM)$ on a given bicomplete 
category $\sM$. A colimit cocone in $\sM$ whose morphisms to the colimit object 
are in the right class $\cM$ is called an $\cM$-\emph{colimit}. This implies that the
morphisms in the colimit diagram also lie in $\cM$, by the right cancellation property, but the converse is not true in general. In what follows, we will implicitly identify a regular cardinal $\alpha$ with its initial ordinal, so that $\alpha$ indexes a (transfinite) sequence whose objects are $\beta < \alpha$. We consider the following smallness condition on an endofunctor $R$ on $\sM$. It was introduced in \cite{Kellyunified}.
\begin{itemize}
\item[($\dagger$)]  There is a well-copowered orthogonal factorization system $(\cE, \cM)$ on $\sM$ and a regular cardinal $\alpha$ such that $R$ sends $\alpha$-indexed $\cM$-colimits to colimits.
\end{itemize}

In any category, there is a notion of a strong epimorphism; it is discussed in detail in
\cite[\S4.3]{BorI}.  As assured by \cite[4.4.3]{BorI}, in all categories $\sM$ that one meets in 
practice there is a canonical orthogonal factorization system $(\cE,\cM)$ such that the 
morphisms in $\cE$ are the strong epimorphisms and the morphisms in $\cM$ are the monomorphisms. 
Then every morphism in $\sM$ factors uniquely as the composite of a strong 
epimorphism and a monomorphism.
In particular, we have this if $\sM$ is locally presentable, in which case this orthogonal factorization system is automatically well-copowered 
by a result of \cite[1.61]{AR}. 

Since all categories considered in this paper are locally presentable, we implicitly work with the canonical well-copowered orthogonal factorization system given by strong epimorphisms and monomorphisms. The extra flexibility added by allowing different choices is required for applications to other contexts, for example, to topological categories. We are ready to state an abstract version of the main argument of \cite{BR, Garner}. Here ($\dagger$) is applied to an endofunctor $R$ of the arrow category $\sM^2$.

\begin{thm}\mylabel{algebraicfactorization}
Let $\sM$ be a bicomplete and locally small category and $\cI \hookrightarrow \sM^2$ be a subcategory of the arrow category. Assume that there is an isomorphism of categories $\cI^{\bs} \cong \Alg_{R}$ over $\sM^2$ for some pointed endofunctor $R$ over $\cod$. If $R$ satisfies
the smallness condition $(\dagger)$, then there exists an algebraic weak factorization system $(\bL, \bR)$ on $\sM$ with underlying weak factorization system $({}^{\bs}(\cI^{\bs}),\cI^{\bs})$. In particular, every morphism $f: X \rtarr Y$ can be factored as
\[\xymatrix{X \ar[r]^{\bL f} &Z \ar[r]^{\bR f} & Y}\] 
where ${\bL}f \in {}^{\bs}(\cI^{\bs})$ and ${\bR}f \in \cI^{\bs}$.
\end{thm}
\begin{proof}[Idea of proof]
On the one hand, smallness $(\dagger)$ allows the construction of the algebraically free monad on $R$, which is a monad $\bR$ together with a natural isomorphism $\Alg_{\bR} \cong \Alg_{R}$ over $\sM^2$. On the other hand, by assumption there is also an isomorphism of categories 
$\cI^{\bs} \cong \Alg_{R}$, which can be used to show that the category of algebras for $\bR$ encodes the structure of an algebraic weak factorization system. The desired lifting properties follow formally. 
\end{proof}

Just as in \cite{BR}, \myref{algebraicfactorization} can be used to construct factorizations for the $h$-model structure on $\sM_A$.

\begin{proof}[Proof of \myref{hfactor}] Define $\cJ_h$ to be the subcategory of the arrow category 
$\sM_A^{\mathbf{2}}$ whose objects 
are the maps $i_0\colon W \rtarr W \otimes I$ for $W \in \sM_A$ and whose morphisms are the maps of arrows induced 
in the evident way by maps $W'\rtarr W$ in $\sM_A$.  Then $\cJ_h^{\boxslash}$ is isomorphic over 
$\sM_A^{\mathrm{2}}$ to the category $\Alg_{R}$ of algebras for the pointed endofunctor ${R}: \sM_A^{\mathrm{2}} \rtarr \sM_A^{\mathrm{2}}$ over $\cod$ constructed as follows. First, note that, for a fixed $f: X \rtarr Y$, the functor $ \sM_A^{\mathrm{op}} \rtarr \Set$ that sends an object $W \in \sM_A$ to the collection of squares 
\begin{equation}\label{liftingproblem} \xymatrix{W \ar[r] \ar[d]_{i_0} & X \ar[d]^f \\
W \otimes I \ar[r] & Y}
\end{equation}
is represented by the mapping cocylinder $Nf$. Thus every such square factors as
\[\xymatrix{ W \ar[r]\ar[d]_{i_0} & Nf  \ar[d]_{i_0}  \ar[r] & X \ar[d]_{{\mL}f} \ar@{=}[r] & X \ar[d]^f \\ 
W \otimes I \ar[r]  & Nf \otimes I \ar[r] & Ef \ar[r]_{{R}f} \ar@{-->}[ur]^(.6)s & Y,}\]
where $Ef$ is the pushout in the central square.
This gives the definition of the endofunctor $R$, and the indicated lift $s$ provides an ${R}$-algebra structure on $f$. An easy check shows that such an algebra structure corresponds to lifts in all squares \eqref{liftingproblem}, satisfying the compatibility conditions present in $\cJ_h^{\boxslash}$. 
The details of the analogous proof in the topological setting can be found in \cite[5.10]{BR}, and the details here
are essentially the same.

We are left with the verification of the smallness hypothesis $(\dagger)$. Since limits and filtered colimits in $\sM_A$ are computed degreewise, sequential colimits commute with pullbacks. This and the fact that $\sM_A$ is locally presentable imply by \cite[5.20]{BR} that the smallness condition $(\dagger)$ is satisfied for the functor $R$ constructed above. Applying \myref{algebraicfactorization} to $\cJ_h^{\boxslash} \cong \Alg_{R}$ completes the proof.
\end{proof}

\begin{rmk}
Note that $\sM_A$ is a Grothendieck abelian category: it has the generator $A$, and filtered colimits are exact. Therefore it is locally presentable \cite[3.10]{Bek} and every object is small \cite[1.2]{HoveySheaves}. This gives another proof that the smallness condition 
$(\dagger)$ is satisfied.
\end{rmk}

\begin{rmk}
Garner's work on the small object argument can be interpreted as saying that the cofibrantly generated case fits into the above framework. To be precise, if $\cJ \subset \sM^{\mathbf{2}}$ is either a set or, more generally, a \emph{small} category, then there exists a pointed endofunctor $R$ such that $\cJ^{\bs} \cong \Alg_R$ over $\sM^2$ \cite[4.22]{Garner}. In this sense, Theorem \ref{algebraicfactorization} contains Garner's variant of Quillen's 
SOA as a special case. 
\end{rmk}

\begin{rmk}
The methods in this section can be generalized to take enrichments into account, thereby producing an \emph{enriched algebraic small object argument}; cf.~\myref{trivial}.
\end{rmk}


\part{Cofibrant approximations and homological resolutions}

\section{Introduction}
Having completed our model theoretic work, we turn to a more calculational point of view.
The theme is to give calculationally useful concrete constructions of cofibrant approximations,
starting from homological algebra and different types of homological resolutions.   The motivation is to understand differential homological algebra conceptually
and calculationally.  In fact, the pre model theoretical literature gives different definitions 
of differential $\Tor$ and $\Ext$ functors based on different kinds of resolutions, and our work 
gives the first proof that these definitions agree.  The early definitions are given in terms of 
what we now recognize to be different cofibrant approximations of the same DG $A$-modules, and 
these explicit cofibrant approximations give tools for calculation.

\subsection{The functors $\Tor$ and $\Ext$ on DG $A$-modules}
We begin with conceptual definitions of the differential $\Tor$ and $\Ext$ functors.  
Of course, we define $\Tor$ and $\Ext$ exactly as in (\ref{TorExt}) for graded $R$-algebras and their modules.  
These are then bigraded.  In bigrading $(p,q)$, $p$ is the homological degree, $q$ is the internal degree, 
and $p+q$ is the total degree.  The differential $\Tor$ and $\Ext$ are graded, not bigraded.

\begin{defn}\mylabel{DTordefn} Define the differential $\Tor$ and $\Ext$ functors by
\begin{equation}\label{TorExt4} 
\Tor_{*}^{A}(N,M) = H_{*}(N\otimes_A X)
\end{equation}
and 
\begin{equation}\label{TorExt5} 
\Ext^{*}_{A}(M,N) = H^{*}\Hom_{A}(X,N),
\end{equation}
where $X\rtarr M$ is a $q$-cofibrant approximation of the DG $A$-module $M$.  
\end{defn}

Conceptually, for $\Tor$, we are taking a derived 
tensor product $(-)\otimes_A M$ with respect to the $q$-model structure and then applying homology.
Similarly, for $\Ext$, we are taking a derived $\Hom$ functor $\Hom_A(M,-)$ and then applying homology.
We shall say very little more about $\Ext$ here, but the parallel should be clear.

Since any two $q$-cofibrant approximations of $M$ are $h$-equivalent over $M$, we can use any 
$q$-cofibrant approximation of $M$ in the definition.  Using \myref{qmodelA}, we see that the 
functor $(-)\otimes_A X$ preserves $q$-equivalences when $X$ is $q$-cofibrant.  This implies
that we can equally well derive the functor $N\otimes_A(-)$.  

\begin{lem}\mylabel{inv1} If $\be\colon Y\rtarr N$ is a $q$-cofibrant approximation of $N$ and 
$\al\colon X\rtarr M$ is a $q$-cofibrant approximation of $M$, then the maps 
\[ H(N\otimes_A X)  \ltarr H(Y\otimes_A X) \rtarr H(Y\otimes_A M) \]
induced by $\al$ and $\be$ are isomorphisms.
\end{lem}  

We have long exact sequences that are precisely analogous to the long exact sequences of
the classical $\Tor$ functors.  We defer the proof to \S\ref{diddle}.

\begin{prop}\mylabel{Torles}  Short exact sequences 
\[ 0\rtarr N'\rtarr N\rtarr N'' \rtarr 0 \]
of DG $A$-modules naturally induce long exact sequences
\[  \cdots \to \Tor_n^A(N',M) \rtarr \Tor_n^A(N,M) \rtarr \Tor_n^A(N'',M) \rtarr \Tor_{n-1}^A(N',M) \to \cdots. \]
\end{prop}

The functors $\Tor$ and $\Ext$ might well be denoted $q\Tor$ and $q\Ext$.  There are relative 
analogues $r\Tor$ and $r\Ext$.

\begin{defn}\mylabel{DTordefntoo} Define the relative differential $\Tor$ and $\Ext$ functors by
\begin{equation}\label{TorExt4R} 
r\Tor_{*}^{A}(N,M) = H_{*}(N\otimes_A X)
\end{equation}
and 
\begin{equation}\label{TorExt5R} 
r\Ext^{*}_{A}(M,N) = H^{*}\Hom_{A}(X,N),
\end{equation}
where $X\rtarr M$ is an $r$-cofibrant approximation of the DG $A$-module $M$.  
\end{defn}

\myref{inv1} and \myref{Torles} apply equally well to $r\Tor$, with the same proofs.
Probably the most standard calculational tool in differential homological algebra is
the bar construction, and we shall see both that it is intrinsic to the relative
functor $r\Tor$ and that its properties imply that $q\Tor$ and $r\Tor$ agree 
unexpectedly often. 

\subsection{Outline and conventions}\label{diddle}

We summarize the content of Part 2 and fix some conventions that we will use throughout. 
In \S\ref{projqcof}, we construct $q$-cofibrant approximations in terms of differential
graded projective resolutions, reinterpreting the early work of Cartan, Eilenberg, and Moore 
\cite{CE, EM, Moore} model theoretically.

We characterize $q$-cofibrant and $r$-cofibrant DG $A$-modules in \S\ref{GMStuff},
where we also place them in the more general cellular context of split DG $A$-modules.
Shifting gears, in \S\ref{GMStuff1} we start from the Eilenberg-Moore spectral sequence 
and relate resolutions to cofibrant approximations. We also show how the bar construction 
and matric Massey products fit into the picture there.

Finally, in \S\ref{GMStuff2}, we show how to construct $q$-cofibrant approximations
from classical projective resolutions of homology modules $H_*(M)$ over homology algebras $H_*(A)$ and 
how to apply the construction to make explicit calculations.  

\begin{uconv}  Since we shall be making more and more reference to homology as we proceed, we agree 
henceforward to abbreviate notation consistently by writing $HA$ and $HM$ instead of $H_*(A)$ and $H_*(M)$, following \cite{GM}. We sometimes regard these as bigraded, and then $H_q$ is understood to have
bidegree $(0,q)$.  When focusing on a specific degree, we write $H_n(M)$ as usual.  
\end{uconv}

To mesh the model categorical filtrations of cell complexes with the standard gradings in homological 
algebra, we must slightly change the filtration conventions on cell objects from Definitions
\ref{cellob} and \ref{cellob2}. There the convention is the standard one in model category theory that, 
for a relative cell complex $W\rtarr Y$, $F_{-1}Y =0$ and $F_0Y = W$. Then a cell complex 
$X$, such as $Y/W$, has $F_0X = 0$. It is harmless mathematically to change the convention to 
$F_{-2}Y = 0$ and $F_{-1}Y = W$, leading to the following convention on cell complexes $X$.

\begin{uconv} We agree to refilter cell complexes $X$ so that $F_{-1}X = 0$ and the non-trivial terms 
start with a possibly non-zero $F_0 X$.
\end{uconv}

In terms of classical homological algebra, $F_0 X$  relates to the $0$th term of projective resolutions, 
as we shall see, and that motivates the shift. We adopt this change throughout the rest of the paper.

\begin{unotn} For brevity of notation, we call enriched $\bF\cI_R$-cell 
complexes {\em $r$-cell complexes} henceforward, and we call their specialization
to ordinary cell complexes {\em $q$-cell complexes}.  Their filtrations are understood to
conform with the conventions just introduced.
\end{unotn}

By our variants of the SOA, every DG $A$-module admits a cofibrant approximation by a $q$-equivalent $q$-cell 
complex and by an $r$-equivalent $r$-cell complex, not just by a retract thereof.   The following proof illustrates how convenient that is.

\begin{proof}[Proof of \myref{Torles}] Let $\al\colon X\rtarr M$ be $q$-cofibrant approximation, where $X$
is a $q$-cell complex.  Then $X$ is isomorphic as an $A$-module to $A\otimes \bar{X}$ for a degreewise free 
graded $R$-module $\bar{X}$, hence $N\otimes_A X$ is isomorphic to $N\otimes \bar{X}$.
Thus the sequence
\[ 0\rtarr N'\otimes_A X\rtarr N\otimes_A X \rtarr N''\otimes_A X \rtarr 0 \]
of DG $R$-modules is isomorphic to the sequence
\[ 0\rtarr N'\otimes \bar X\rtarr N\otimes \bar X \rtarr N''\otimes \bar X \rtarr 0, \]
which is exact since $\bar{X}$ is degreewise free.  The resulting long exact sequence of homology groups 
gives the conclusion.
\end{proof}

\section{Projective resolutions and $q$-cofibrant approximations}\label{projqcof}

There is both tension and synergy between model category theory and classical homological algebra.
We explore the relationship in this section.  We first show that the classical projective resolutions 
of chain complexes, which are due to Cartan and Eilenberg \cite[\S XVII.1]{CE} and which generalize
directly to DG $R$-modules, are $q$-cofibrant approximations.    

Building on \cite{CE}, Moore \cite{Moore} developed projective resolutions of DG $A$-modules.
This is more subtle, but Moore found definitions that make the generalization transparently 
simple, as we shall recall.  We will show that his projective resolutions are also $q$-cofibrant
approximations. 

In Moore's work and throughout the early literature, there are bounded below
hypotheses on the DG algebras and modules.  These are not satisfied in the most interesting 
examples, which are bounded above with our grading conventions.  We avoid this condition wherever
possible.

\subsection{Projective classes and relative homological algebra}\label{projclassR}  
As we have already noted, the following notion of a projective class is the starting point
of relative homological algebra, as developed by Eilenberg and Moore \cite{EM}. It gives 
a general context for Moore's projective resolutions.  Much later, the notion also served
as the starting point for a model theoretic development of relative homological algebra
in work of Christensen and Hovey \cite{CH}. The notion is usually restricted to abelian 
categories, but it applies more generally.   

\begin{defn}\mylabel{projclass}
Let $p\colon E\rtarr M$ be a map in a category $\sM$ and let $P$ be an object of $\sM$.  
Say that $p$ is $P$-surjective or that $P$ is $p$-projective if $p_*\colon \sM(P,E)\rtarr \sM(P,M)$ 
is a surjection. A {\em projective class} $(\sP,\sE)$ in $\sM$ is a class $\sP$ of
objects and a class $\sE$ of maps such that 
\begin{enumerate}[(i)]
\item $\sE$ is the class of all maps that are $P$-surjective for all $P\in \sP$;
\item $\sP$ is the class of all objects that are $p$-projective for all $p\in \sE$;
\item for each object $M$ in $\sM$, there is a map $p\colon P\rtarr M$ with $P\in \sP$ and $p\in \sE$.
\end{enumerate}
\end{defn}

The notion of a projective class is useful in categories with kernels but not in general. 
The presence of kernels allows the construction of projective resolutions.  

\begin{rmk} The original definition of \cite[p.~5]{CE} was a little different, but it is 
essentially equivalent to \myref{projclass} in the presence of an initial object and kernels.  
The point is that (iii) then allows one to construct a map $P\rtarr K$ in $\sE$, where $K$
is the kernel of an arbitrary map $f\colon X\rtarr Y$.  From here, it is straightforward 
to use $(\sP,\sE)$ to construct projective resolutions of objects of $\sM$.
\end{rmk}

Projective classes are analogous to what one sees in model categories if
one considers cofibrant objects but does not introduce cofibrations in general. If 
$(\sW,\sC,\sF)$ is a model structure on $\sM$, $\sQ$ is the class of cofibrant 
objects, and $\sA\sQ$ is the class of acyclic cofibrant objects (those $X$ such that 
$\emptyset \rtarr X$ is an acyclic cofibration), then $(\sQ,\sW\cap\sF)$ and $(\sA\sQ,\sF)$ 
are candidates for projective classes in $\sM$. Certainly (ii) and (iii) are satisfied,
but there might be too few maps in $\sF$ for (i) to be satisfied: the lifting condition
against cofibrations might be more restrictive than just the lifting condition against
cofibrant objects.  Projective classes that are not parts of model categories appear
naturally, and their associated projective resolutions can often be interpreted model
categorically as cofibrant approximations.  We are not interested here in the general theory, 
but the examples that Cartan, Eilenberg, Mac\,Lane, and Moore focused on in 
\cite{CE, EM, MacHom, Moore} are directly relevant to our work.  

For a DG $R$-module $M$, let $B_n(M)\subset Z_n(M)\subset M_n$ be
the boundaries and cycles of $M$. Identifying $M_n/Z_n(M)$ with $B_{n-1}(M)$, we have exact sequences
\begin{equation}\label{ein}   0\rtarr B_n(M) \rtarr Z_n(M) \rtarr H_n(M) \rtarr 0
\end{equation}
\begin{equation}\label{zwei}  0\rtarr Z_n(M) \rtarr  M_n \rtarr B_{n-1}(M) \rtarr 0.
\end{equation}

\begin{defn}\mylabel{sprojdefn}  A DG $R$-module $P$ is {\em $s$-projective}\footnote{The $s$ stands for strong or strongly,
as in \cite{EM}; the term ``proper'' is also used.} if the $R$-modules $B_n(P)$ and $H_n(P)$, and therefore also the $R$-modules $Z_n(P)$,
$P_n/B_n(P)$, and $P_n$, are projective. Let $\sP_s$ denote the class of $s$-projective DG $R$-modules.   
\end{defn}

\begin{lem}\mylabel{projstruc} A DG $R$-module $P$ is $s$-projective if and only if it is isomorphic to a direct sum
over $n\in \bZ$ of DG $R$-modules $S^n_R\otimes H_n$ and $D^n_R\otimes B_{n-1}$ for projective $R$-modules 
$H_n$ and $B_{n-1}$.  Therefore, $s$-projective DG $R$-modules are $q$-cofibrant.  
\end{lem}
\begin{proof}  Clearly DG $R$-modules of the specified form are $s$-projective.  For the converse,
a splitting of the sequence (\ref{ein}) identifies $Z_n(P)$ with $H_n(P)\oplus B_n(P)$. A splitting
$\si\colon B_{n-1}(P) \rtarr P_n$ of (\ref{zwei}) then identifies $P_n$ with $Z_n(P) \oplus \si B_{n-1}(P)$.
The differential identifies $\si B_{n-1}(P)\subset P_n$ with $B_{n-1}\subset P_{n-1}$.
\end{proof}

\begin{defn} A map $p\colon E\rtarr M$ of DG $R$-modules is an {\em $s$-epimorphism} if 
$p_n\colon E_n\rtarr M_n$ and $p_n\colon Z_n(E)\rtarr Z_n(M)$ are epimorphisms for all $n$;
then $p_n\colon H_n(E)\rtarr H_n(M)$ and $p_n\colon B_n(E)\rtarr B_n(M)$ are also
epimorphisms for all $n$. Let $\sE_s$ denote the class of $s$-epimorphisms.
\end{defn}

\begin{prop}\mylabel{enough} The pair $(\sP_s,\sE_s)$ is a projective class in the category $\sM_R$.
\end{prop}
\begin{proof}  We must verify (i)-(iii) of \myref{projclass}.  If $P$ is $s$-projective, $p\colon E\rtarr M$ 
is an $s$-epimorphism, and $f\colon P\rtarr M$ is a map of DG $R$-modules, then we can lift $f$ to a map $\tilde{f}\colon P\rtarr E$ by lifting each $f_n\colon Z_n(P)\rtarr Z_n(M)$ to $Z_n(E)$ and lifting the restriction of $f$ to
$\si B_{n-1}(P) \subset P_n$, using the epimorphism $p\colon E_n\rtarr M_n$.  Since $S^n_R$ and $D^n_R$ are $s$-projective, 
a map that is $P$-surjective for all $P\in \sP$ is in $\sE$, which verifies (i). 

We next prove (iii). Thus let $M$ be any DG $R$-module.  For each $n$, choose epimorphisms 
$\et_n\colon B_n\rtarr B_n(M)$ and $\ze_n\colon H_n \rtarr H_n(M)$, where $B_n$ and $H_n$ are projective.  Let  
$Z_n = B_n\oplus H_n$ and define $\epz_n\colon Z_n\rtarr Z_n(M)$ to be $\et_n$ on $B_n$ and a lift of $\ze_n$ 
to a map $H_n\rtarr Z_n(M)$ on $H_n$.  Then define $P_n = Z_n \oplus B_{n-1}$ and define $\epz\colon P_n \rtarr M_n$ 
to be $\epz_n$ on $Z_n$ and a lift of $\et_{n-1}$ to a map $B_{n-1}\rtarr M_n$ on $B_{n-1}$.  Then $\epz\colon Z_n\rtarr Z_n(M)$  and $\epz\colon P_n\rtarr M_n$ are epimorphisms.  Define $d\colon P_n\rtarr P_{n-1}$ to be zero on $Z_n$ and the identity from $B_{n-1}\subset P_n$ to $B_{n-1}\subset P_{n-1}$.  Then $\epz$ is a map of DG $R$-modules and $\epz \in \sP_s$. Finally,
for (ii), if $M$ is $s$-projective, then a lift of the identity map of $M$ along $\epz$ displays $M$ as a retract of the
$s$-projective DG $R$-module $P$, and it follows that $M$ is $s$-projective.
\end{proof}

\begin{cor}\mylabel{acycproj} The class $\sA_q\sQ_q$ of $q$-acyclic $q$-cofibrant objects in $\sM_R$ coincides with the 
class $\sA_q\sP_s$ of 
$s$-projective complexes $P$ such that $H_*(P)=0$. 
\end{cor}
\begin{proof} Clearly $P$ is in $\sA_q\sQ_q$ if and only if $P$ is $p$-projective for all $p\in \sF_q$.  
Since $\sE\subset \sF_q$, $P$ is then in $\sP_s$.  Conversely, if $P$ is in $\sP_s$ and $H_*(P) = 0$, then
$P$ is in $\sA_q\sQ_q$ by \myref{projstruc}.
\end{proof}

\subsection{Projective resolutions are $q$-cofibrant approximations in $\sM_R$}\label{versus}

Projective resolutions relate the projective class $(\sP_s,\sE_s)$ to the class $\sQ_q$ of 
$q$-cofibrant $R$-chain complexes.   With our grading conventions, \cite[XVII.1]{CE} defines a 
projective resolution 
$\epz\colon P\rtarr M$ to be a right-half-plane bicomplex $P$ augmented over $M$ such that the induced 
chain complexes $H_{*,q}(P)$ and $B_{*,q}(P)$ are projective resolutions of $H_q(M)$ and $B_q(M)$. It follows that the induced chain complexes $Z_{*,q}(P)$ and $P_{*,q}$ are projective resolutions of $Z_q(M)$ and $M_q$. 

We construct a projective resolution $\epz\colon P\rtarr M$ of a DG $R$-module $M$ in the
usual way. Via the proof of \myref{enough}, we first construct an $s$-projective DG $R$-module $P_{0,*}$ 
and an $s$-epimorphism $\epz\colon P_{0,*}\rtarr M$ with kernel $K_{0,*}$.  Inductively, we construct an 
$s$-projective DG $R$-module $P_{p,*}$ and an $s$-epimorphism $P_{p,*} \rtarr K_{p-1,*}$ with 
kernel $K_{p,*}$, and we have the differential 
$$d\colon P_{p,*} \rtarr K_{p-1,*}\subset P_{p-1,*}.$$
It is a map of DG $R$-modules.  Then $\{P_{p,q}\}$ and the maps $\epz\colon P_{0,*}\rtarr M$ specify
a bicomplex  over $M$ with
vertical differential $d^0\colon P_{p,q} \rtarr P_{p,q-1}$, given by the differentials on the $P_{p,*}$, and horizontal differential
$d^1\colon P_{p,q} \rtarr P_{p-1,q}$. 
To ensure that $d^0d^1 + d^1d^0 =0$ we set $d^1 = (-1)^q d$ on $P_{p,q}$.\footnote{Warning: bicomplexes are symmetric structures. For purposes
of comparison with $q$-cell complexes we have {\em reversed} the roles of $p$ and $q$ from those they play in \cite{CE, EM, Moore}.}

This construction gives a projective resolution in the sense of \cite[XVII.1]{CE}, as we see by 
inspection of the proof of \myref{enough}.  This proves the first statement of the 
following result; it is \cite[XVII.1.2]{CE}, which gives details of the rest of the proof.
 
\begin{prop}\mylabel{projres}  Every DG $R$-module $M$ admits a projective resolution $P$.   
If $P$ and $Q$ are projective resolutions of $M$ and $N$ and $f\colon M\rtarr N$ is a map of DG
$R$-modules, then there is a map $\tilde f\colon P\rtarr Q$ of projective resolutions over $f$.   
If $\tilde f$ and 
$\tilde g$ are maps over homotopic maps $f$ and $g$, then $\tilde{f}$ and $\tilde{g}$ are homotopic.
\end{prop} 

The total complex of a bicomplex $\{P_{p,q}\}$ is the DG $R$-module $\TotP$ specified by $\TotP_n = \sum_{p+q=n} P_{p,q}$ with
differential $d^0+ d^1$.  If $\epz\colon P\rtarr M$ is a projective resolution, we continue to write 
$\epz\colon \TotP \rtarr M$ for the induced map of DG $R$-modules from the total complex of $P$ to $M$.

As a bicomplex, $P$ has two filtrations. We are more interested in the filtration by the homological degree $p$.  With it, $F_pP$ is the sum of the $P_{p-r,*}$ for $0\leq r\leq p$.  The filtration quotient $F_pP/F_{p-1}P$ is $P_{p,*}$, and its differential 
is rarely zero. However, we have the following key observation. 

\begin{lem}\mylabel{Indeed}  The total complex $\TotP$ of a projective resolution is $q$-cofibrant.
\end{lem}
\begin{proof}
We cannot apply \myref{qcofibrant} since we are not assuming that $P$ is bounded below.
However, the filtration quotients are $q$-cofibrant by \myref{projstruc}, hence 
the inclusions  $F_{p-1}P \rtarr F_p P$ are $q$-cofibrations 
by \myref{qcofibration}.  By induction, each $F_pP$ is $q$-cofibrant,
hence so is their colimit $P$.
\end{proof}

This is more surprising than it may look. The cellular filtration quotients 
$F_pP/F_{p-1}P$ of a $q$-cell complex $P$ are direct sums of sphere complexes $S^{q+1}_R$ 
and have differential zero.   Moreover, the attaching maps $S^q_R\rtarr F_{p-1}X$ 
can have components in filtration $F_{p-r}P$ where $r>1$, hence the differential on 
$F_pP$ can have components in $F_{p-r}P$ for $1 < r\leq p$.  In retrospect, the theory 
of \cite{GM, M, MN} that first motivated this paper starts from that insight.  \myref{Indeed} 
implies that the total complexes $\TotP$ of projective resolutions can be equipped both with a 
structure of bicomplex and with an entirely different filtration as a retract of a $q$-cell complex.  

\begin{thm}\mylabel{projcofR} A projective resolution $\epz\colon \TotP\rtarr M$ is a $q$-cofibrant 
approximation.\end{thm}
\begin{proof}
By construction $\epz\colon \TotP\rtarr M$ is a 
degreewise epimorphism and thus a $q$-fibration. By \myref{Indeed}, it suffices to show that $\epz\colon \TotP\rtarr M$ is $q$-acyclic.  There is an easy spectral sequence argument when $M$ and $P$ are bounded below. We will prove a generalization without any such hypothesis in \myref{projcofA} below, using a model theoretic approach. 
\end{proof}

For a DG $R$-module $N$ of right $R$-modules, 
we give $N\otimes P$ the bigrading  
\[ (N\otimes P)_{p,q} = \sum_{i+j = q}  N_i\otimes P_{p,j}. \]
If we filter by the internal degree $q$, we obtain a spectral sequence $E^r_{p,q}$ with differentials
$d^r\colon E^r_{p,q} \rtarr E^r_{p+r-1,q-r}$.  With our (perhaps peculiar) notations, the differential
$d^0$ is induced by the bicomplex differential $d^1$ on $P$, which gives a projective resolution $P_{*,j}$ 
of $M_j$ for each fixed $j$.  Therefore
\[ E^1_{p,q} = \Tor^R_{p,q}(N,M) = \sum_{i+j = q} \Tor_{p}^R(N_i,M_j)\] 
with differential $d^1$ induced by the bicomplex differential $d^0$.  Assuming that either $N$ or $M$ is a 
complex of flat $R$-modules, $E^2_{p,q}(N,M) = 0$ for $p>0$ and 
\begin{equation}\label{E2first}
 E^{2}_{0,q} = E^{\infty}_{0,q} = H_q(N\otimes M). 
\end{equation}

In Boardman's language \cite{Board}, this spectral sequence has entering differentials and it
need not converge.  In fact, \myref{Notsemi} gives a striking example where convergence fails.
In that example, $E^2_{0,q} = \bZ/2$ for all integers $q$ and yet the desired target is zero.
This is where boundedness hypotheses enter classically.
 
\begin{lem}\mylabel{cutesy} If $N$ or $M$ is degreewise $R$-flat and both are
bounded below, then
\[ (\id\otimes \epz)_*\colon H(N\otimes P)\rtarr H(N\otimes M) \]
is an isomorphism.  In particular, taking $N=R$, $\epz\colon P\rtarr M$ is a $q$-equivalence.
\end{lem}

Now consider the filtration on $N\otimes_R P$ induced by the homological degree $p$.  Here we have 
$$d^r\colon E^r_{p,q}\rtarr E^r_{p-r,q+r-1}$$ 
in the resulting spectral sequence, with $d^0$ induced by the differential on $N$ and the 
bicomplex differential $d^0$.  
Since the $P_{p,j}$ are all projective, we have 
\[ E^1_{p,q} = \sum_{i+j = q} H_i(N)\otimes H_j(P_{p,*}) \]
with differential induced by the bicomplex differential $d^1$.   
Since $H_*(P_{*,j})$ is a projective resolution of $H_j(M)$,
\begin{equation}\label{E2second} 
E^2_{p,q} = \Tor^R_{p,q}(HN,HM) = \sum_{i+j = q} \Tor^R_{p}(H_i(N),H_j(M)). 
\end{equation}
In Boardman's language \cite{Board}, this spectral sequence has exiting differentials and 
converges to $H_*(N\otimes P)$, without bounded below hypotheses.  In view of \myref{cutesy},
this gives the following version of the K\"unneth spectral sequence. 

\begin{thm}\mylabel{KunnSS}  If $N$ or $M$ is degreewise $R$-flat and both are bounded below,
the spectral sequence $\{E^r\}$ converges from 
$E^2_{*,*} = \Tor^R_{*,*}(HN,HM)$ to $H(N\otimes P)$.
\end{thm}

\subsection{The projective class $(\sP_s,\sE_s)$ in $\sM_A$}\label{projclassA}

This section is parallel to \S\ref{projclassR}.  It will lead us to $q$-cofibrant approximations in the next.
Recall the projective class $(\sP_s,\sE_s)$ of DG $R$-modules from \S\ref{projclassR}.

\begin{defn}\mylabel{projtoo}  A DG $A$-module $P$ is {\em $s$-projective} if it is a retract of $A\otimes Q$ for some
$s$-projective DG $R$-module $Q$. Let $\sP_s$ denote the class of $s$-projective DG $A$-modules. 
\end{defn}

Non-trivial retracts can appear, for example, if $A$ itself is the direct sum of sub DG $A$-modules.
\myref{projstruc} directly implies the following analogue.

\begin{lem}\mylabel{projstructoo} A DG $A$-module $P$ is $s$-projective if and only if it is isomorphic to a 
retract of a direct sum over $n\in \bZ$ of DG $A$-modules $A \otimes S^n_R\otimes H_n$ and $A \otimes D^n_R\otimes B_{n-1}$ for 
projective $R$-modules $H_n$ and $B_{n-1}$.  Therefore $s$-projective DG $A$-modules are $q$-cofibrant.  
\end{lem}

\begin{defn}\mylabel{epistoo}  A map $p\colon E\rtarr B$ of DG $A$-modules is an $s$-epimorphism if 
$\bU p$ is an $s$-epimorphism of DG $R$-modules.  Let $\sE_s$ denote the class of
$s$-epimorphisms of DG $A$-modules.  
\end{defn}

The action map $f \colon A\otimes M\rtarr M$
of any DG $A$-module $M$ is an $s$-epimorphism since the unit of the adjunction gives that
$\bU M$ is a retract of $\bU(A\otimes M)$.  Similarly, if $p\colon E\rtarr \bU M$ is an $s$-epimorphism
of DG $R$-modules, its adjoint $\tilde{p}\colon A\otimes E \rtarr M$ is an $s$-epimorphism
of DG $A$-modules.

\begin{prop}\mylabel{enoughtoo} The pair $(\sP_s,\sE_s)$ is a projective class in the category $\sM_A$.
\end{prop}
\begin{proof} We must verify (i)-(iii) of \myref{projclass}.  We first show that if $P$ is $s$-projective, $p\colon E\rtarr M$ 
is an $s$-epimorphism, and $f\colon P\rtarr M$ is a map of DG $A$-modules, then $f$ lifts to a map 
$\tilde{f}\colon P\rtarr E$.  Since this property is inherited by retracts, we may assume that $P = A\otimes Q$,
where $Q$ is an $s$-projective $R$-module.  Then the conclusion is immediate by adjunction from the analogue
for DG $R$-modules.  If a map $p\colon E\rtarr M$ is $P$-surjective for all $P\in \sP_s$, then, again by adjunction
from the examples $P = A\otimes Q$, $\bU p$ is an $s$-epimorphism. This verifies (i).  If a DG $A$-module $P$ is
$p$-projective for all $p\in \sE_s$ and $Q\rtarr \bU P$ is an $s$-epimorphism of DG $R$-modules, where $Q$ is
$s$-projective, then $P$ is a retract of $A\otimes Q$ and is thus $s$-projective. This verifies (ii). 
If $M$ is a DG $A$-module and $Q\rtarr \bU M$ is an $s$-epimorphism of DG $R$-modules where $Q$ is $s$-projective,
then its adjoint $A\otimes Q\rtarr M$ is an $s$-epimorphism of DG $A$-modules, verifying (iii).
\end{proof}

Exactly as in \myref{acycproj}, this has the following consequence.

\begin{cor}\mylabel{acycprojtoo} The class $\sA_q\sQ_q$ of $q$-acyclic $q$-cofibrant objects in $\sM_A$ coincides with 
the class $\sA_q\sP_s$ of $q$-acyclic $s$-projective DG $A$-modules $P$. 
\end{cor}

\begin{rmk}
Note that it is unreasonable to take $\sE_s$ to be the class of fibrations in a model structure on $\sM_A$ since 
$0 \to \bF S^n_R$ would then be an acyclic cofibration. 
\end{rmk}

The following result was used without proof when $A=R$, where it is elementary, but we make it explicit here.  
It is \cite[2.1]{Moore}.

\begin{lem}\mylabel{homology} If $N$ is a right DG $A$-module and $P$ is an $s$-projective left DG $A$-module, then
\[ H(N\otimes_A P) \iso HN\otimes_{HA} HP. \]
\end{lem}
\begin{proof} We may assume that $P = A\otimes Q$ where $Q$ is an $s$-projective DG $R$-module.
Then $N\otimes_A P\iso N\otimes Q$, hence $H(N\otimes_A P)\iso HN\otimes HQ$. On the other hand,
$H(A\otimes Q)\iso HA\otimes HQ$, hence $HN\otimes_{HA}HP \iso HN\otimes HQ$.
\end{proof}

\subsection{Projective resolutions are $q$-cofibrant approximations in $\sM_A$}\label{versustoo}

We mim\-ic \S\ref{versus}.  We ignore the retracts in \myref{projtoo} and use only $s$-projective DG $A$-modules
of the form $P=A\otimes Q$ for an $s$-projective $R$-module $Q$. Let us say that a sequence 
$\xymatrix@1{ L \ar[r]^{f} & M \ar[r]^{g} & N\\}$ of DG $A$-modules is $s$-exact if $f$ is the 
composite of an $s$-epimorphism $L\rtarr K$ and the inclusion $K\rtarr M$ of a kernel of $g$.
Then we can define a projective resolution of $M$ to be an $s$-exact sequence
\[ \cdots \rtarr P_{p,*} \rtarr P_{p-1,*} \rtarr \cdots \rtarr P_{1,*} \rtarr P_{0,*} \rtarr M \rtarr 0 \]
such that each $P_{p,*}$ is $s$-projective.  

We construct projective resolutions $P$ of
DG $A$-modules $M$ as in \S\ref{versus}.  Their terms are of the form $P_{p,*} = A\otimes Q_{p,*}$.  Here
$P_{p,q} = \sum_{i+j=q} A_i\otimes Q_{p,j}$. We first construct 
an $s$-projective DG $R$-module $Q_{0,*}$ and an $s$-epimorphism $Q_{0,*}\rtarr \bU M$. We take its 
adjoint to be
$\epz\colon P_{0,*}\rtarr M$, with kernel $K_{0,*}$.  Inductively, we construct an 
$s$-projective chain complex $P_{p,*}$ and an $s$-epimorphism $P_{p,*} \rtarr K_{p-1,*}$ with kernel $K_{p,*}$
in the same way, and we have the differential 
$$d\colon P_{p,*} \rtarr K_{p-1,*}\subset P_{p-1,*}.$$
It is a map of DG $A$-modules.  As in \S\ref{versus}, $\{P_{p,q}\}$ and the maps $\epz\colon P_{0,*}\rtarr M$ specify
a bicomplex over $M$ with vertical differential $d^0\colon P_{p,q} \rtarr P_{p,q-1}$ and horizontal differential
$d^1\colon P_{p,q} \rtarr P_{p-1,q}$. The differentials on the $P_{p,*}$ specify $d^0$.
To ensure that $d^0d^1 + d^1d^0 =0$ we set $d^1 = (-1)^q d$ on $P_{p,q}$.

 This proves the first statement of the following result, which is stated as 
\cite[2.1]{Moore}.  Moore leaves the rest of the proof to the reader, and so shall we.
 
\begin{prop}\mylabel{projrestoo}  Every DG $A$-module $M$ admits a projective resolution $P$.   
If $P$ and $Q$ are projective resolutions of $M$ and $N$ and $f\colon M\rtarr N$ is a map of DG $A$-modules, 
then there is a map $\tilde f\colon P\rtarr Q$ of projective resolutions over $f$.   If $\tilde f$ and 
$\tilde g$ are maps over homotopic maps $f$ and $g$, then $\tilde{f}$ and $\tilde{g}$ are homotopic.
\end{prop} 

We apply the general discussion of bicomplexes in \S\ref{versus}.  As before, we write $\TotP$ for the total complex of a projective resolution $\epz\colon P\rtarr M$. As a bicomplex, $P$ has two filtrations. We are again more interested in the filtration by the homological degree $p$.  With it, $F_pP$ is the sum of 
the $P_{p-r,*}$ for $0\leq r\leq p$.  The filtration quotient $F_pP/F_{p-1}P$ is $P_{p,*}$. Using 
\myref{projstructoo} and \myref{qcharcofibration}, the proof of \myref{Indeed} applies to give the 
following analogue.

\begin{lem}\mylabel{Indeedtoo}  The total complex $\TotP$ of a projective resolution is $q$-cofibrant.
\end{lem}

We conclude that the total complexes of $P$ of projective resolutions can be equipped 
both with a structure of bicomplex and with an entirely different filtration as a retract 
of a $q$-cell complex.  We shall now prove the following theorem, generalizing \myref{projcofR}.

\begin{thm}\mylabel{projcofA} A projective resolution $\epz\colon \TotP\rtarr M$ is a $q$-cofibrant 
approximation.
\end{thm}
\begin{proof}
By construction $\epz\colon \TotP\rtarr M$ is a 
degreewise epimorphism and thus a $q$-fibration. By \myref{Indeedtoo}, it suffices to show that $\epz\colon \TotP\rtarr M$ is 
a $q$-equivalence.  We might like to use the spectral sequence obtained by filtering by internal degree, but we have made no
boundedness assumption, hence that spectral sequence need not converge.   Instead, we construct a solution to any lifting problem
\[ \xymatrix{ \bF S^n_R \ar[r]^-p \ar[d] & \TotP \ar[d]^\epz \\ \bF D^{n+1}_R \ar[r]_-m \ar@{-->}[ur]^-\ell &  M.}\] 
By the adjunction $\xymatrix@1{\bF \colon \sM_R \ar@<.4ex>[r] & \ar@<.4ex>[l] \sM_A \colon \bU}$,
this is equivalent to solving the adjunct lifting problem for the underlying map of DG $R$-modules $\epz\colon\TotP\rtarr M$: \begin{equation}\label{eq:proj-resol-lifting-problem} \xymatrix{ S^n_R \ar[r]^-p \ar[d] & \TotP \ar[d]^\epz \\ D^{n+1}_R \ar[r]_-m \ar@{-->}[ur]^-\ell & M.}\end{equation} Thus we are free to work in the underlying context of DG $R$-modules. 

A commutative square of the form \eqref{eq:proj-resol-lifting-problem} corresponds to a pair of elements $m \in M_{n+1}$ and $p \in Z_n\TotP$ such that $d(m)= \epz(p)$. We may write $p = \sum_{i+j=n} p_{i,j}$, with $p_{i,j} \in P_{i,j}$. With this notation $\epz(p) = \epz(p_{0,n}) = d(m)$ and the condition $p \in Z_n\TotP$ holds if and only if
$$d^0(p_{i,j}) = (-1)^j d(p_{i+1,j-1}).$$ 
By the definition of a direct sum,  we must have $p_{i,n-i} =0$ for $i \gg 0$.

A solution to the lifting problem \eqref{eq:proj-resol-lifting-problem} is given by an element $\ell \in \TotP_{n+1}$ such that 
$\epz(\ell)=m$ and $d(\ell)=p$. Writing $\ell = \sum_{i+j=n+1} \ell_{i,j}$, the first condition is that $\epz(\ell_{0,n+1})=m$ and the second condition is that \[  d^0(\ell_{i,j+1}) + (-1)^jd(\ell_{i+1,j}) = p_{i,j} \qquad \forall\  i + j =n.\]  We must also ensure that we can choose $\ell_{i,j+1}=0$ for $i\gg 0$.

Since $\epz \colon P_{0,n+1} \to M_{n+1}$ is surjective, we may choose $\ell_{0,n+1} \in P_{0,n+1}$ 
such that $\epz(\ell_{0,n+1}) = m$. The next step is to find $\ell_{1,n} \in P_{1,n}$ such that 
\[(-1)^nd(\ell_{1,n}) = p_{0,n} - d^0(\ell_{0,n+1}).\] By the exactness of the resolution 
$P_{*,n} \to M_n$ at $P_{0,n}$, the calculation 
\[\epz (p_{0,n}) - \epz d^0(\ell_{0,n+1}) = d(m) - d\epz(\ell_{0,n+1}) = 0\] 
implies that this can be done.

Continuing inductively, we use the exactness of $P_{*,j} \to M_j$ at $P_{i,j}$ to find $\ell_{i+1,j} \in P_{i+1,j}$ such that \begin{equation}\label{eq:necessary-condition} (-1)^j d(\ell_{i+1,j}) = p_{i,j} - d^0(\ell_{i,j+1}).\end{equation} 
The calculation 
\[ d(p_{i,j}) - d d^0(\ell_{i,j+1}) = (-1)^{j+1}d^0(p_{i-1,j+1}) - d^0 d(\ell_{i,j+1})= \pm d^0d^0(\ell_{i-1,j+2}) =0\] 
implies that this can be done.

To show that the sum $\ell = \sum_{p+q=n+1} \ell_{p,q}$ is finite, we refine our construction for $p \gg 0$. Let $i$ be maximal such that $p_{i,j} \neq 0$. In this case, the right-hand side of \eqref{eq:necessary-condition} is in $Z_j (P_{i,*}) \subset P_{i,j}$ because $d^0(p_{i,j}) = (-1)^j d(p_{i+1,j-1}) = 0$. By the exactness of $Z_j(P_{*,j}) \to Z_jM$, we can choose $\ell_{i+1,j}$ to be a vertical cycle, so that $d^0(\ell_{i+1,j})=0$.  This implies that we may take $\ell_{p,q}=0$ for all $p> i+1$.
\end{proof}

Now consider $N\otimes_A P$ for a right DG $A$-module $N$.  We again have two spectral sequences. First consider the 
spectral sequence obtained by filtering by internal degree. When $N = A$, the $E^0$-term is a projective resolution 
of the $A$-module $M$.  Therefore, for general $N$, 
\[  E^1_{p,q} = \Tor^A_{p,q}(N,M), \]
where the ordinary $\Tor$ functor defined without use of the differentials on $A$, $N$, and $M$ is understood.
If the underlying $A$-module of either $N$ or $M$ is flat, then $E^1_{p,q}(N,M) = 0$ for $p>0$ and 
\begin{equation}\label{E2firsttoo}
 E^{2}_{0,q} = E^{\infty}_{0,q} = H_q(N\otimes_A M).
\end{equation}

Under boundedness assumptions, this gives a more familiar second proof and a
generalization of \myref{projcofA}. 

\begin{thm}\mylabel{cutesytoo} If $N$ or $M$ is $A$-flat and $A$, $N$, and $M$ are bounded below, then
\[ (\id\otimes \epz)_*\colon H(N\otimes_A \TotP)\rtarr H(N\otimes_A M) \]
is an isomorphism.  In particular, taking $N=A$, $\epz\colon \TotP\rtarr M$ is a $q$-equivalence.
\end{thm}

It follows that $H(N\otimes_A \TotP) = \Tor^A_*(M,N)$, hence that \myref{cutesytoo} has the
following reinterpretation. It generalizes \cite[p.~7-09]{Moore}.  We do not know how to prove it
using $q$-cofibrant approximations constructed by either the SOA or the methods of \cite{GM}, and
we will use it in our discussion of semi-flat DG $A$-modules and the bar construction.

\begin{cor}\mylabel{myfav}  If $M$ or $N$ is $A$-flat and $A$, $N$, and $M$ are bounded below, then
\[ \Tor^A_*(N,M) \iso H(N\otimes_A M).  \]
\end{cor}

Now consider the induced homological filtration on $N\otimes_A P$. Since the $P_{p,*}$ are all $s$-projective, 
\myref{homology} applies, and we see that $HP$ is a projective resolution of the $HA$-module $HM$. 
Therefore
\begin{equation}\label{E2secondtoo}
 E^2_{p,q} = \Tor^{HA}_{p,q}(HN,HM).
\end{equation}

We can think of this as a generalized K\"unneth spectral sequence since we now have the following
analogue of \myref{KunnSS}.  
\begin{thm}\mylabel{KunnSSA}  If $N$ or $M$ is $A$-flat and $A$, $N$, and $M$ are bounded below, the 
spectral sequence $\{E^r\}$ converges from 
$E^2_{*,*} = \Tor^{HA}_{*,*}(HN,HM)$ to $H(N\otimes_A M)$.
\end{thm}

In general, without the flatness or bounded below hypotheses, the spectral sequence converges to 
$H(N\otimes_A P)$.  Since $H(N\otimes_A P)= \Tor^A_*(N,M)$, this gives  a version of the 
Eilenberg-Moore spectral sequence. 

\section{Cell complexes and cofibrant approximations}\label{GMStuff} 

In this section, we describe ways of recognizing cofibrant DG $\sM_A$-modules when we see them.  
To do this, we first give characterizations of $q$- and $r$-cofibrant objects and cofibrations
and then develop a general cellular framework, starting from what we call split DG $A$-modules.
We focus on those model categorical cell complexes whose filtrations relate to the degrees of
flat or projective resolutions.  This should be viewed as analogous to singling out the CW complexes among
the cell complexes seen in the standard $q$-model structure on topological spaces.  However, it is 
considerably more subtle in that the relevant filtrations need {\em not} be the filtrations of model 
theoretic cell complexes.

\subsection{Characterization of $q$-cofibrant objects and $q$-cofibrations}\label{qchar}

The goal of this section is to give more explicit descriptions of the $q$-cofibrant objects and the $q$-cofibrations in $\sM_A$. 
By \myref{qmodelA}, we know that $q$-cofibrations are retracts of relative $q$-cell complexes, but we want a more tractable characterization analogous to Propositions \ref{qcofibrant} and \ref{qcofibration}.  
The results here will serve as models for analogous results about the $r$-model structure. 

\begin{defn}\mylabel{qsemifree} A DG $A$-module $X$
is \emph{$q$-semi-projective} if its underlying $A$-module is projective 
and if the DG $R$-module $\Hom_A(X,Z)$ is $q$-acyclic for all $q$-acyclic 
DG $A$-modules $Z$.  
\end{defn}

\begin{defn} 
A monomorphism $i\colon W\rtarr X$ of DG $A$-modules is a \emph{$q$-semi-projective extension}
if $X/W$ is $q$-semi-projective. Note that the extension is then $A$-split.  
\end{defn}

The following observations are immediate from the definitions.

\begin{lem}\mylabel{semifreeretract} A retract of a $q$-semi-projective $A$-module is $q$-semi-projective.
A retract of a $q$-semi-projective extension is a $q$-semi-projective extension.
\end{lem}

\begin{prop}\mylabel{qcharcofibration0}  If a map $i\colon W\rtarr Y$ of DG $A$-modules 
is a $q$-semi-projective extension, then it is a $q$-cofibration. In particular, a 
$q$-semi-projective $A$-module $X$ is $q$-cofibrant.
\end{prop}
\begin{proof}  Let $X=Y/W$ and assume that $X$ is $q$-semi-projective.  Let $p\colon E\rtarr B$ be a
$q$-acyclic $q$-fibration.  We must find a lift $\la$ in any lifting problem
\[ \xymatrix{
W\ar[r]^-g \ar[d]_i & E \ar[d]^p\\
Y \ar@{-->}[ur]^{\la} \ar[r]_-f & B.\\} \]
Since $X$ is projective, we can write $Y=W\oplus X$ as $A$-modules, and  
we can then write the differential on $Y$ in the form
\[ d(w,x)= (d(w)+t(x),d(x)) \]
where $t$ is a degree $-1$ map of $A$-modules, so that $t(ax) = (-1)^{deg\, a}t(x)$,
such that $dt+td = 0$.  The first formula is forced by the assumption that $d$ on
$Y$ is a degree $-1$ map of $A$-modules, and the second is forced by $d^2=0$. 
We write $f= f_1+f_2$, where $f_1\colon W\rtarr B$ and $f_2\colon X\rtarr B$,
and we write $\la = \la_1 + \la_2$ similarly.  We can and must define $\la_1 = g$
to ensure that $g = \la i$.  We want $p\la_2(x) = f_2(x)$ and 
\[ d\la_2(x) = \la d(0,x) = \la (t(x),d(x)) = gt(x) +\la_2 d(x). \]
Since $X$ is a projective $A$-module and $p$ is an epimorphism of $A$-modules, there 
is a map $\tilde{f_2}\colon X\rtarr E$ of  $A$-modules, but not necessarily DG $A$-modules, 
such that $p\tilde{f}_2 = f_2$.  The map $\tilde{f}_2$ is a first approximation to the required map $\la_2$.

Let $Z=\ker(p)$. Since $p$ is a $q$-equivalence, $Z$ is $q$-acyclic. 
Since $X$ is $q$-semi-projective, $\Hom_A(X,Z)$ is $q$-acyclic. Define $k\colon X\rtarr E$ by
\[ k = d \tilde{f}_2 -\tilde{f_2} d - gt. \]
We claim that $pk =0$, so that $k$ may be viewed as a map $X\rtarr Z$ of degree $-1$. 
To see this, note that $df = fd$ implies $df_2 = f_1t + f_2d$. 
Since $pd=dp$, 
\[ pk = dp\tilde{f}_2 -p\tilde{f_2}d -pg t = df_2-f_2d - f_1t = 0. \]
Moreover, 
\[ dk+kd = -d\tilde{f_2}d -dgt + d\tilde{f}_2d -gtd = 0, \]
so that $k$ is a cycle of degree $-1$ in $\Hom_A(X,Z)$. Therefore $k$ is a boundary.
Thus there is a degree $0$ map of $A$-modules $\ell\colon X\rtarr Z\subset E$ such that 
$d\ell - \ell d = k$. The map $\la_2 = \tilde{f}_2 - \ell$ is as required. 
\end{proof}

To obtain a converse to the theorem, we use a definition that encodes a reformulation
and generalization of the notion of a $q$-cell complex.

\begin{defn}\mylabel{qsplit} A \emph{$q$-split filtration} of a DG $A$-module $X$ is an 
increasing sequence $\{F_p X\}$ of DG $A$-submodules such that $F_{-1} X = 0$,  
$X = \cup_p  F_p X$, and each  $F_pX/F_{p-1}X$ is isomorphic as a DG $A$-module 
to $A\otimes K_p$ for some degreewise free DG $R$-module $K_p$.  Then 
the inclusions $F_{p-1} X\rtarr F_{p}X$ are $A$-split (but not DG $A$-split).  
The filtration is \emph{cellularly $q$-split} if the differential on each $K_p$ is zero.  
\end{defn}  

\begin{lem}\mylabel{cellqsplit} The cellular filtration of a $q$-cell complex is cellularly $q$-split.
\end{lem}
\begin{proof} 
This holds since $F_pX/F_{p-1}X$ is a direct sum 
of sphere DG $A$-modules $A\otimes S^n_R$ and is thus of the form 
$A\otimes V_p$ for a free $R$-module $V_p$ with zero differential. 
\end{proof}

\begin{rmk} Even if we weaken the requirement on the quotients  $F_pX/F_{p-1}X$
by allowing them to be retracts of DG $A$-modules $A\otimes K_p$ such that the $K_p$ are
degreewise projective DG $R$-modules, it is not true that the induced filtration $W\cap F_pX$ 
on a retract $W$ of $X$ is $q$-split.
\end{rmk} 
 
\begin{rmk}  The term ``semi-free'' is sometimes used in the literature for a DG $A$-module
with a cellularly $q$-split filtration.  As we shall see, these are essentially the same
as the $q$-cell complexes. 
\end{rmk}

\begin{prop}\mylabel{qcharcofibrant0} If a DG $A$-module $X$ admits a cellularly $q$-split 
filtration or if $X$ is bounded below and admits a $q$-split filtration, then $X$ is $q$-semi-projective.
\end{prop}
\begin{proof}  Assume that $X$ has a $q$-split filtration.
Successive splittings of filtration subquotients imply
that $X$ is isomorphic as an $A$-module (but not as a DG $A$-module) to $\oplus F_{p}X/F_{p-1}X$.  
Therefore $X$ is $A$-free.  More generally, each $X/F_pX$ splits correspondingly and we have $A$-split 
short exact sequences of DG $A$-modules
\[ 0\rtarr F_{p}X/F_{p-1}X \rtarr X/F_{p-1}X \rtarr X/F_{p}X \rtarr 0. \]
These give rise to short exact sequences of chain complexes 
\[  \Hom_A(X/F_pX,Z) \rtarr \Hom_A(X/F_{p-1}X,Z) 
\rtarr \Hom_A(F_{p}X/F_{p-1}X,Z). \] 
Observe that $\Hom_A(A\otimes K,Z)\iso  \Hom(K,\bU Z)$ for a DG $R$-module $K$ and a DG $A$-module $Z$.
Now let $Z$ be $q$-acyclic. We claim that each $\Hom_A(F_pX/F_{p-1}X, Z)$ is $q$-acyclic under either
of our hypotheses.  If $K$ is degreewise projective with zero differential, then 
$\Hom(K, \bU Z)$ is $q$-acyclic since the functor $\Hom(-,\bU Z)$ 
converts direct sums to cartesian products and since $\Hom(R,\bU Z)\iso Z$.  
 This implies the claim when the filtration on $X$ is cellularly
$q$-split. If $X$ is bounded below, then each $K_p$ is bounded below. By \myref{qcofibrant}, 
each $K_p$ is therefore $q$-cofibrant or equivalently $q$-semi-projective in $\sM_R$.  In 
particular, $\Hom(K_p,\bU Z)$ is $q$-acyclic and thus again each $\Hom_A(F_pX/F_{p-1}X, Z)$ is $q$-acyclic.

By the long exact homology sequences of our short exact sequences of chain complexes, each map
\[ H_*(\Hom_A(X/F_pX,Z)) \rtarr H_*(\Hom_A(X/F_{p-1}X,Z)) \] 
is an isomorphism.  It is not obvious that this implies that $H_*(\Hom_A(X,Z))=0$, but it does, by an application of Boardman's \cite[7.2]{Board}.  In detail, with 
\[ D^1_{p,q} = H_{p+q}(\Hom_A(X/F_{p-1}X,Z)) \ \ \ \text{and}\ \ \ 
   E^1_{p,q} =H_{p+q}(\Hom_A (F_pX/F_{p-1}X,Z)),\]   
our long exact sequences give an exact couple, and it gives rise to a right half-plane spectral sequence $E^r_{p,q}$ with differentials $d^r\colon E^r_{p,q}\rtarr E^r_{p+r,q-r-1}$ and with $E^2 = 0$. In Boardman's language, since we clearly have that $\lim_p D^1_{p,*-p} =0$ and $RE_{\infty}=0$ (see \cite[pp.~65, 67]{Board}), the spectral sequence converges conditionally to $\colim H_* D^1_{p,*-p}$, which is realized at $p=1$ by $H_*(\Hom_A(X,Z))=0$.  Applying \cite[7.2]{Board} to compare our spectral sequence to the spectral sequence of the identically zero exact couple, we see that $H_*(\Hom_A(X,Z))=0$.
\end{proof}

Note that requiring $X$ to be bounded below implicitly requires $A$ to be bounded below.
We put things together to prove the following results.

\begin{thm} \mylabel{qcharcofibrant}  Consider the following conditions on a DG $A$-module $X$.
\begin{enumerate}[(i)] 
\item $X$ is $q$-semi-projective.
\item $X$ is $q$-cofibrant.
\item $X$ is a retract of a DG $A$-module that admits a cellularly $q$-split filtration.
\item $X$ is a retract of a DG $A$-module that admits a $q$-split filtration.
\end{enumerate}
Conditions (i), (ii), and (iii) are equivalent and imply (iv). Moreover, if $X$ is bounded below, 
then (iv) implies (i). 
\end{thm}
\begin{proof}
\myref{qcharcofibration0} shows that (i) implies (ii), \myref{cellqsplit} implies
that (ii) implies (iii), and (iii) trivially implies (iv).  By \myref{semifreeretract}
and \myref{qcharcofibrant0}, (iii) and if $X$ is bounded below (iv) imply (i). 
\end{proof}

\begin{rmk}\mylabel{qcharcofibrationrmk}
In view of \myref{projectives}, the equivalent conditions (i), (ii), and (iii) are strictly stronger than (iv). This should be contrasted with the analogous result for the $r$-model structure, \myref{rcharcofibrant} below.
\end{rmk}

\begin{thm} \mylabel{qcharcofibration} A map $W\rtarr Y$ of DG $A$-modules is a $q$-cofibration
if and only if it is a monomorphism with $q$-cofibrant cokernel.
\end{thm}
\begin{proof} The forward implication is evident and the reverse implication holds
by Theorems \ref{qcharcofibration0} and \ref{qcharcofibrant}. 
\end{proof}

\subsection{Characterization of $r$-cofibrant objects and $r$-cofibrations}\label{rchar}

This section is parallel to \S\ref{qchar}.  Its goal is to give more explicit descriptions of the $r$-cofibrant objects and $r$-cofibrations in $\sM_A$.  By \myref{enrichedDGAcomparison}, these are retracts of enriched 
$r$-cell complexes, but we want a more tractable characterization.

\begin{defn}\mylabel{rsemifree} A DG $A$-module $X$
is \emph{$r$-semi-projective} if its underlying $A$-module is relatively projective and if 
$\Hom_A(X,Z)$ is a $q$-acyclic DG $R$-module for any $r$-acyclic DG $A$-module $Z$.  
\end{defn}

\begin{defn} 
An $R$-split monomorphism $i\colon W\rtarr X$ of DG $A$-modules is an \emph{$r$-semi-projective extension}
if $X/W$ is $r$-semi-projective. By \myref{proj}, the extension is then $A$-split.  
\end{defn}

\begin{lem} A retract of an $r$-semi-projective $A$-module is $r$-semi-projective.
A retract of an $r$-semi-projective extension is an $r$-semi-projective extension.
\end{lem}

\begin{prop}\mylabel{rcharcofibration0}  If a map $i\colon W\rtarr X$ of DG $A$-modules 
is an $r$-semi-projective extension, then it is an $r$-cofibration. In particular, an $r$-semi-projective
$A$-module $X$ is $r$-cofibrant.
\end{prop}
\begin{proof}
Changing $q$ to $r$ and projective to relatively projective, the argument is the same as 
the proof of \myref{qcharcofibration0}.
\end{proof}

Just as $r$-cell complexes generalize $q$-cell complexes, we have the following generalization
of a $q$-split filtration.

\begin{defn} An {\em $r$-split filtration} of a DG $A$-module $X$ is an increasing sequence of $R$-split inclusions $F_{p-1} X\rtarr F_{p}X$ of DG $A$-submodules such that $F_{-1}X = 0$,  $X = \cup_p  F_p X$, 
and each $F_pX/F_{p-1}X$ is isomorphic as a DG $A$-module to a direct summand of $A\otimes K_p$ for 
some DG $R$-module $K_p$. By \myref{proj} applied to the $R$-split quotient maps $F_{p}X\rtarr F_pX/F_{p-1}X$, the inclusions $F_{p-1} X\rtarr F_{p}X$ are $A$-split (but not DG $A$-split). The filtration is 
{\em cellularly $r$-split} if the differential on each $K_p$ is zero.
\end{defn}  

By the same proof as that of \myref{cellqsplit}, this generalizes $r$-cell complexes.

\begin{lem}\mylabel{cellrsplit} The cellular filtration of an $r$-cell complex is cellularly $r$-split.
\end{lem}

The following result is considerably stronger than its analogue \myref{qcharcofibrant0}.

\begin{thm}\mylabel{rcharcofibrant0} If a DG $A$-module $X$ admits an $r$-split 
filtration then $X$ is $r$-semi-projective.
\end{thm}
\begin{proof}
The argument is exactly like the proof of \myref{qcharcofibrant0}.
The key change is that $\Hom_A(A\otimes K,Z)\iso \Hom(K,\bU Z)$ is 
$q$-acyclic for any DG $R$-module $K$, not necessarily degreewise $R$-projective, 
since $\bU Z$ is $r$-acyclic and thus chain homotopy equivalent to $0$.
This eliminates the need for a bounded below hypothesis.
\end{proof} 

\begin{thm} \mylabel{rcharcofibrant}  The following conditions on a DG $A$-module $X$ are equivalent.
\begin{enumerate}[(i)] 
\item $X$ is $r$-semi-projective.
\item $X$ is $r$-cofibrant.
\item $X$ is a retract of a DG $A$-module that admits a cellularly $r$-split filtration.
\item $X$ is a retract of a DG $A$-module that admits an $r$-split filtration.
\end{enumerate}
\end{thm}
\begin{proof}
\myref{rcharcofibration0} shows that (i) implies (ii), \myref{cellrsplit} implies
that (ii) implies (iii), and (iii) trivially implies (iv).  By \myref{rcharcofibrant0}, 
(iv) implies (i). 
\end{proof}

\begin{thm}\mylabel{rcharcofibration} A map $W\rtarr Y$ of DG $A$-modules is an $r$-cofibration
if and only if it is an $R$-split monomorphism with $r$-cofibrant cokernel.
\end{thm}
\begin{proof} The forward implication is evident and the reverse implication holds
by Theorems \ref{rcharcofibration0} and \ref{rcharcofibrant}. 
\end{proof}

\subsection{From $r$-cell complexes to split DG $A$-modules}\label{splitone}

The following definition combines \cite[1.2 and 1.4]{GM}. It is implicit in 
\cite{M}.\footnote{\cite{M} was submitted in 1967, the same year that model categories first appeared \cite{Quillen}.}  It specifies a generalized variant of the notion of an $r$-split 
filtered DG $A$-module, as we shall see. The generalization will allow explicit descriptions of
cofibrant approximations that do not come in nature as retracts of $q$- or $r$-cell complexes.
We now focus more on the splitting than the filtration since that gives us more precise calculational
control. Up to minor streamlining, we adopt the terminology of \cite{GM}.

\begin{defn}\mylabel{ancdef} A DG $A$-module $X$ is {\em split} if the following properties hold. As an $A$-module,
\[ X = \sum_{p\geq 0}  A\otimes \bar{X}_{p,*} \]
for a sequence of graded $R$-modules $\bar{X}_{p,*}$ (not DG $R$-modules) graded so that the component of 
$\bar{X}_{p,*}$ in degree $p+q$ is $\bar{X}_{p,q}$.  Then 
\[  X_n = \sum_{i+p+j =n} A_i\otimes \bar{X}_{p,j}.  \]
We view $\bar{X}$ as a bigraded $R$-module, and then $X$ itself is bigraded by
\[  X_{p,q} = \sum_{i+j = q} A_i\otimes \bar{X}_{p,j}. \]
We require $X$ to be a filtered DG $A$-module with 
\[ F_pX = \sum_{0\leq k\leq p} A\otimes \bar{X}_{k,*}. \]
Then the differential on $X$ necessarily has the form
\begin{equation}\label{diff1}
  d = \sum_{r\geq 0} d^r, \ \  d^r\colon X_{p,q} \rtarr X_{p-r,q+r-1}, \ \ \text{where} \sum_{i+j=r} d^id^j = 0. 
\end{equation}
Since $X$ is a DG $A$-module, it follows that
\begin{equation}\label{diff2}
d^0(ax) = d(a)x + (-1)^{deg\, a}ad^0(x)  \ \ \text{and} \ \ d^r(ax) = (-1)^{deg\, a} ad^r(x) \ \ \text{for}\ \ r\geq 1, 
\end{equation}
where $a\in A$ and $x\in \bar{X}$. We say that $X$ is \emph{cell-like} if $d^0= 0$ on $\bar{X}$.  We say that $X$ is {\em distinguished} if it is cell-like and each $\bar{X}_{p,q}$ is a free $R$-module.
\end{defn}

\begin{ex}
The total complex $\TotP$ of a projective resolution in the sense of \S\ref{versustoo} is a split DG $A$-module, 
but it is not cellular in general. 
\end{ex}

It is no accident that the $d^r$ look like differentials in a spectral sequence, as we shall see in \S\ref{EMSS}.
It is tempting to require $d^0(\bar{X})=0$ in the definition of split, but that would rule out the bar construction and projective resolutions; see \S\ref{barsec} and
\S\ref{versustoo}. We cannot resist inserting the following quotes from \cite[pp.~3--4]{GM} about split DG $A$-modules.
``These objects are precisely the most general filtered DG $A$-modules that can be expected to be
of computational value. $\dots$ For historical reasons, differential homological algebra has been developed
using only those split objects such that $d^r = 0$ for $r>1$ ($d^0$ and $d^1$ are usually called the
`internal' and `external' differentials).  This restriction is unnecessary and, in our view, undesirable.''  
We now see that use of multicomplexes, as defined by Wall \cite{Wall}, is dictated by model category theory.

\begin{prop}\mylabel{neatojet} An $r$-cell complex $X$ in $\sM_A$ has a canonical structure 
of a cell-like split DG $A$-module.
The $q$-cell complexes $X$ are characterized as those $r$-cell complexes that are distinguished when 
considered as split DG $A$-modules.
\end{prop}
\begin{proof}  Our convention is that $F_{-1}X=0$. We first note that the splitting
\[ F_{p+1} X \iso F_{p} X \oplus F_{p+1}X/F_{p}X, \ \ p\geq 0, \]
of underlying $A$-modules is canonical, although not functorial.  The inclusions 
$i_n\colon S^{n-1}_R\subset D^{n}_R$ 
have the obvious retractions $r_n$ of graded $R$-modules that send the copy of $R$ in 
degree $n$ to $0$.  Applying $\bF$ and tensoring with
$R$-modules $V_i$, there result canonical retractions of all of the canonical inclusions
\[   \sum_i  \bF S^{n_i-1}_R \otimes V_i \rtarr \sum_i  \bF D^{n_i}_R \otimes V_i.  \]
For $q$-cell complexes, we take all $V_i$ to be $R$. For each $p$, we have such a
canonical inclusion $i_p\colon C_p\rtarr D_p$ with a retraction $r_p$ and, for some
attaching map $j_p$ of DG $A$-modules, we have a
pushout square in the diagram  
\[ \xymatrix{
C_p \ar[d]_{i_p} \ar[r]^{j_p}  &  F_{p}X \ar[d] \ar@{=}[ddr] & \\
D_p \ar[r] \ar[drr]_{j_pr_p} & F_{p+1} X \ar@{-->}[dr] & \\
& & F_{p} X.\\} \]
The dotted arrow is given by the universal property of pushouts, and its kernel maps isomorphically 
onto $F_{p+1}X/F_{p}X$.  This gives the promised canonical splitting. 

Since $F_{p}X/F_{p-1}X$ is relatively $A$-free for $p\geq 0$, we can write it as $A\otimes \bar{X}_{p,*}$ 
as an $A$-module, ignoring the differential.  Specifying the bigrading as in \myref{ancdef}, we see that $X$ is 
indeed a split DG $A$-module.   To see that it is cell-like, consider
the generating $R$-module $D^{n}_R\otimes V$ of a cell mapping into $F_{p+1}X$.  Since its boundary 
$S^{n-1}_R\otimes V$ maps into $F_{p}X$, $d$ sends the image of $D^{n}_R\otimes V$ into $F_{p}X$, 
hence $d^0(\bar{X}) =0$. 

Now suppose given a distinguished split DG $A$-module $X$, so that each $\bar{X}_{p,q}$ a free $R$-module.
Let $\{x_i\}$ be an $R$-basis for $\bar{X}$.  For $x_i\in \bar{X}_{p+1,q}$, let $y_i$ be a basis element
for a copy of $S^{p+q}_R$, and let $C_p$ be the direct sum of the $\bF S^{p+q}_R$ for those 
$y_i$ of bidegree $(p,q)$ for some $q$.  Define $j_p\colon C_p\rtarr F_{p} X$ by $j_p(y_i) = d(x_i)$.
Then it is easy to see that the diagram
\[ \xymatrix{
C_p \ar[d]_{i_p} \ar[r]^{j_p}  &  F_{p}X \ar[d]  \\
D_p \ar[r] & F_{p+1} X  \\} \]
is a pushout, showing that $X$ is a $q$-cell complex with the $j_p$ as attaching maps.  The converse is clear.
\end{proof}

It is not clear whether or not every cell-like split DG $A$-module arises this way from an 
$r$-cell complex, but we expect not.

\subsection{From relative cell complexes to split extensions}\label{splittwo}
In a less obvious way, \cite{GM} also considers relative
cell complexes  $W\rtarr Y$.  In effect, it shows that they are essentially the same 
thing as maps $X\rtarr M$ out of cell complexes.  To see this, we first extend our two 
notions of semi-projective extensions.

\begin{defn}\mylabel{splitext} A \emph{split extension} is an $R$-split monomorphism 
$i\colon W\rtarr Y$ of DG $A$-modules such that $X=Y/W$ is a cell-like split DG $A$-module. 
 Then the quotient map  $Y\rtarr X$ is $R$-split, hence $i$ is $A$-split by \myref{proj}.
Therefore the underlying $A$-module of $Y$ is isomorphic to $W\oplus X$.  
Fixing the splitting, the differential 
on $Y$ necessarily has the form
\[ d(w,x) = (d(w) + \be(x), d(x)) \ \ \text{for} \ \ w\in W\ \ \text{and} \ \ x\in X, \]
where $\be\colon X\rtarr W$ is a degree $-1$ map of DG $A$-modules, meaning that $\be$ maps
$X_n$ to $W_{n-1}$ and satisfies 
\begin{equation}\label{signsick}
 d\be = -\be d\ \ \text{and} \ \ \be(ax) = (-1)^{deg\, a}a\be(x) \ \ \text{for} \ \ a\in A\ \  \text{and} \ \  x\in X.
\end{equation}
These formulas are forced by $d^2=0$ and the Leibniz formula
\[ d(aw) = d(aw,ax) = d(a)(w,x) + (-1)^{deg\, a}ad(w,x). \]
Moreover, $Y$ is a filtered DG $A$-module with
\[ F_{-1} Y = W \ \ \text{and} \ \ F_pY = W\oplus F_pX \ \ \text{for} \ \ p\geq 0.\]
Observe that $i\colon W\rtarr Y$ determines and is determined by $\be\colon X \rtarr W$.  
 We call $Y$ the split extension determined by $\be$.
\end{defn}

The following model theoretic interpretation is immediate from the definitions,
Theorems \ref{qcharcofibration} and \ref{rcharcofibration}, and
Theorems \ref{qcharcofibrant} and \ref{rcharcofibrant}.  

\begin{prop} Let $i\colon W\rtarr Y$ be a split extension with quotient $X$. If $X$ is an $r$-cell complex, 
then $i$ is an $r$-semi-projective extension and is thus an $r$-cofibration.  If $X$ is a $q$-cell $R$-module, 
then $i$ is a $q$-semi-projective extension and is thus a $q$-cofibration.  
\end{prop}

To relate split extensions to maps $X\rtarr M$ (of degree $0$), we use a construction suggested by (\ref{signsick}).  
For any integer $q$, we have the usual $q$th suspension functor  $\SI^q\colon \sM_A\rtarr\sM_A$
defined by $\SI^q M = M\otimes S^q_R$.  It is an isomorphism of categories with inverse $\SI^{-q}$. 
We introduce a signed variant of $\SI^{-1}$. 

\begin{defn}  Define an isomorphism of categories $\UP\colon \sM_A\rtarr \sM_A$ by letting 
$(\UP M)_n = M_{n+1}$, writing elements in the form $\overline{m} \in \UP M$ for $m\in M$.  Define 
$d(\overline m) = - \overline{d(m)}$ and define the action of $A$ by $a\overline{m} = (-1)^{deg\, a}\overline{am}$.  
A quick check of signs shows that $\UP M$ is a DG $A$-module.  For a map $\ph\colon M\rtarr N$ of DG $A$-modules, 
define a map $\UP \ph\colon \UP M\rtarr \UP N$ of DG $A$-modules by $(\UP f)_n = f_{n+1}$. 
\end{defn}

Observe that a map $\al\colon X\rtarr M$ of DG $A$-modules can be identified with the degree $-1$ 
map $\be \colon X\rtarr \UP M$ of DG $A$-modules specified 
by $\be(x) = \overline{\al(x)}$.  The following is now a conceptual version of 
\cite[1.1]{GM}, which was the ad hoc starting point of \cite{GM}.  It constructs a split
extension from a map with domain $X$. Since every DG $A$-module $W$ is of the from $\UP M$
for some $M$, \myref{splitext} gives the inverse construction of a map $\al$ with 
domain $X$ from a split extension and thus from a relative cell complex.

\begin{defn}\mylabel{ancdef2}  For a map $\al\colon X\rtarr M$ of DG $A$-modules, where $X$ is a cell-like split
DG $A$-module, let $i\colon \UP M\rtarr X^{\al}$ denote the split extension determined by $\be \colon X\rtarr \UP M$, as specified in \myref{splitext} (thus $X^\alpha$ here corresponds to $Y$ there).
We extend (\ref{diff1}) and (\ref{diff2}) by defining 
\begin{equation}\label{diff3}
d^0 = -d: M\rtarr M \ \ \text{and} \ \ d^{p+1} = \al\colon X_{p,q} \rtarr M_{p+q}.
\end{equation}
Setting $X_{-1,q} = (\UP M)_{q-1} = M_q$, the equation 
\[ d\al = \al d\colon X_{p,q} \rtarr M_{p+q-1} = X_{-1,p+q}\] 
becomes 
\[ -d^0d^{p+1} = \sum_{0\leq j\leq p} d^{p+1-j}d^j, \ \ \text{hence}\ \  \sum_{i+j=p+1}d^id^j = 0.\]
\end{defn}

\begin{rmk}  The notation $F\al$ for $X^{\al}$ might be reasonable\footnote{$X^{\al}$ was 
misleadingly called a mapping cylinder in \cite{GM}.} since we have a rough analogy with
topological fiber sequences
\[ \xymatrix@1{ 
\OM M \ar[r]^-{i} & F\al \ar[r] & X \ar[r]^-{\al} & M.\\} \]
\end{rmk}

\section{From homological algebra to model category theory}\label{GMStuff1} 

Calculationally, our work begins with the Eilenberg-Moore spectral 
sequence, abbreviated EMSS.  Split DG $A$-modules give rise to spectral sequences that 
are candidates for the EMSS.  We define resolutions $\al\colon X\rtarr M$ and in 
particular distinguished and K\"unneth resolutions of DG $A$-modules $M$ in \S\ref{EMSS}.  
Distinguished resolutions are particularly nice $q$-cofibrant approximations, whereas
K\"unneth resolutions are tailored to give the weakest data sufficient to construct 
the EMSS with the correct $E^2$-term and the correct target. The target is given by 
differential torsion products. Even weaker kinds of 
resolutions, namely semi-flat resolutions, give the correct target even though they need 
not give the correct $E^2$-term, and these too are defined in \S\ref{EMSS}. 

In \S\ref{barsec}, we show that the classical bar construction gives $r$-cofibrant approximations 
of all DG $A$-modules $M$ for any DG $R$-algebra $A$, even though the bar construction is never itself 
an $r$-cell complex when the differential on $A$ is non-zero.  Under mild hypotheses, the bar
construction also gives semi-flat resolutions, which means that it behaves homologically as if
it were a $q$-cofibrant rather than just an $r$-cofibrant approximation.  This implies that
$\Tor$ ($=q\Tor$) and $r\Tor$ agree far more often than one would expect from model considerations alone. 

The fact that our preferred resolutions are given by multicomplexes and not just bicomplexes has 
structural implications for the EMSS in terms of matric Massey products.  We indicate briefly how 
that works in \S\ref{matric}.

\subsection{Split, K\"unneth, and semi-flat DG $A$-modules and the EMSS}\label{EMSS}

A filtered DG $R$-module $Y$ gives rise to a spectral sequence $E^rY$ of DG $R$-modules 
starting from 
\[ E^0_{p,q}Y = (F_pY/F_{p-1}Y)_{p+q}. \]
We are interested in the cases $Y= X$ and $Y=X^{\al}$ for a map $\al\colon X\rtarr M$ of DG $A$-modules,
where $X$ is split.  In the latter case, we have
\[ E^1_{-1,q}X^{\al} = H_q(M) \ \ \text{and} \ \ E^1_{p,q}X^{\al} = E^1_{p,q}X \ \ \text{for}\ \ p\geq 0. \]
The differentials are of the form $d^r\colon E^r_{p,q}\rtarr E^r_{p-r,q+r-1}$, and $d^0$ is given
by the summand $d^0$ of $d$.  The complex $E^1_{*,*}X^{\al}$ takes the form
\begin{equation}\label{Eone}
 \cdots \rtarr E^1_{p,*}X \rtarr E^1_{p-1,*}X \rtarr \cdots \rtarr E^1_{0,*}X \rtarr H_*(M) \rtarr 0.
\end{equation}
\begin{defn}[{\cite[1.1]{GM}}]\mylabel{resdef}  We say that $\al\colon X\rtarr M$ is a \emph{resolution} of $M$ if 
the sequence (\ref{Eone}) is exact. We say that $\al$ is a {\em distinguished resolution} of $M$ if $\al$ is a 
resolution and $X$ is a distinguished DG $A$-module, that is, a $q$-cell complex.
\end{defn}  
Since $\{E^r_{*,*}X^{\al}\}$ is a right-half plane spectral sequence with homological grading, there is no convergence problem \cite[\S6]{Board}.  Filtering $M$ itself by $F_{-1}M=0$ and $F_pM = M$ for $p\geq 0$, we can reinterpret (\ref{Eone}) as the 
$E^1$-term of a map of spectral sequences $E^rX\rtarr E^rM$ induced by $\al$.  Using the convergence \cite[7.2]{Board}, we have the following result.

\begin{prop}\mylabel{seecof}  If $\al\colon X\rtarr M$ is a resolution of $M$, then $\al$ is a $q$-equivalence. \end{prop}

\begin{rmk}  We emphasize that we do not have a converse to \myref{seecof}.  In particular, we have no reason to
believe that a general $q$-cofibrant approximation $\al\colon X\rtarr M$ is a resolution.  We shall be giving
three different homological constructions of resolutions that have $q$- or $r$-cofibrant domains.
\end{rmk}

\begin{defn}[{\cite[1.1]{GM}}]\mylabel{Kunndef}  For a right DG $A$-module $N$ and a split DG $A$-module $X$, 
give $N\otimes_A X$ the induced filtration,
$F_p(N\otimes_A X) = N\otimes_A F_pX$.  There is an evident K\"unneth map
\[ \ka\colon  HN\otimes_{HA} E^1X \rtarr E^1(N\otimes_A X). \]
A split DG $A$-module $X$ is {\em K\"unneth} if each $E^1_{p,*}X$ is a flat
$HA$-module and $\ka$ is an isomorphism for every $N$. We say that $\al\colon X\rtarr M$ is a
{\em K\"unneth resolution} of $M$ if $\al$ is a resolution and $X$ is K\"unneth.  When this holds, 
(\ref{Eone}) is a resolution of $HM$ by flat $HA$-modules and therefore
\[ E^2_{p,q} (N\otimes_A X) = \Tor^{HA}_{p,q}(HN, HM). \]
\end{defn}

\begin{ex}
The total complex $\TotP$ of a projective resolution in the sense of \S\ref{versustoo} is K\"unneth. 
\end{ex}

\begin{defn}  If $\al\colon X\rtarr M$ is a K\"unneth resolution, we call the spectral sequence 
$\{E^r_{p,q}(N\otimes_A X)\}$ an {\em Eilenberg-Moore spectral sequence} (EMSS).
\end{defn}

We have identified the $E^2$-term as a classical $\Tor$ functor.  The target is the differential $\Tor$ functor of \myref{DTordefn}.  Similarly, if $X$ is distinguished and $N$ is a left DG $A$-module, then
\[ H^{*,*}\Hom_{HA}(E^1_{*,*}X,HN) = \Ext^{*,*}_{HA}(HM,HN) \]
and we have a cohomological EMSS whose target is the differential $\Ext$ functor.  We shall not consider it in any detail here.

In what follows, we repeatedly use the isomorphism of DG $R$-modules
\begin{equation}\label{obvious}
E^1_{p,*} = (N\otimes_A X_{p,*}; d^0) \iso (N\otimes \bar{X}_{p,*}; d\otimes \id + \id\otimes d^0),
\end{equation}
for a split DG $A$-module $X$, where we do not always assume that $d^0=0$ on $\bar{X}$ but we do assume that $d^0(\bar X) \subset \bar{X}$.  The $d^0$ on the left is the differential on $N\otimes_A X_{p,\ast}$ viewed as the $E^0$ term of the spectral sequence. On the right, the $d$ is the differential on $N$ and the $d^0$ is the differential on $\bar{X}$ under our assumption that $d^0(\bar X) \subset \bar{X}$. 

\begin{lem}\mylabel{ofcourse}  If $X$ is a cell-like DG $A$-module such that each $\bar{X}_{p,q}$ is a flat $R$-module,
then $X$ is a K\"unneth DG $A$-module. In particular, distinguished DG $A$-modules are K\"unneth.
\end{lem}
\begin{proof}
Here $d^0=0$ on the right side of (\ref{obvious}).  Since homology commutes with tensor
products with flat $R$-modules, the conclusion is immediate.
\end{proof}

We next describe a more general kind of resolution that will lead to surprising
invariance properties of the target of the EMSS.  We need the generality to deal with the 
bar construction in \S\ref{barsec}. 
When $A=R$, the following definition is due to \cite{Danes}. Recall that, ignoring differentials,
a graded $A$-module $X$ is $A$-flat if the functor $(-)\otimes_A X$ on graded right $A$-modules $N$ is exact.

\begin{defn}  A DG $A$-module $X$ is {\em semi-flat} if the underlying $A$-module
of $X$ is $A$-flat and the functor $(-)\otimes _A X$ on right DG $A$-modules 
preserves $q$-equivalences.   We say that $\al\colon X\rtarr M$ is a
{\em semi-flat resolution} of $M$ if $\al$ is a resolution and $X$ is semi-flat.
\end{defn}

\begin{rmk}\mylabel{Notsemi}  Degreewise free DG $R$-modules need not be semi-flat. If $R = \bZ/4$ and
$X$ is the degreewise free DG $R$-module of \myref{projectives}, then $X$ is not semi-flat.   In fact if 
$\epz\colon P\rtarr \bZ/2$ is a classical $R$-projective resolution, then $H_*(P\otimes_R X) = 0$
but $H_*(\bZ/2\otimes_R X)$ is $\bZ/2$ in every degree.
\end{rmk}

The following result shows that split DG $A$-modules are very often K\"unneth or semi-flat 
DG $A$-modules even when they are not cellular.  

\begin{prop}\mylabel{KunnethRec} Let $X$ be a split DG $A$-module such that $d^0(\bar{X})\subset \bar{X}$
and each $\bar{X}_{p,q}$ is $R$-flat. Then $X$ is semi-flat under either of the following hypotheses:
\begin{enumerate}[(i)]
\item  $R$ is a PID
\item  $A$ and each $\bar{X}_{p,\ast}$ is bounded below.
\end{enumerate}
If, further, each $H_{p,q}(\bar{X},d^0)$ is $R$-flat, then $X$ is K\"unneth.
\end{prop}
\begin{proof}  Since $N\otimes_A X \iso N\otimes_R \bar{X}$ as graded $R$-modules and exactness is
seen degreewise, it is clear that $X$ is $A$-flat.  We use (\ref{obvious}) to see that the functor
$(-)\otimes_A X$ preserves $q$-equivalences.  The classical K\"unneth theorem in case (i) and the 
K\"unneth spectral sequence, \myref{KunnSS}, in case (ii) ensure that $E^1(N\otimes_A X)$ depends functorially on $HN$ 
and $H\bar{X}$, although it need not reduce to  $HN\otimes H\bar{X}$ in general.  By the naturality 
of the K\"unneth theorem or the naturality and convergence of the K\"unneth spectral sequence, 
together with the convergence of the spectral sequence $\{E^r(N\otimes_A X)\}$, we conclude (as in \cite[7.2]{Board})
that the functor $E^1((-)\otimes_A X)$ and therefore the functor $(-)\otimes_A X$ 
preserve $q$-equivalences if (i) or (ii) holds.  When $H\bar{X}$ is degreewise $R$-flat, 
$E^1(N\otimes_A X)\iso HN\otimes H\bar{X}$. Taking $N=A$ we see that $E^1X$ is $HA$-flat and that
\[  HN\otimes_{HA} E^1X = HN\otimes_{HA} (HA\otimes H(\bar{X};d^0)) = E^1(N\otimes_A X).  \qedhere \]
\end{proof}

\begin{rmk} If $R$ is a Noetherian ring and $C$ is a projective $R$-module, then $\Hom_R(C,R)$ is a flat $R$-module, 
but it need not be projective. For example, each $C^q(X;R)$ is a flat $R$-module for any space $X$.  
Since examples of the form $A=C^*(X;R)$ appear naturally in algebraic topology, this 
gives concrete motivation for considering degreewise $R$-flat DG $R$-algebras $A$ and DG $A$-modules $M$; see \S\ref{TopEMSS}.
\end{rmk}

\begin{defn}\mylabel{gamma}  Let $\ga\colon \Tor^A(N,M)\rtarr  H(N\otimes_A M)$ denote the natural
map induced by $\be\otimes \al$  or, equivalently, $\be\otimes \id$ or $\id\otimes \al$, as in \myref{inv1}.
\end{defn} 

Classically, when there are no differentials on $A$, $M$, and $N$, $\ga$ reduces to the
natural isomorphism 
\[   \Tor^A_{0,*}(N,M) = N\otimes_A M.\]  
In the absence of differentials,  we also have that  $\Tor^A_{p,*}(N,M) =0$ if $N$ or $M$ is $A$-flat. 
The following direct consequence of the definition of a semi-flat DG $A$-module is the closest 
we can get to these assertions in the differential graded case. 

\begin{prop}  If $M$ or $N$ is semi-flat, then $\ga\colon \Tor^A_*(N,M)\rtarr H(N\otimes_A M)$
is an isomorphism.
\end{prop}

Of course, we can compute $\Tor^A_*(N,M)$ using arbitrary K\"unneth resolutions,
as reflection on the EMSS makes clear.  But in fact we have the following more general result,
which will become relevant when we consider the bar construction in \S\ref{barsec}.  

\begin{prop}\mylabel{inv2}  If $\al\colon X\rtarr M$ is a $q$-equivalence, where $X$ is semi-flat,
then $\Tor^A_*(N,M)$ can be computed as $H(N\otimes_A X)$.
\end{prop}
\begin{proof}
Let $\be\colon Y\rtarr N$ be a $q$-cofibrant approximation. By \myref{inv1} and the definition
of semi-flat, $\al$ and $\be$ induce isomorphisms
\[   H(Y\otimes_A M)  \ltarr H(Y\otimes_A X) \rtarr H(N\otimes_A X). \qedhere\]
\end{proof} 

\begin{rmk}\mylabel{paradox}
In this generality, we do not even know that $X$ is a split DG $A$-module, although
we do not know examples where that fails.  Even when that holds, we cannot expect 
$E^2(N\otimes_A X)$ to be $\Tor^{HA}_{*,*}(HN,HM)$.  However, in view of the existence of K\"unneth resolutions of any $M$, we conclude from
\myref{inv2} that we do have an EMSS with that $E^2$-term that converges
to $H(N\otimes_A X)$.
\end{rmk}

\subsection{The bar construction and the $r$-model structure}\label{barsec}
We assume familiarity with the two-sided bar construction $B(N,A,M)$ for a DG algebra $A$ and right and left $A$-modules $N$ and $M$. It is the total complex associated to the evident simplicial DG $R$-module with $p$-simplices $N\otimes A^{\otimes p}\otimes M$; see for example \cite[App A]{GM}.  The following result is a reinterpretation of \cite[A.8]{GM}.  Let $JA$ denote the
cokernel $A/R$ of the unit of $A$; it is a quotient DG $R$-module of $A$.  Usually $A$ is augmented, and then $JA$ may be identified with the augmentation ideal $IA$.

\begin{prop}\mylabel{bar}  For any DG $R$-algebra $A$ and DG $R$-module $M$, 
the standard map $\epz\colon B(A,A,M)\rtarr M$ is an $r$-cofibrant 
approximation of $M$. It is functorial in both $A$ and $M$.
\end{prop}
\begin{proof}  Give $B(A,A,M)$ its simplicial filtration. Its filtration quotient $F_p/F_{p-1}$ is the 
relatively free $A$-module $A\otimes (JA)^p\otimes M$.  As an $A$-module, $B(A,A,M)$ is 
$A\otimes \bar{B}$, where $\bar{B}_{p,*} = (JA)^{\otimes p}\otimes M$, and it is the direct sum
of its filtration quotients $F_p/F_{p-1}$.  Thus the filtration is $r$-split and $B(A,A,M)$ is $r$-cofibrant
by \myref{rcharcofibrant}.  Moreover, $\epz\colon B(A,A,M)\rtarr M$ is an $R$-split epimorphism and thus an 
$r$-fibration; the unit $\io \colon M\rtarr A\otimes M = F_0$ gives the splitting. The standard homotopy between 
the identity and $\io \com \epz$ shows that $\epz$ is an $r$-equivalence.  
\end{proof}

Since $B(A,A,M)$ is $r$-cofibrant, it is a retract of an $r$-cell complex.  While that is obvious 
from our model categorical work, it is nevertheless a little mysterious: we have no direct way of seeing 
it using homological methods and the simplicial filtration.  The following result complements the previous one.

\begin{prop}\mylabel{semibar}  If $JA$ and $M$ are degreewise $R$-flat, then $B(A,A,M)$ is semi-flat
under either of the following hypotheses:
\begin{enumerate}[(i)]
\item  $R$ is a PID,
\item  $A$ and $M$ are bounded below.
\end{enumerate}
If, further, $HJA$ and $HM$ are degreewise $R$-flat, then $B(A,A,M)$ is K\"unneth and
$\epz\colon B(A,A,M)\rtarr M$ is a K\"unneth resolution of $M$. 
\end{prop}
\begin{proof}  Observe that $A$ is $R$-flat if $JA$ is $R$-flat and $HA$ is $R$-flat if
$HJA$ is $R$-flat. Except for the last clause, this is immediate from the proofs of Propositions \ref{KunnethRec} and \ref{bar}. 
The differential $d^0$ on $\bar{B}$ is the internal differential induced by the 
differentials on $JA$ and $M$. The K\"unneth theorem gives that if
$JA$, $M$, $HJA$, and $HM$ are degreewise $R$-flat, then (\ref{Eone}) for $X=B(A,A,M)$ is the 
flat $HA$-resolution $B(HA,HA,HM)$ of $HM$.
\end{proof} 

Since $\epz\colon B(A,A,M)\rtarr M$ is an $r$-equivalence and thus a $q$-equivalence, \myref{inv2}
gives that 
\[  \Tor^A_*(N,M) = H(N\otimes_A B(A,A,M)) = HB(N,A,M) \]
whenever $B(A,A,M)$ is semi-flat.  By \myref{paradox}, we then have an EMSS converging from 
$E^2= \Tor^{HA}_{*,*}(HN,HM)$ to $HB(N,A,M)$, even though it may not come from the simplicial
filtration of $B(N,A,M)$.  However, when $B(A,A,M)$ is K\"unneth, the spectral sequence does
come from that filtration, which then gives the correct $E^2$-term.  

Observe that \myref{bar} gives
\[ r\Tor^A_*(N,M) = HB(N,A,M) \]
for a right DG $A$-module $N$ and
\[ r\Ext_A^*(N,M) = H\Hom_A(B(A,A,M),N) \]
for a left DG $A$-module $N$.  The simplicial filtration gives a spectral sequence converging to 
$r\Tor^A(N,M)$.  By \cite[\S IX.8]{MacHom}, we can define relative classical $\Tor$ functors 
$r\Tor^{HA}_{*,*}(HN,HM)$ starting from \myref{proj}.  We do \emph{not} have an identification 
of $E^2(N,A,M)$ with $r\Tor^{HA}(HN,HM)$ in general.  However, it is now clear that $r\Tor$ and $\Tor$
agree under surprisingly mild hypotheses.

\begin{thm}\mylabel{Yesflat}  Assume that $A$ and $M$ are degreewise $R$-flat and (i) or (ii)
of \myref{semibar} holds.  Then
\[  \Tor^A(N,M) = r\Tor^A(N,M) \]
for all right DG $A$-modules $N$.
\end{thm}

Under these hypotheses, we can use the bar construction just as if $B(A,A,M)$ were a 
$q$-cofibrant approximation of $M$, even though $B(A,A,M)$ is not cell-like and need not be $q$-cofibrant. Of course, by the SOA, $B(A,A,M)$ admits a $q$-cofibrant approximation $\ze\colon X\rtarr B(A,A,M)$.  When we can find such a $\ze$ which is an $r$-equivalence over $M$, we can conclude that $B(A,A,M)$ is  $h$-equivalent to $X$ and is thus a $(q,h)$-cofibrant approximation of $X$. That is presumably not possible  in general. However, when $R$ is a field (or semi-simple), the $q$-, $r$-, and $h$-model structures on $\sM_R$ coincide, hence the $q$- and $r$-model structures on $\sM_A$ coincide.  In that case, $\epz\colon B(A,A,M)\rtarr M$ is a $q$-cofibrant approximation of $M$ even though $B(A,A,M)$ is not cell-like, hence not distinguished and not a $q$-cell complex.

\subsection{Matric Massey products and differential torsion products}\label{matric}

Let us return to the map $\ga\colon \Tor^A_*(N,M)\rtarr H(N\otimes_A M)$ of \myref{gamma}.
It is not an isomorphism in general. The following curious substitute for this isomorphism
relies on matric Massey products, as defined in \cite{MayMat} and recalled in \cite[\S5]{GM}.

\begin{thm}[{\cite[5.9]{GM}}]\mylabel{matrix1} The image of $\ga$ is the set $D(N,A,M)$ of all elements of all 
matric Massey products
$\langle V_0,V_1,\cdots,V_p, V_{p+1}\rangle$, $p\geq 0$, where $V_0$ is a row matrix in $HN$, the $V_i$ for
$1\leq i\leq p$ are matrices with entries in $HA$, and $V_{p+1}$ is a column matrix with entries in $HM$.
\end{thm}

The letter $D$ stands for ``decomposable.''  When $p=0$, we understand $\langle V_0,V_1 \rangle$ to be the image 
(up to signs) of $V_0\otimes V_1$ in $H(N\otimes_A M)$.    The proof uses nothing but the homological material 
we have summarized.  The essential point, explained in detail in \cite[pp 49--57]{GM}, is that the formula $d^2=0$
for the differentials of the multicomplex $N\otimes_A X$ is so similar to the boundary conditions that specify
defining systems for matric Massey products that the entire spectral sequence $\{E^r(N\otimes_A X)\}$ can be described
in terms of matric Massey products.  That discussion starts from a distinguished resolution $X$ of $M$,
but it applies to any $q$-cell approximation.

When $A$ has an augmentation $\epz\colon A\rtarr R$, so that $R$ is a DG $A$-module, the 
special cases $M=R$ (or $N=R$) and $M=N=R$ are of particular importance in the applications.  
We then let $IA = \ker\, \epz$ and $IHA = \ker\, H\epz$.   The inclusion $\io\colon IA\rtarr A$ induces 
\[ H(\io\otimes \id)\colon H(IA\times_A M) \rtarr H(A\otimes_A M) = HM \]
and we let $D(HA;HM)$ denote the image of $D(IA,A,M)$ in $HM$.   We have a natural map
\[ \pi\colon HM \rtarr \Tor^{A}_*(R,M), \]
namely the ``edge homomorphism''
\[ R = R\otimes HM \rtarr R\otimes_{HA}HM = E^2_{0,*} \rtarr E^{\infty}_{0,*} = F_0\Tor^A_*(R,M). \]
By \cite[5.12]{GM}, \myref{matrix1} implies the following special case.

\begin{cor}\mylabel{matrix2} The kernel of $\pi\colon HM\rtarr \Tor^{A}_*(R,M)$ is $D(HA;HM)$. 
\end{cor}

Specializing further, we have a suspension homomorphism \cite[3.7]{GM}
\[ \si\colon IHA \rtarr E^2_{1,*} \rtarr E^{\infty}_{1,*}\subset \Tor^A_*(R,R).\]
The inclusion results from the fact that $R = F_0\Tor^A_*(R,R)$ is a direct summand of $\Tor^A_*(R,R)$.  
The inclusion $\io\otimes \io\colon IA\otimes_A IA\rtarr A\otimes_A A$ induces a map 
\[ H(IA\otimes_{A}IA)\rtarr HA, \]
and we let $DHA$ denote the image of $D(IA,A,IA)$ in $HA$.  By \cite[5.13]{GM}, \myref{matrix2} implies the 
following further special case.

\begin{cor}\mylabel{matrix3} The kernel of $\si\colon IHA\rtarr \Tor^A_*(R,R)$ is $DHA$. 
\end{cor}

\subsection{Massey products and the classical $\Ext$ functor}\label{ExtA}

We record an application of \S\ref{matric}.  We show that all elements of the $\Ext$ groups of a connected 
algebra $A$ over a field are decomposable in terms of matric Massey products, starting from the indecomposable 
elements of $A$ itself.  An analogous result holds for $A$-modules.  Thus we assume here that $R$ is a field 
and we consider a connected graded $R$-algebra $A$ (so that $A_n = 0$ for $n<0$ and $A_0 = R$) and an $A$-module $M$, 
both of finite type over $R$.  These do not have differentials. We are thinking, for example, of the Steenrod 
algebra $A$ and the cohomology $M$ of a spectrum.  The augmentation $\epz\colon A\rtarr R$ makes $R$ an $A$-module,
and we have the bar construction $B(R,A,M)$.  We write $B(A) = B(R,A,R)$.  The dual of $B(A)$ is the cobar construction 
$C(A)$, which is a DG $R$-algebra, and we write $C(A;M)$ for the dual of $B(R,A,M)$, which is a (left) DG $C(A)$-module.   
Then

\[ HC(A) = \Ext_{A}^{*,*}(R,R) \ \ \text{and}\ \ HC(A;M) = \Ext_{A}^{*,*}(M,R).\]

The $R$-module $\Ext_A^{1,*}(R,R)$ is dual to the $R$-module $IA/(IA)^2$of indecomposable elements of $A$, 
and the $R$-module $\Ext_A^{0,*}(M,R)$ is dual to the $R$-module  $M/(IA)M$ of indecomposable elements of 
the $A$-module $M$.  We sketch how a version of the EMSS proves the following result \cite[5.17]{GM}.

\begin{thm}\mylabel{wowsy}  $\Ext_A^{*,*}(R,R)$ is generated by $\Ext_A^{1,*}(R,R)$ under matric Massey products.
$\Ext_A^{*,*}(M,R)$ is generated by $\Ext_A^{0,*}(M,R)$ under matric Massey products.
\end{thm}

We use $\Tor_{C(A)}^{*}(C(A;M),R)$ to prove this; the special case $R=M$ leads to the first statement.  
We have the algebraic EMSS converging from $\Tor_{HA}^{*,*}(HC(A;M),R)$ to $\Tor_{C(A)}^*(C(A;M),R)$.  A
standard relation between the bar and cobar constructions evaluates the target \cite[5.16]{GM}. Let $M^*$
denote the dual of $M$.

\begin{prop}  $\Tor^0_{C(A)}(C(A;M),R) = M^*$ and $\Tor^n_{C(A)}(C(A;M),R) = 0$ for $n\neq 0$.
\end{prop}

This is a consequence of the fact that, ignoring differentials, $C(A;M)$ is free as a 
right $C(A)$-module.  Filtering $B(C(A;M),C(A),R)$ so that $d_0$ is given by the simplicial (external)
differential, we get a spectral sequence converging from the classical $\Tor$, with internal differentials
ignored, to $\Tor^*_{C(A)}(C(A;M),R)$. It trivializes to give the stated conclusion.  From here, the 
deduction of \myref{wowsy} from \ref{matrix2} and \ref{matrix3} is easy \cite[p.~61]{GM}.  The point is that all elements
except the specified generators are in the kernels identified as matric Massey product decomposables in the cited
results.  In fact, by the precursor \cite{MayOld} to \cite{MayMat}, the EMSS here is itself an algorithm
for the computation of  $\Ext_A^{*,*}(N,R)$ and, in particular,  $\Ext_A^{*,*}(R,R)$.

\section{Distinguished resolutions and the topological EMSS}\label{GMStuff2}

Here we construct distinguished resolutions of arbitrary DG $A$-modules $M$, 
as defined in \myref{resdef}; they are K\"unneth resolutions by \myref{ofcourse}.  
We emphasize that these are generally {\em not} $q$-cofibrant approximations
and that, as far as we know, $q$-cofibrant approximations need {\em not} give 
K\"unneth resolutions.

Resolutions $\al\colon X\rtarr M$ must be $q$-equivalences, but they need not be epimorphisms, 
hence they need not be $q$-cofibrant approximations even when $X$ is $q$-cofibrant, as holds by 
\myref{neatojet} for distinguished resolutions.  On the other hand, $q$-cofibrant approximations
need not have the control over $E^1X$ needed to give resolutions as defined in terms of (\ref{Eone}),
let alone K\"unneth resolutions.

However, if $X$ is $q$-cofibrant and $\ga\colon Y\rtarr M$ is a $q$-cofibrant approximation in 
the usual sense that $\ga$ is a $q$-acyclic $q$-fibration, then we obtain a lift $\la\colon X\rtarr Y$ 
over $M$. Since $\la$ is then a $q$-equivalence between $q$-bifibrant objects, it is an $h$-equivalence.  
Thus we may use distinguished resolutions just as if they were model theoretical $q$-cofibrant approximations.  

As we explain in \S\ref{resolve}, \cite{GM} gives a purely homological construction 
of a distinguished resolution of any $M$. These resolutions can be small enough to actually compute with, 
as we illustrate in \S\ref{cupone} in the case when $H_*(A)$ is a polynomial algebra.  The smallness is 
directly correlated with the fact that $\al$ need not be a $q$-fibration: for calculations, 
that is an advantage rather than a disadvantage.   

We then reap the harvest and show in \S\ref{TopEMSS} how our work, especially \myref{goody},
applies to give explicit calculations in algebraic topology.  In particular, we explain
the proof of \myref{Yeah}.  On a more theoretical level, we show that the kernels of various 
maps of cohomological interest are determined by matric Massey products. 

\subsection{The existence and essential uniqueness of distinguished resolutions}\label{resolve}
We all know how to construct classical projective $HA$-resolutions of $HA$-modules, and there is an ample arsenal of 
known examples.  The following result is an analogue of the classical existence
result for projective resolutions.  It allows us to lift projective $HA$-resolutions to 
distinguished $A$-resolutions. 

\begin{thm}[{\cite[2.1]{GM}}]\mylabel{exist}  Let $M$ be a DG $A$-module and let
\begin{equation}\label{Eone2}
 \cdots \rtarr HA\otimes \bar{X}_{p,*} \rtarr HA\otimes \bar{X}_{p-1,*}  \rtarr \cdots \rtarr
 HA\otimes \bar{X}_{0,*}  \rtarr HM \rtarr 0
  \end{equation}
be a projective $HA$-resolution of $HM$, where each $\bar{X}_{p,q}$ is a projective $R$-module.  Then the filtered $A$-module
$X=A\otimes \bar{X}$ with filtration $F_pX = \sum_{k\leq p} A\otimes \bar{X}_{k,*}$ admits a differential $d$ and a map 
$\al\colon X\rtarr M$ such that $\al$ is a distinguished resolution of $M$ and the complex (\ref{Eone}) coincides with the 
complex (\ref{Eone2}).   
\end{thm}
\begin{proof} The paper \cite{GM} works with right rather than left DG $A$-modules (with $A$ denoted $U$) and its signs
and details have several times been checked with meticulous care.  The description of $X^{\al}$ as a bigraded $A$-module
is forced, and so is the definition of $d^0$.  One first uses projectivity to define $d^1$ so that the complexes 
(\ref{Eone}) and (\ref{Eone2}) agree. One then uses projectivity to define the $d^r$ on $\bar{X}_{p,*}$ for $r\geq 2$ 
and $p\geq 1$ by induction on $p$ and, for fixed $p$, by induction on $r$ in such a way that (\ref{diff1}) is satisfied.  
The construction of $d^{p+1}$ on $\bar{X}_{p,*}$ gives $\al$. The details \cite[pp.~12-15]{GM} are a bit tedious,
but they are entirely straightforward.
\end{proof}

To state our result on comparisons of resolutions, we need an implication of Definitions \ref{ancdef}
and \ref{ancdef2}, as in \cite[1.3]{GM}. 

\begin{rmk} Let $\al\colon X\rtarr M$ and $\al'\colon X'\rtarr M'$ be maps of DG $A$-modules where $X$
and $X'$ are split and let 
$g\colon X^{\al}\rtarr (X')^{\al'}$ be a map of filtered DG $A$-modules.  On filtration $-1$, $g$ specifies a
map $k\colon M\rtarr M'$ of DG $A$-modules.  On $X_{p,q}$ for $p\geq 0$, $g$ has components 
$g^r\colon X_{p,q}\rtarr X'_{p-r,q+r}$ for $0\leq r\leq p$ and $t\colon X_{p,q} \rtarr M'_{p+q+1}$.
Let $K = \sum_{0\leq r\leq p} g^r\colon X_{p,*}\rtarr F_pX'$.  Then $K\colon X\rtarr X'$ is a map 
of filtered DG $A$-modules and $dt+td = \al' K - k\al$.  Therefore, $g$ determines and is determined by the
homotopy commutative diagram 
\[  \xymatrix{  X \ar[r]^-{K} \ar[d]_{\al} & X' \ar[d]^{\al'} \\
M \ar[r]_{k} & M' \\} \]
of DG $A$-modules and the specific homotopy $t$.  We write $g = (K,k,t)$. 
\end{rmk}

Again remembering that the classical $\Tor$ functor can be computed by use of flat resolutions,
the remark implies the following result by tensoring with a right DG $A$-module $N$ and
passing to the resulting map of spectral sequences.

\begin{lem}\mylabel{KunnethIso}  Let $\al\colon X\rtarr M$ and $\al'\colon X'\rtarr M'$ be 
K\"unneth resolutions and let $g =(K,k,t)\colon X^{\al}\rtarr  (X')^{\al'}$ be a 
map of filtered DG $A$-modules.   Then, for any right DG $A$-module $N$, 
\[ E^2(\id\otimes K) \colon  E^2(N\otimes_A X) \rtarr E^2(N\otimes_A X') \]
can be identified with
\[  \Tor^{HA}(\id,Hk) \colon \Tor^{HA}(HN,HM)\rtarr \Tor^{HA}(HN,HM'). \]
Therefore, if $Hk\colon HM\rtarr HM'$ is an isomorphism, then 
\[ H(\id\otimes_A K)\colon H(N\otimes_A X) \rtarr H(N\otimes_A X') \]
is an isomorphism.
\end{lem}

The following result is the analogue of the comparison result between projective complexes 
and resolutions in classical homological algebra.  It allows us to compare distinguished 
resolutions to general resolutions.  

\begin{thm}[{\cite[1.7]{GM}}] \mylabel{compare} Let $\al\colon X\rtarr M$ be a map of DG $A$-modules, where $X$ is distinguished, 
let $\al'\colon X'\rtarr M'$ be a resolution of a DG $A$-module $M'$, and let $k\colon M\rtarr M'$ be a map of DG $A$-modules.  
Then there is a map
$g = (K,k,t) \colon X^{\al}\rtarr (X')^{\al'}$ of filtered DG $A$-modules.  If $g' = (K',k,t')$ is another such map, then
there is a homotopy $s\colon g\htp g'$ of DG $A$-modules such that $s(M) = 0$ and $s(F_pX^{\al}) \subset F_{p+1} (X')^{\al'}$
for $p\geq 0$. 
\end{thm}
\begin{proof} The proof is by induction on $p$, using the requirement that $dg = gd$.
It can be better written than the argument of \cite[pp.~7-8]{GM}, but it is straightforward.
\end{proof}

\begin{cor}  If $\al\colon X\rtarr M$ and $\al'\colon X'\rtarr M$ are distinguished resolutions
of $M$, then $X$ and $X'$ are $h$-equivalent over $M$.
\end{cor}

Of course, since distinguished DG $A$-modules are $q$-cofibrant, the corollary is also 
immediate from model category theory.

\subsection{A distinguished resolution when $H_*(A)$ is a polynomial algebra}\label{cupone}

Many of the applications of \cite{GM, M, MN} are based on an explicit example of \myref{exist}.
We assume in this section that $HA$ is a polynomial algebra on generators $x_i$ indexed on
some ordered set $I$.  Since $2x^2 = 0$ if $x$ has odd degree, the $x_i$ must have even degree 
unless $R$ has characteristic $2$. When $HA$ is commutative, it usually is so because 
$A$ is chain homotopy commutative via a homotopy $\cup_1\colon A\otimes A\rtarr A$. Very often 
$\cup_1$ satisfies the Hirsch formula, which means that it is a graded derivation.  We assume that 
we have such a ``$\cup_1$-product'' on $A$. Explicitly, for
$a\in A_p$, $b\in A_q$, and $c\in A_r$, we require
\[ d(a\cup_1 b) = ab-(-1)^{pq}ba - d(a)\cup_1 b- (-1)^{p} a\cup_1 d(b) \]
and
\[ (ab)\cup_1c = (-1)^p a(b\cup_1 c) + (-1)^{qr}(a\cup_1 c)b. \]
We also assume that we have an augmentation $A\rtarr R$ that induces the
standard augmentation $\epz\colon HA\rtarr R$, $\epz(x_i) = 0$. 

We have the Koszul resolution $K(HA)$ of $R$. It is the differential $HA$-algebra $HA\otimes E\{y_i\}$, where 
the bidegree of $y_i$ is $(1,deg\, x_i)$.  Here $E$ denotes an exterior algebra and $d(y_i) = x_i$. Let 
$K(A) = A\otimes E\{y_i\}$ and let $\epz\colon K(A)\rtarr R$ be the evident augmentation.  \myref{exist} gives 
a differential $d$ on $K(A)$, but in this case we do not need to rely on that result: we can construct the 
differential explicitly so that $\epz$ is a distinguished resolution of $R$.  We shall not give full details,
since the only problem is to get the signs right and that was done with care in \cite[pp.~16-17]{GM}, although
working with right rather than left modules.\footnote{One
lengthy check of signs was left to the reader, but the senior author still has handwritten full details. Using the
transposition isomorphism 
$$t\colon E\{y_i\}\otimes A \rtarr A\otimes E\{y_i\}, \ \  t(x\otimes a) = (-1)^{deg\, a\, deg\,x}(a\otimes x),$$
and defining $d = tdt$ on $A\otimes E\{y_i\}$ gives correct signs for our left $A$-module resolution.}   Let 
$a_i\in A_i$ be a representative cycle of $x_i$.   For
an ordered sequence of indices $S=\{i_1 <\cdots < i_p\}$, let $\ell(S)=p$ and define $a_S$ and $y_S$ by induction on $p$.
If $S=\{i\}$, then  $a_S = a_i$ and $y_S= y_i$.  If $S=\{i,T\}$, then $a_S = a_i\cup_1 a_T$ and $y_S = y_iy_T$.  We require 
$K(A)$ to be a DG $A$-algebra, hence to define $d$ on $K(A)$, we need only define the $d(y_S)$.  We consider all partitions of
$S$ as $S = U\cup V$ where $U\cap V=\emptyset$ and $U$ and $V$ are nonempty.  Then 
\begin{equation}\label{Koszul}
d(y_S) = \sum_{U,V}\si(U,V)\, a_U \otimes y_V
\end{equation}
for appropriate signs $\si(U,V)$, so chosen that $dd=0$ and $\si(U,V)$ is as dictated by $E^1_{*,*}K(A) = K(HA)$ when $\ell(U)=1$. 

Now assume that $N$ is another augmented DG $R$-algebra that is homotopy commutative via a $\cup_1$-product satisfying the Hirsch
formula. Let $f\colon A\rtarr N$ be a map of DGAs that commutes with the $\cup_1$-product and give $N$ a structure of right DG 
$A$-module via $f$.  We have the following result \cite[2.3]{GM}. Give $HN$ the zero differential.

\begin{thm}\mylabel{goody}  Suppose there is a map $g\colon N\rtarr HN$ of DG $R$-algebras such that
$Hg\colon HN\rtarr HN$ is the identity map and $g$ annihilates all $\cup_1$-products.  Then
\[ \Tor^A_*(N,R) = \Tor^{HA}_*(HN,R), \]
where $\Tor^{HA}_*(HN,R)$ is graded by total degree.
\end{thm}
\begin{proof}  Regard $HN$ as a DG $A$-module via $gf\colon A\rtarr HN$. The map 
\[ Tor^A_*(g,\id)\colon \Tor^{A}_*(N,R)\rtarr \Tor^{A}_*(HN,R) \]
is an isomorphism.  We may compute the target by use of the DG $HN$-algebra
\[ HN\otimes_A K(U) = HN\otimes E\{y_i\} \]
with differential 
\[d(n\otimes y_S) = (-1)^{deg\, n}n\otimes d(y_S) = \sum_{U,V}(-1)^{deg\, n}\si(U,V)\, ngf(a_U)\otimes y_V. \]
Since $f$ commutes with $\cup_1$ and $g$ annihilates $\cup_1$, the only non-zero terms occur with $\ell(U)=1$,
so that $d = \id\otimes d^1$ on $HN\otimes_A KA$.  Therefore
\[ HN\otimes_A K(A) = HN \otimes_A K(HA)\]
as DG $R$-modules, and the conclusion follows.
\end{proof}

As emphasized in \cite{GM,MN}, this is not merely a
statement about the EMSS.  Of course, it implies that
$E^2 = E^{\infty}$, but it also implies that there are
no non-trivial additive extensions from $E^{\infty}$ to $\Tor^A_*(N,R)$.
A spectral sequence argument would leave open the possibility of 
such extensions.

\subsection{The topological Eilenberg-Moore spectral sequence}\label{TopEMSS}
We briefly indicate how \myref{goody} applies to algebraic topology.  Here we assume that our commutative
ring $R$ is Noetherian and that all spaces in sight have integral homology of finite type.  

Consider a pullback square
\[  \xymatrix{
D\ar[d] \ar[r] \ar[d] & E \ar[d]^{p} \\
X\ar[r]_f & Y,\\} \]
where $p$ is a $q$-fibration with fiber $F$ and and $\pi_1(Y)$ acts trivially on $F$.  Eilenberg and Moore 
\cite{EM2} 
prove that  
\[  H^*(D;R) \iso \Tor^{*}_{C^*(Y;R)}(C^*(X;R),C^*(E;R)),\]
where $\Tor$ is regraded cohomologically; see also \cite[3.3]{GM}. Here $C^*$ is the (normalized) singular
cochain functor. It takes values in DG $R$-algebras with a $\cup_1$-product satisfying the Hirsch formula.
The associated EMSS is a spectral sequence of DG $R$-algebras which converges to the algebra $H^*(D;R)$
\cite[3.5]{GM}.  The hypothesis on $N$ in \myref{goody} is satisfied by $C^*(X;R)$ for certain products
of Eilenberg Mac\, Lane spaces $X$ and in particular for $X=BT^n$, the classifying space of the $n$-torus $T^n$
\cite[4.1, 4.2]{GM}.  This leads to the following corollary of \myref{goody} \cite[4.3]{GM},
of which
\myref{Yeah} is a special case.  The essential additional ingredient is that, if $T^n$ is a torus, then
there is a $q$-equivalence $C^*(BT^n;R)\rtarr H^*(BT^n;R)$ that annihilates $\cup_1$-products \cite[4.1]{GM}. 

We assume that $E$ is contractible, so that $D$ is homotopy equivalent to the fiber $Ff$ of $f$.

\begin{thm}\mylabel{goodytoo} Assume that $H^*(Y;R)$ is a polynomial algebra and that there is a map $e\colon BT^n\rtarr X$ such
that $H^*(BT^n;R)$ is a free $H^*(X;R)$ module via $e^*$.  Then for any map $f\colon X\rtarr Y$,
\[  H^*(Ff;R) \iso \Tor^*_{H^*(Y;R)}(H^*(X;R),R) \]
as a graded $R$-module, and $H^*(Ff;R)$ admits a filtration such that its associated graded algebra is 
isomorphic to  $\Tor^{*,*}_{H^*(Y;R)}(H^*(X;R),R)$ as a bigraded $R$-algebra.
\end{thm}

The proof proceeds by reduction to the case $X = BT^n$, where one shows that 
\[ \Tor^{*}_{C^*(Y;R)}(C^*(X;R),R) = \Tor^{*}_{H^*(Y;R)}(H^*(X;R),R) \]
using that they are computed by the same DG $R$-modules.  The hypothesis on $X$ is often satisfied when 
$X=BG$ for a compact Lie group $G$ with maximal torus $T^n$. It holds if $H_*(G;\bZ)$ has no $p$-torsion 
for any prime $p$ which divides the order of $R$ by \cite[4.5, 4.6]{GM}.  In particular, it holds for any 
$R$ if $G=U(n)$, $SU(n)$, $Sp(n)$ and, if $R$ has odd characteristic, $O(n)$ and $SO(n)$.  It also often
holds when $G$ is a suitable finite $H$-space \cite{MN}.  Therefore the theorem has many applications \cite{GM,MN}.  
 
\begin{rmk}  The explicit construction (\ref{Koszul}) of the differential in terms of $\cup_1$-products
on the distinguished resolution in \S\ref{cupone} allows it to be used to obtain explicit 
calculations even when \myref{goodytoo} does not apply.  In  \cite{Schochet}, Schochet used 
it to exhibit a two-stage Postnikov system with non-trivial differentials in its Eilenberg-Moore 
spectral sequence.
\end{rmk}

The relationship between $\Tor$ and matric Massey products in \myref{matrix1} leads to the following applications to 
special cases of our pullback diagram.  

\begin{cor}  If $i\colon F\rtarr E$ is the inclusion of the fiber of $p\colon E\rtarr Y$, then 
$\ker\, i^* = D(H^*(E;R);H^*(Y;R))$. 
The kernel of the suspension 
$$\si^*\colon \tilde{H}^*(Y;R)\rtarr H^{*-1}(\OM Y;R)$$ 
is $DH^*(Y;R)$. 
\end{cor}

There is a conceptually dual application to the calculation of $H_*(B(Y,G,X);R)$ for a topological
group $G$, a right $G$-space $Y$, and a left $G$-space X, where $B(Y,G,X)$ is the topological 
two-sided bar construction (e.g \cite[3.9]{GM}).   Here we have a dual to the last result.

\begin{cor} For a topological group $G$, the kernel of the suspension 
$$\si_*\colon \tilde{H}_*(G;R)\rtarr H_{*+1}(BG;R)$$
is $DH_*(G;R)$. 
\end{cor}


\begin{thebibliography}{10}

\bibitem{AR}
J. Ad{\'a}mek and J. Rosick{\'y}
Locally presentable and accessible categories.
London Mathematical Society Lecture Note Series 189(1994). 

\bibitem{BR}
T. Barthel and E. Riehl.
On the construction of functorial factorizations for model categories.
Algebraic \& Geometric Topology 13(2013), 1089--1124.

\bibitem{Bek}
T. Beke.
Sheafifiable homotopy model categories.
Math. Proc. Cambridge Philos. Soc. 129(2000), 447--475.

\bibitem{Board}
J.M. Boardman.
Conditionally convergent spectral sequences.
Cont. Math 239(1998), 49--84.

\bibitem{BorI}
F. Borceux.
Handbook of categorical algebra 1.
Encyclopedia of mathematics and its applications Vol 50. 
Cambridge University Press. 1994.

\bibitem{BorII}
F. Borceux.
Handbook of categorical algebra 2.
Encyclopedia of mathematics and its applications Vol 51. 
Cambridge University Press. 1994.

\bibitem{CE}
H. Cartan and S. Eilenberg.
Homological algebra.
Princeton Univ. Press. 1956.

\bibitem{CH}
J.D. Christensen and M. Hovey.
Quillen model structures for relative homological algebra.
Math. Proc. Camb. Phil. Soc. 133(2002), 261--293.
Postscript in ArXiv:math/0011216v5, 2013.

\bibitem{Danes}
L.W. Christensen and H. Holm.
The direct limit closure of perfect complexes.
ArXiv:1301.0731v2.

\bibitem{Cole}
M. Cole.
Mixing model structures.
Topology and its applications 153(2006), 1016--1032.

\bibitem{Cole2}
M. Cole.
Many homotopy categories are homotopy categories. 
Topol. Appl. 153 (2006) 1084--1099.  


\bibitem{Cole3}
M. Cole.
The homotopy category of chain complexes is a homotopy
category.  Preprint, 1999.

\bibitem{EM}
S. Eilenberg and J.C. Moore.
Foundations of relative homological algebra. Mem. Amer. Math. Soc. No. 55 1965.

\bibitem{EM2}
S. Eilenberg and J.C. Moore. 
Homology and fibrations, I.
Comm. Math. Helv. 40(1966), 199--236.

\bibitem{Garner} R. Garner. 
Understanding the small object argument. 
Appl. Categ. Structures. 17(2009), 247--285. 

\bibitem{GM}
V.K.A.M. Gugenheim and J.P. May.  
On the theory and applications of differential torsion products.  
Memoirs Amer. Math. Soc. No. 142, 1974.

\bibitem{Hardy}
G.H. Hardy.
A mathematician's apology.
Cambridge University Press. 1967.

\bibitem{Hirschhorn}
P. Hirschhorn.  
Model categories and their localizations.
Mathematical Surveys and Monographs, vol 99.
American Mathematical Society. 2003.

\bibitem{Hovey}
M. Hovey.
Model categories. 
Mathematical Surveys and Monographs, vol 63.
American Mathematical Society. 1999.

\bibitem{HoveySheaves}
M. Hovey.
Model category structures on chain complexes of sheaves.
Trans. Amer. Math. Soc. 353(2001), 2441--2457.

\bibitem{JoyalTierney}
A. Joyal and M. Tierney. 
Quasi-categories vs Segal spaces. 
Categories in algebra, geometry and mathematical physics, 
Volume 431 Contemp. Math., pages 277--326. Amer. Math. Soc., Providence, RI, 2007.
 
\bibitem{Kellyunified} 
G.M. Kelly. 
A unified treatment of transfinite constructions for free algebras, 
free monoids, colimits, associated sheaves, and so on. 
Bull. Austral. Math. Soc. 22(1980) 1--83. 

\bibitem{MacHom}
S. Mac\,Lane.
Homology. 
Academic Press. 1963.

\bibitem{MM}
M.A. Mandell and J.P. May.
Equivariant orthogonal spectra and $S$-modules.
Memoirs Amer. Math. Soc. 755. 2002.

\bibitem{MayOld}
J.P. May
The cohomology of augmented algebras and generalized Massey products for DGA-algebras.  
Trans. Amer. Math. Soc. 122(1966), 334--340.

\bibitem{M}
J.P. May.
The cohomology of principal bundles, homogeneous spaces, and 
two-stage Postnikov systems.  Bull. Amer. Math. Soc. 74(1968), 334--339.

\bibitem{MayMat}
J.P. May.
Matric Massey products. 
J. Algebra 12(1969), 533--568.

\bibitem{MN}
J.P. May and F. Neumann.
On the cohomology of generalized homogeneous spaces. 
Proc. Amer. Math. Soc. 130(2002), 267--270.

\bibitem{MP}
J.P. May and K. Ponto.
More Concise Algebraic Topology: Localization, Completion, and Model Categories.
Chicago Lectures in Mathematics.
The University of Chicago Press. Chicago. 2012.

\bibitem{MaySig} 
J.P. May and J. Sigurdsson.
Parametrized homotopy theory.
Amer. Math. Soc. 2006.

\bibitem{Milnor}
J. Milnor.
On spaces having the homotopy type of CW-complex.
Trans. Amer. Math. Soc. 90(1959), 272--280.

\bibitem{Moore}
J.C. Moore.
Alg\`{e}bre homologique et homologie des espaces classifiants.
S{\'e}minaire Henri Cartan, 12 no. 1 (1959--1960), Exp. No. 7, 37 p.

\bibitem{Quillen}
D.G. Quillen.
Homotopical algebra.
Lecture Notes in Math., Vol. 43 Springer, Berlin, 1967.

\bibitem{Radu}
A. Radulescu-Banu.
Cofibrations in Homotopy Theory.
arXiv:math/0610009 [math.AT].

\bibitem{Riehl}
E. Riehl.
Categorical homotopy theory.
New Mathematical Monographs.
Cambridge University Press. 2014.

\bibitem{Riehlalg}
E. Riehl. Algebraic model structures. 
New York J. Math. 17(2011) 173--231.

\bibitem{Schochet}
C. Schochet.
A two-stage Postnikov system where $E_2 \neq E_{\infty}$ in the Eilenberg-Moore 
spectral sequence.
Trans. Amer. Math. Soc. 157(1971), 113--118. 

\bibitem{SV}
R. Schw\"anzl and R.M. Vogt.
Strong cofibrations and fibrations in enriched categories.
Arch. Math. 79(2002), 449--462.

\bibitem{schwedeshipley}
S. Schwede and B. Shipley.
Algebras and modules in monoidal model categories.
Proc. London Math. Soc. 80(2000), 491--511.

\bibitem{Wall}
C.T.C. Wall.
Resolutions for extensions of groups. 
Math. Proc. Camb. Phil. Soc. 57(1961), 251--255.

\bibitem{Weibel}
C. Weibel.
An introduction to homological algebra.
Cambridge studies in advanced mathematics 38(1994).


\end{thebibliography}
\end{document}